\documentclass[11pt]{article}
\usepackage{graphicx} 
\usepackage[backend=biber, bibstyle=R/mybibstyle, citestyle=ieee-alphabetic, date=year, giveninits=true, dashed=true, sorting=nyt]{biblatex}
\usepackage{R/mypreamble}
\usepackage{R/vertexalgebras}
\usepackage{fullpage}

\def\d{{\delta}}
\def\e{{\epsilon}}
\def\s{{\sigma}}
\def\a{{\alpha}}
\def\t{{\tau}}

\def\o{{\omega}}
\def\ga{{\gamma}}
\def\p{{\pi}}
\def\UUU{{\mathfrak U}}

\newcommandx{\br}[4][1=\l, 2=\h]{[#3\,_{#1}\,#4]_{#2}}
\newcommandx{\pbr}[4][1=\l, 2=\h]{\{#3\,_{#1}\,#4\}_{#2}}
\newcommandx{\brl}[3][1=\l]{[#2\,_{#1}\,#3]}
\newcommandx{\pbrl}[3][1=\l]{\{#2\,_{#1}\,#3\}}

\title{\textbf{\uppercase{\large{$\h$-Vertex algebras and chiralization of star products}}}}
\author{Simone Castellan
}
\date{}

\begin{document}

\maketitle

\begin{abstract}
    \noindent We develop the theory of $\h$-vertex algebras, algebraic structures closely related to vertex algebras but with a deformed translation covariance axiom. We establish their structure theory, including analogues of Goddard’s Uniqueness Theorem, the Reconstruction Theorem, Borcherds Identity, and the OPE Expansion Formula, and introduce the associated notions of $\h$-Lie conformal and $\h$-Poisson vertex algebras. The formalism provides a natural and simplified construction of the Zhu algebra. The main application is to the chiralization of classical star-products: we show that every star-product on the symmetric algebra of a Lie algebra (or its central extensions) admits a chiralization, and we derive explicit formulae for these chiral star-products, including the Moyal–Weyl and Gutt star-products. Setting $\h=0$ recovers explicit deformation quantizations of a broad class of Poisson vertex algebras, including the classical limits of free-boson, $\beta\gamma$-system, affine, and Virasoro vertex algebras.
\end{abstract}

\section{Introduction}

Following the seminal paper by Zhu \cite{zhuModularInvarianceCharacters1996}, it is well-known that, to every vertex operator algebra (VOA) $V$, one can associate an associative algebra $\Zhu(V)$. This Zhu algebra has deep connections to the representation theory of $V$. Technically speaking though, the Zhu algebra is not \emph{directly} associated to the VOA. The construction requires an intermediate step, which is a deformation of the vertex operators $Y(a,z)\mapsto Y_\h(a,z)$. Following \cite{desoleFiniteVsAffine2006}, we call $Y_\h(a,z)$ the $\h$-deformed vertex operators. One defines two products on $V$, $\ast$ and $\circ$, as the $-1$ and $-2$ modes of $Y_\h(a,z)b$. Then, as a vector space, $\Zhu(V)=V/(V\circ V)$, with product induced by $\ast$. There are two different constructions for the operators $Y_\h(a,z)$. The first one is
\begin{equation}
    Y_\h(a,z):=(1+\h z)^{\Delta_a}Y(a,z),
\end{equation}
where $\Delta_a$ is the conformal degree  of $a$. It is used in Zhu's original paper. The second one is due to Huang \cite{huangDifferentialEquationsDuality2005}; it uses the change of variable $z\mapsto \frac{1}{\h}\log(1+\h z)$:
\begin{equation}
    Y_\h(a,z):=Y\left(a,\frac{1}{\h}\log(1+\h z) \right).
\end{equation}
Notice how the first construction needed a conformal structure (or at least a conformal grading \cite{desoleFiniteVsAffine2006}), while the second one works for every vertex algebra. The obtained $\h$-deformed vertex operators are not the same, but the corresponding Zhu algebras are isomorphic. 

Our goal is to switch the focus to these $\h$-deformed vertex operators and study the algebraic structure they form. This leads us to what we call an $\h$-vertex algebra. We define it as a vector space $V$ with a state-field correspondence $Y_\h(a,z)=\sum_{n\in\ZZ}a_{(n,\h)}z^{-n-1}$. It satisfies the usual vacuum and locality axioms of a vertex algebra, while the translation covariance axiom is deformed to
\begin{equation}
    [\del,Y_\h(a,z)]=(1+\h z)\del_z Y_\h(a,z).
\end{equation}
From this point of view, an $\h$-vertex algebra is a deformation of a vertex algebra with respect to the parameter $\h$. As we show in Proposition \ref{prop:changeofvariable}, $\h$-vertex algebras are in fact in bijection with vertex algebras, via the change of variables $z\mapsto \frac{1}{\h}\log(1+\h z)$. However, the axiomatic viewpoint allows us to obtain relations for the deformed $(n,\h)$-products that do not follow easily from the change of variables definition. Our main results are the following:
\begin{itemize}
    \item A complete structure theory of $\h$-vertex algebras. We show that all the usual results about the structure theory of vertex algebras admit a generalization to $\h$-vertex algebras. Among others, we prove results analogous to Goddard's Uniqueness Theorem, the Reconstruction Theorem, Borcherds Identity, and the OPE Expansion Formula (Section \ref{sec:structurehvertex}).
    \item The construction of an analogue of the $\l$-bracket, which we call the $\h$-bracket (Definition \ref{def h bracket}). We show that the $\h$-bracket controls both the commutator and the associator of the $(-1,\h)$-product (Theorem \ref{teo relazion h bracket h -1 prod}). Our theory of $\h$-vertex algebras, and the $\h$-bracket in particular, makes computations with the deformed $(n,\h)$-products much easier. This can be useful for computations of higher-level Zhu algebras, which are usually quite complicated to construct explicitly \cite{dongVertexOperatorAlgebras1998}.    
    Our main application is to the explicit computation of chiral star-products in Section \ref{section 4}.
    \item A more natural construction of the Zhu algebra, using the $\h$-vertex algebra formalism. The novelty of our approach is to relate the Zhu algebra directly to the $\h$-vertex algebra formed by the $\h$-deformed vertex operators. Using the $\h$-bracket formalism, we show that every $\h$-vertex algebra has a very natural associative algebra quotient. This algebra is easily seen to be isomorphic to the Zhu algebra of the original vertex algebra. Our proof avoids the usual complicated computations (Section \ref{sec:zhuhvertex}). 
\end{itemize}
Interestingly, the general philosophy to go from a formula for the $\l$-bracket to a formula for the $\h$-bracket is to move from infinitesimal calculus to calculus of finite differences (see Section \ref{sec:hbracket} and Appendix \ref{sec:finitedifferences} for more details). For example, the usual commutator formula for $(-1)$-product in a vertex algebra is
\begin{equation*}
    a_{(-1)}b-b_{(-1)}a=\int^0_{-\del}[a\,_\l\,b]\, d\l .
\end{equation*}
To get the commutator formula for the $(-1,\h)$-product, we have to replace the integral with its discrete analogue, called the definite sum (the inverse of the finite difference): 
\begin{equation*}
    a_{(-1,\h)} b-b_{(-1,\h)} a=\sum\nolimits^0_{-\del-\h}[a\,_\l\,b]_\h\, \d\l .
\end{equation*}
In \cite{liVertexFalgebrasTheir2011}, Li introduced the notion of an $F$-vertex algebra, associated to any formal group law $F$. Under his definition, ordinary vertex algebras and $\h$-vertex algebras are the $F$-vertex algebras associated to the additive and multiplicative formal group laws, respectively. See Section \ref{sec:F} for more details.

For a more uniform exposition, we prove all our results about $\h$-vertex algebras starting from the axioms, even when a different proof would be possible using the change of variables and the properties of vertex algebras. In \cite{desoleFiniteVsAffine2006}, the authors derived some identities of the $\h$-deformed vertex operators, in the conformal grading definition. For those, we provide new proofs, which only use the axioms of $\h$-vertex algebras and thus can be applied with greater generality. Since vertex algebras are $\h$-vertex algebras at the limit $\h\rightarrow0$, our results recover the structure theory of classical vertex algebras.

\subsection{Chiralization of star-products}\label{sec:introchiralstarproducts}

Our main application — and original motivation — for the $\h$-vertex algebra formalism is the computation of what we call \emph{chiral star-products}. 

The notion of star-products originates from physics, in the context of the so-called phase-space formulation of quantum mechanics. From a mathematical point of view, a star-product $\star$ is a deformation of the commutative product of a Poisson algebra $\mc A$, such that the star-commutator $a\star b-b\star a$ is a deformation of the Poisson bracket of $\Aa$. If we are given an associative algebra $A$, which is a quantization of $\Aa$, and a quantization map $\phi\colon \Aa\to A$, we can define a star-product $\star_\phi$ by pulling-back the non-commutative product of $A$ via $\phi$: $a\star_\phi b=\phi^{-1}(\phi(a)\phi(b))$. As an algebra, $(\mc A,\star_\phi)\cong A$. This isomorphism is tautological; however, in some interesting cases, we are able to express $\star_\phi$ with a closed formula that uses only operations of $\Aa$. Famous examples include the Moyal-Weyl and Gutt star-products.

Let $(V,\omega)$ be a symplectic vector space and consider the symmetric algebra $\Ss(V)$, with the Poisson bracket induced by $\omega$. The Moyal-Weyl star-product \cite{groenewoldPrinciplesElementaryQuantum1946} is 
\begin{equation}\label{eq:intromoyal}
    a\star b:=m\circ e^{\frac\omega2}(a\otimes b),  \quad \text{where } \ m(a\otimes b):=ab.
\end{equation}
Here $\omega$ is identified with the Poisson bivector and acts as a bidifferential operator. It is well-known that $(\mc S(V),\star)$ is isomorphic to the Weyl algebra of $(V,\omega)$. Similarly, let $\g$ be a Lie algebra and consider $\mc S(\g)$, with the Poisson bracket induced by the Lie bracket. Then $(\mc S(\g),\star)$ is isomorphic to the enveloping algebra $\UUU(\g)$, where $\star$ is the Gutt-star-product \cite{guttExplicitSpAst1983}. Its formula is similar to \eqref{eq:intromoyal}, with a deformation of the Poisson bivector instead of $\omega$.

From an algebraic perspective, star-products exist because Poisson algebras are the classical limit of associative algebras. In a sense, constructing a star-product is inverting the classical limit. From this point of view, vertex algebras are very similar to associative algebras. They too have a classical limit, called Poisson vertex algebras. These are commutative, differential algebras with an additional compatible $\l$-bracket $\pbrl{\cdot}{\cdot}$, called the Poisson $\l$-bracket. Moreover, there is another Zhu functor that associates a Poisson algebra to a Poisson vertex algebra. In many interesting cases, including freely generated vertex algebras, the following diagram commutes:

\begin{equation}\label{introdiagramma}
    \begin{tikzcd}
        &\text{Poisson vertex algebras}\arrow[d,"\Zhu"]  &\text{Vertex algebras}\arrow[l, "{\text{cl. limit}}"']\arrow[d,"\Zhu"']\\
        &\text{Poisson algebras}   &\text{Associative algebras}\arrow[l,"{\text{cl. limit}}"]
    \end{tikzcd}
\end{equation}
To chiralize means to ``invert'' the Zhu functors. For example, the chiralization of the Weyl algebra is the $\beta\gamma$-system, the chiralization of $\UUU(\g)$ is the affine vertex algebra $V^k(\g)$. 

It is natural to ask what the chiral analogue of a star-product is. The immediate answer would be a deformation of a Poisson vertex algebra structure into a vertex algebra. These were called ``star-deformations'' in \cite{liVertexAlgebrasVertex2004}. As for star-products, a star-deformation can be implicitly constructed from a known quantization of the Poisson vertex algebra via a quantization map. The interesting question is to provide a closed formula for the star-deformation, using only operations of the Poisson vertex algebra. Such a formula is particularly useful because vertex algebras have very complicated relations and are difficult to control, while Poisson vertex
algebras are much more tractable. From the physics point of view, a star-deformation corresponds to constructing a (very special) quantum field theory inside the formalism of classical field theory.

There is an important difference between star-products and star-deformations. While giving a star-product automatically gives a deformation of the Poisson bracket via the star-commutator, the same cannot be done for Poisson vertex algebras. In fact, the $\l$-bracket is a structure independent from the normally ordered product. So, one needs to specify a deformation of the commutative product of the Poisson algebras into a normally ordered product and a deformation of the Poisson $\l$-bracket into a compatible $\l$-bracket. In this paper, we use the integral $\l$-bracket formalism introduced by De Sole and Kac \cite{desoleFiniteVsAffine2006}. It is a single structure that contains the data of both the normally ordered product and the $\l$-bracket. This way, it is sufficient to compute a single deformation formula  $I_{\l,\star}$. More details in Section \ref{sec:integrallbracket}.

Suppose to have Poisson vertex, vertex, Poisson and associative algebras $\Vv,V,\Aa,A$ that fit in diagram \eqref{introdiagramma}. Let $\star$ be a star-product for $\Aa$, such that $(\Aa,\star)\cong A$. We expect a chiralization of $\star$ to be a deformation of $\Vv$ into $V$, but we would also expect to recover $\star$ after applying the Zhu functors. Unfortunately, a star-deformation will not satisfy this second condition. This is because the product in the Poisson Zhu algebra $\Aa$ is induced by the product on $\Vv$, but the product on the Zhu algebra $A$ is not induced by the normally ordered product on $V$, as discussed. It is instead induced by the $(-1,\h)$ product of the associated $\h$-vertex algebra. This motivates us to define a \emph{chiral star-product} $I_{\l,\h,\star}$ as a deformation of the Poisson vertex algebra $\Vv$ into the $\h$-vertex algebra associated to $V$ (Definition \ref{def:chiralstarproduct}). This is now compatible with the Zhu functor; if, after applying the Zhu functor, we obtain back $\star$, we say that $I_{\l,\h,\star}$ is a chiralization of $\star$.

We prove that any star-product on the symmetric algebra of a Lie algebra (or a central extension of a Lie algebra) can be chiralized (Theorem \ref{thm:liftingquantizationmap}). Both the Moyal-Weyl and Gutt star-products are included in this class. Moreover, we provide explicit formulae for all these chiral star-products. The most general case is in Theorem \ref{teo sum h star bracket general case}. The formula simplifies in the case of free-field vertex algebras (Corollary \ref{cor star sum h bracket caso free field}). Since a vertex algebra is an $\h$-vertex algebra at the limit $\h\to 0$, as an easy corollary we also obtain explicit formulae for star-deformations for a big class of Poisson vertex algebras (Theorem \ref{teo star integral l bracket}). This includes the classical limit of many well-studied vertex algebras, like affine, Virasoro, free boson, $\beta\gamma$-system vertex algebras, and more. In the case of free-field vertex algebras freely generated by a set $\{u_i\}_{i\in\Ii}$, we have:
\begin{equation}
    I_{\l,\star}(a,b)=m\circ\exp(\sum_{i\in\Ii}L_i^\l\otimes \del_{u_i})(a\otimes b),
\end{equation}
where $m$ is the multiplication map and
\begin{equation}
        L_i^\l(a)=\sum_{j\in\Ii}\int^{\l}_{-\del}\pdv{a}{u_j}\pbr[x+\del^{(1)}][]{u_j \,}{u_i}\, d x.
\end{equation}
We note similarities with classical star-products like Moyal-Weyl and Gutt. We still have the exponentiation of an operator which, in the end, depends only on the structure constants of the Poisson $\l$-bracket. This is interesting in relation to the known theory of ordinary star-products. Kontsevich proved that there is a bijection between equivalence classes of star-products and equivalence classes of formal Poisson structures \cite{kontsevichDeformationQuantizationPoisson2003}. That is, the formula for a star-product ultimately depends on the structure constants of the Poisson bracket. Our result gives some (very preliminary) indication that something similar may be true for Poisson vertex algebras as well.

Since \cite{liVertexAlgebrasVertex2004}, there has been little work on star-deformations. However, there is renewed interest in the topic, as evidenced by the appearance of new works \cite{butsonInverseHamiltonianReduction2025a} and \cite{wangQuantumDispersionlessKdV2024}, where deformations of Poisson vertex algebras are applied in a geometric and integrable system setting, respectively.

\paragraph{Organization of the paper.} In Section \ref{section 2} we recall some facts and definitions about the theory of vertex algebras. In Section \ref{section 3} we define $\h$-vertex algebras and proceed with the study of their structure theory. The exposition is inspired by \cite{desoleFiniteVsAffine2006} and by classical texts on vertex algebras, mainly \cite{frenkelVertexAlgebrasAlgebraic2004, kacVertexAlgebrasBeginners1998}. In Section \ref{section 4} we discuss the problem of chiralization of star-products, we derive explicit formulae for a class of Poisson vertex algebras, which, as a consequence, give us formulae for the star-deformations of Poisson vertex algebras. We recall results about formal distributions and the calculus of finite differences in the Appendices \ref{sec:formaldistributions} and \ref{sec:finitedifferences}. 

\paragraph{Acknowledgments.} A first version of this paper was written during my PhD. I am grateful for the constant encouragement, guidance, and support provided by my supervisors Daniele Valeri and Gwyn Bellamy, which made this work possible. I would also like to thank Alberto De Sole for the many helpful conversations and for useful suggestions to simplify the proofs of Section \ref{section 4}, and Benoit Vicedo, for explaining to me the change of variable construction. I am grateful to the organizers of the programs ``Thematic Program on Vertex and Chiral Algebras'' (IMPA, Rio de Janeiro, Spring 2022), ``Infinite dimensonal algebras Lie and related structures'' (La Sapienza, Roma, September 2022) and ``Quantum symmetries: Tensor categories, Topological quantum field theories, Vertex algebras'' (CRM, Montréal, Autumn 2022), during which I made significant progress on this project. The author acknowledges the financial support of EPSRC, through a PhD scholarship and a research assistant position (Grant Number EP/V053728/1).

\paragraph{Notation.}

If not differently stated, all vector spaces are assumed to be over $\CC$.

We will often consider non-associative products. In case of products between three or more elements, unless we specify differently, we always mean that the products should be computed from right to left.

\begin{itemize}
    \item $U[z]$ polynomials with coefficients in a vector space $U$.
    \item $U[z,z^{-1}]$ Laurent polynomials.
    \item $U[[z]]$ formal power series.
    \item $U((z)):=(U[[z]])[z^{-1}]$ formal Laurent series.
    \item $U[[z,z^{-1}]]$ bilateral series.
    \item $\NN:=\{0,1,2,\dotso\}$ the natural numbers.
\end{itemize}
\section{Preliminaries}\label{section 2}

\subsection{Structure theory of vertex algebras}\label{sec:structuretheoryvertex}

We briefly recall the main results about the structure theory of vertex algebras, to help the comparison with the structure theory of $\h$-vertex algebras developed in Section \ref{section 3}. We refer to \cite{frenkelVertexOperatorAlgebras1988,kacVertexAlgebrasBeginners1998}.

\begin{defi}\label{def va}
    A vertex algebra is the data of a vector space $V$, a nonzero vector $\z\in V$ (called the vacuum vector), and a linear map $Y:V\rightarrow \End(V)[[z,z^{-1}]]$ (called the state-field correspondence), denoted by $$a\mapsto Y(a,z)=\sum_{n\in\ZZ}{a_{(n)}z^{-n-1}}, $$ satisfying the following axioms:
    \begin{enumerate}[(i)]
        \item (Fields) $Y(a,z)$ is a field for all $a\in V$, i.e. for all $a,b\in V$, $a_{(n)}b=0$ for $n\gg0$. The fields $Y(a,z)$ are called vertex operators.

        \item (Vacuum) $Y(\z,z)=\id_V $, $Y(a,z)\z\in V[[z]] $ and $Y(a,z)\z|_{z=0}=a $.
        \item (Translation covariance) The endomorphism $\del\in\End(V) $ defined as $\del a:=a_{(-2)}\z $, for all $a$ in $V$, is called the (infinitesimal) translation operator. Then
        $$[\del,Y(a,z)]=\del_z Y(a,z).$$
        \item (Locality) For all $a,b\in V$, there exists $N>0$ (dependant on $a$ and $b$) such that
        $$ (z-w)^N[Y(a,z),Y(b,w)]=0. $$
    \end{enumerate}
\end{defi}

\begin{theorem}\label{thm:structurevertexalgebra}
    Let $V$ be a vertex algebra. The vertex operators $Y(a,z)$ satisfy the following identities:
    \begin{enumerate}
        \item $Y(a,z)\z=e^{z\del}a$, for all $a\in V$;
        \item $(\del a)_{(n)}=-n a_{(-n-1)}$, for all $a\in V$ and $n\in\ZZ$;
        \item (Skewsymmetry)
        $$Y(a,z)b=e^{z\del}Y(b,-z)a, \quad \forall a,b\in V; $$
        \item ($n$-product identity)
        $$Y(a_{(n)}b,z)=Y(a,z)_{(n)}Y(b,z), \quad \forall a,b\in V; $$
        \item (Borcherds identity)
        \begin{equation}\label{eq borcherds}
    \begin{aligned}
        i_{z,w}(z-w)^{n}Y(a,z)Y(b,w)-i_{w,z}(z-w)^nY(b,w)Y(a,z)\\
        =\sum_{j\geq0}\frac{1}{j!}Y(a_{(n+j)}b,w)\del_w^{j}\delta(z-w), \quad \forall a,b\in V, n\in\ZZ,
    \end{aligned}
    \end{equation}
    where $i_{z,w}$ and $i_{w,z}$ are the expansion in the domains $|w|<|z|$ and $|z|<|w|$  \eqref{def espansione w>z}.
    \end{enumerate}
\end{theorem}

\begin{oss}
    The $(n)$-product identity implies the following formula for the OPEs of vertex operators (see Theorem \ref{teo OPE}):
    \begin{equation*}
        Y(a,z)Y(b,w)=\sum_{n\geq0}\frac{Y(a_{(n)}b,w)}{i_{z,w}(z-w)^{n+1}}+{:}Y(a,z)Y(b,w){:}\, .
    \end{equation*}
    Thus, the singular part of the OPE is controlled by the non-negative $(n)$-products.
\end{oss}
The constraints the vertex operators need to satisfy force them to be unique in the following sense.

\begin{theorem}[Goddard's uniqueness theorem]
    Let $V$ be a vertex algebra and $B(z)$ an $\End(V)$-valued field, which is mutually local with all the fields $Y(a,z)$, $a\in V$. Suppose that, for some $b\in V$, $B(z)\z=e^{z\del}b$.    Then $B(z)=Y(b,z)$.
\end{theorem}
A consequence of Goddard's uniqueness theorem is the following result, which is useful for constructing examples of vertex algebras.

\begin{theorem}[Reconstruction Theorem]
    Let $V$ be a vector space, $\z$ a nonzero element in $V$, and $\del\in \End(V)$. Let $\{a^i(z)\}_{i\in I} $, ($I$ an index set) be a collection of $\End(V)$-valued fields such that:
    \begin{enumerate}
        \item $[\del, a^i(z)]=\del_z a^i(z)$;
        \item $\del \z=0$ and $a^i(z)\z \in V[[z]] $;
        \item the linear map defined by $a^i(z)\mapsto a^i:=a^i(z)\z|_{z=0}$ is injective;
        \item $a^i(z)$ and $a^j(z)$ are mutually local for all $i, j\in I$;
        \item the vectors $a^{i_1}_{(j_1)}\dots a^{i_n}_{(j_n)}\z $, with $j_s<0$ and $i_s\in I$ span $V$.
    \end{enumerate}
    Then the formula 
    \begin{equation}
        \begin{aligned}
            Y(a^{i_1}_{(j_1)}\dots &a^{i_n}_{(j_n)}\z, z)={:}\frac{\del^{-j_1-1}}{(-j_1-1)!}Y(a^{i_1},z)\dots \frac{\del^{-j_n-1}}{(-j_n-1)!} Y(a^{i_n},z)\!:\, ,
        \end{aligned}
    \end{equation}
    defines the (unique) structure of vertex algebra on $V$ such that $\z$ is the vacuum vector, $\del$ the infinitesimal translation operator, and $Y(a^i,z)=a^i(z)$, for all $i\in I$.
    
\end{theorem}

\subsubsection{The conformal structure}
In many important examples (in particular, all of those coming from conformal field theories), the vertex algebra comes with an associated conformal structure. That is, there is a vector $\o\in V$ such that $Y(\o,z)=L(z):=\sum_{n\in\ZZ}L_nz^{-n-2}$ (note the unusual indexing) satisfies the following:
\begin{enumerate}[(i)]
    \item The operators $\{L_n\}$ generate the Virasoro Lie algebra:
    $$[L_m,L_n]=(m-n)L_{m+n}+\delta_{m,-n}\frac{m^3-m}{12}c\id_V, $$
    for some $c\in\CC$, called the central charge;
    \item $L_{-1}=\del$;
    \item $L_0$ is a diagonalizable operator. 
\end{enumerate}

\begin{defi}
    A vertex algebra $V$ with a conformal vector $\o$ is called a {\em vertex operator algebra} (VOA for short).
\end{defi}
Let $V$ be a VOA. The operator $L_0$ induces a grading $V=\bigoplus V[\Delta]$, called the {\em conformal grading}, by eigenspace decomposition. For an eigenvector $a\in V$, let $\Delta_a$ be the corresponding eigenvalue.

\subsection{\texorpdfstring{$\lambda$}{lambda}-bracket formalism}\label{sec:lambdaformalism}

Let $V$ be a vertex algebra. Since $(\del a)_{(n)}=-n a_{(n-1)}$ for all $a\in V$, $n\in\ZZ$, all the negative products can be obtained from the $(-1)$-product via
\begin{equation}\label{eq:translationcovariancenproducts}
    a_{(n)}b=\frac{1}{n!}(\del^n a)_{(-1)}b, \quad \forall a,b\in V.
\end{equation}
The whole vertex algebra is then determined by the $(-1)$ product and the non-negative products (or by the singular part of the OPEs, from the physical point of view). The non-negative products can be packed together to create a Lie-like structure, called the $\l$-bracket. The following definition is due to Kac \cite{kacVertexAlgebrasBeginners1998}.

\begin{defi}\label{def:vertexliealgebra}
A vertex Lie algebra is a vector space $V$ together with an endomorphism $\del$ and a bilinear operation, called a $\l$-bracket
$$[\cdot_\l\cdot]:V\otimes V\rightarrow V\otimes\CC[\l], $$
that satisfies the axioms:
\begin{enumerate}[(i)]
    \item (skewsymmetry) $[a_{\l} b]=-[b_{-\l-\del}a]$;
    \item (sesquilinearity) $[\del a_\l b]=-\del[a_\l b], \qquad [a_\l \del b]=(\l+\del)[a_\l b]$;
    \item (Jacobi identity) $[[a_\l b]_{\l+\mu}c]=[a_\l[b_\mu c]]-[b_\mu [a_\l c]]$.
\end{enumerate}
Note, here we have used the following notation: if $[a_\l b]=\sum_{n\geq0}{\frac{\l^n}{n!}a_{(n)}b}$, then by $[a_{\l+\del} b]$ we mean $\sum_{n\geq0}{\frac{(\l+\del)^n}{n!}(a_{(n)}b)}$, with $\del$ acting on the coefficients $a_{(n)}b$.
\end{defi}

\begin{oss}
    Every vertex algebra is a vertex Lie algebra, with $\del$ given by the infinitesimal translation operator, and the $\l$-bracket defined as
    \begin{equation}
        \brl{a}{b}:=\Res_z(e^{\l z}Y(a,z)b)=\sum_{n\geq0}\frac{\l^n}{n!}a_{(n)}b.
    \end{equation}
\end{oss}
The following is an equivalent definition of a vertex algebra, due to Bakalov and Kac \cite{bakalovFieldAlgebras2003}.

\begin{defi}\label{def va l bracket}
    A vertex algebra is a quintuple $(V,\ket{0},\del,[\cdot_\l\cdot],:\,:)$ such that
\begin{enumerate}[(i)]
    \item $(V,\del,[\cdot_\l\cdot])$ is a vertex Lie algebra;
    \item $(V,\z,\del,:\,:)$ is a unital, non-commutative, non-associative differential algebra (i.e.
$\del$ is a derivation of the product $:\,:$), with unity $\z$, which satisfies the following two axioms:
    \begin{equation}
       {:}ab{:}-{:}ba{:}\,=\int_{-\del}^{0}{[a_{\l}b]d\l},
    \end{equation}
    \begin{equation}
        {:}({:}ab{:})c{:}-{:}a({:}bc{:}){:}={:}\left(\int_{0}^{\del}d\l\, b\right)[a_{\l}c]{:}+{:}\left(\int_{0}^{\del}d\l\, a\right)[b_{\l}c]{:}\, .
    \end{equation}
    \item The $\l$-bracket and the product $:\,:$ are related by the so-called right and left non-commutative Wick formulae
    \begin{equation}
        \brl{a}{{:}bc{:}}={:}\brl{a}{b}c{:}+{:}b\brl{a}{c}{:}+\int_{0}^{\l}{[[a_{\l}b]_\mu c]d\mu},
    \end{equation}
    \begin{equation}
        [{:}ab{:}_{\l}c]={:}a[b_{\l+\del^{(1)}}c]{:}+{:}b[a_{\l+\del^{(1)}}c] {:}+\int_0^\l{[b_\mu[a_{\l-\mu}c]]d\mu}.
    \end{equation}
\end{enumerate}
In Definition \ref{def va l bracket} we have used the following notation: 
$${:}\left(\int_{0}^{\del}{d\l}\,b\right)[a_{\l}c]{:}\ =\sum_{n\geq0}{{:}\left(\int_0^\del{\frac{\l^n}{n!}}d\l \, b\right)a_{(n)}c{:}}=\sum_{n\geq0}{:}\frac{\del^{j+1}b}{(j+1)!}(a_{(j)}c){:}\, . $$
 By $\del^{(1)}$ (respectively $\del^{(2)}$) we mean that the $\del$ acts only on the first term of the product (respectively the second), i.e. we have $$:\!b[a_{\l+\del^{(1)}} c]\!:\ =\sum_{n\geq0}:\!\left(\frac{(\l+\del)^n}{n!}b\right) a_{(n)}c\!:.$$
 The product $:\,:$ is called the normally ordered product and corresponds to the $(-1)$ product of Definition \ref{def va}.
\end{defi}
Definition \ref{def va l bracket} is usually more convenient for computations and allows us to think of vertex algebras as analogues of associative algebras. To any associative algebra we can associate a Lie algebra using the commutator, and the universal enveloping algebra provides an adjoint to this functor. Similarly, to any vertex algebra, we can associate a vertex Lie algebra (that now controls both the commutator and the associator of the normally ordered product), and there is an adjoint universal construction. 

\begin{theorem}\label{thm:universalvertexalgebra}
    Let $R$ be a vertex Lie algebra. There exists a ``universal vertex algebra'' $V(R)$ such that $R\hookrightarrow V(R)$ and, for any vertex algebra $V$ and vertex Lie algebra morphism $f:R\rightarrow V$, there is a vertex algebra morphism $V(R)\rightarrow V$ such that the following diagram commutes
    \begin{equation}
        \begin{tikzcd}
V(R) \arrow[rd]                          &   \\
R \arrow[u, "i", hook] \arrow[r, "f"] & V
\end{tikzcd}
    \end{equation}
\end{theorem}
\begin{proof}
    We define a bracket on $R$ as
    \begin{equation}
        [a,b]:=\int^0_{-\del}\brl{a}{b}d\l.
    \end{equation}
    By a straightforward check, using properties of the integral, this is a Lie bracket. Consider the universal enveloping algebra $\UUU(R_{Lie})$ associated to this Lie algebra structure. We define a vertex algebra structure on $\UUU(R_{Lie})$ by constructing a normally ordered product and a $\l$-bracket satisfying the axioms of Definition \ref{def va l bracket}. It is sufficient to define them on the PBW basis of $\UUU(R_{Lie})$. We proceed by mutual recursion on the degree of the monomials. For $a,b\in R$ the $\l$-bracket is the one on $R$. For all $a\in R$ and $B\in\UUU(R_{Lie})$, ${:}aB{:}=aB$. The formulae of Definition \ref{def va l bracket} extend these recursively, in a unique way, to the whole of $\UUU(R_{Lie})$. Then $V(R)$ is $\UUU(R_{Lie})$ with this vertex algebra structure. The PBW property of $\UUU(R_{Lie})$ implies that the PBW monomials with respect to the normally ordered product form a basis of $V(R)$. The map for the universal property is defined on the PBW monomials as
    $ {:}u_1\cdots u_n{:}\mapsto {:}f(u_1)\cdots f(u_n){:}\, . $
\end{proof}

\begin{oss}
    This construction of $V(R)$ differs from the one most commonly found in the literature (see for example \cite{frenkelVertexAlgebrasAlgebraic2004}). A proof of their equivalence can be found in \cite[Section 1.7]{desoleFiniteVsAffine2006}. 
\end{oss}
Theorem \ref{thm:universalvertexalgebra} holds in the greater generality of {\em non-linear vertex Lie algebras} (see \cite[Theorem 3.9]{desoleFreelyGeneratedVertex2005a}). A non-linear vertex Lie algebra is a vector space $R$ with a $\l$-bracket valued in $\Tt(R)[\l]$, that satisfies the same axioms of Definition \ref{def:vertexliealgebra}. Some technicalities are needed to give meaning to the axioms on $\Tt(R)$. We shall consider a special case of non-linear vertex Lie algebras, that we call ``sub-linear''.

\begin{defi}\label{def:sublinearLCA}
    A sub-linear vertex Lie algebra is a vector space $R$, together with an endomorphism $\del$ and a $\l$-bracket
    $$\brl{\cdot}{\cdot}:R\otimes R\rightarrow (R\oplus \CC)[\l], $$
    such that the axioms of Definition \ref{def:vertexliealgebra} hold on $R\oplus \CC$, with the $\l$-bracket extended by
    $$\brl{\CC}{-}=\brl{-}{\CC}=0. $$
\end{defi}
Alternatively, a sub-linear vertex Lie algebra is the quotient of a central extension of a vertex Lie algebra by putting the central element equal to a constant.

\begin{es}\label{es:sublinearvertexliealgebras}
    Many important examples of vertex algebras are of the form $V(R)$, where $R$ is a sub-linear vertex Lie algebra.
    \begin{enumerate}
        \item The $\beta\gamma$-system of rank $n$ is the vertex algebra $V(R)$, where $R$ is the free 
        $\CC[\del]$-module generated by $x_1,\dots,x_n,y_1,\dots,y_n$, with bracket induced by
        $$\brl{x_i}{y_j}=\delta_{i,j} \quad \forall i,j\in\{1,\dots,n\}. $$
        \item Let $U$ be a vector space with a symmetric bilinear form $(\cdot | \cdot)$. The free-boson vertex algebra is $V(R)$, where $R=\CC[\del]\otimes U$ with $\l$-bracket induced by
        $$\brl{u}{v}=\l (u|v), \quad \forall u,v\in U. $$
        The special case where $U$ is one-dimensional is also known as the Heisenberg vertex algebra.
        \item Let $\g$ be a Lie algebra with a symmetric, invariant bilinear form $(\cdot | \cdot)$. The affine vertex algebra at level $k\in\CC$, denoted as $V^k(\g)$, is $V(R)$, where $R=\CC[\del]\otimes \g$, with $\l$-bracket induced by
        $$\brl{u}{v}=[a,b]+k(u|v)\l, \quad \forall u,v\in U. $$
        \item The Virasoro vertex algebra $Vir_c$, with central charge $c\in\CC$, is $V(R)$, where $R$ is the free $\CC[\del]$-module generated by a single element $L$, with $\l$-bracket
        $$\brl{L}{L}=(2\l+\del)L+\frac{c}{12}\l^3. $$
    \end{enumerate}
\end{es}

\subsubsection{The integral \texorpdfstring{$\l$}{lambda}-bracket}\label{sec:integrallbracket}

There is an additional, equivalent definition of a vertex algebra, introduced by De Sole and Kac in \cite[Section 1.9]{desoleFiniteVsAffine2006}. It rewrites the axioms of Definition \ref{def va l bracket} in terms of a single operator $I_\l(\cdot,\cdot)$ called the integral $\l$-bracket. 

\begin{defi}\label{def va integral}
A vertex algebra is a quadruple $(V,\z,T,I_\l)$, where $V$ is a vector space, $0\neq\z\in V$, $\del\in\End(V)$ and $I_\l:V\otimes V\rightarrow V\otimes\CC[\l]$ is a bilinear map that satisfies:
\begin{enumerate}[(a)]
\item
(unity)
\begin{equation}
I_{\lambda}(|0\rangle,a)=I_{\lambda}(a,|0\rangle)=a\, \quad \forall a\in V;
\end{equation}
\item (sesquilinearity)
\begin{equation}
\frac{d}{d\lambda }I_{\lambda}(\del a,b)=-\lambda \frac{d}{d\lambda}I_{\lambda}(a,b)\,, \quad \del\left(I_{\lambda}(a,b)\right)=I_{\lambda}(\del a,b)+I_{\lambda}(a,\del b)\, \quad \forall a,b\in V;
\end{equation}
\item
(skewsymmetry)
\begin{equation}
I_{\lambda}(a,b)=I_{-\lambda-\del}(b,a)\,, \quad \forall a,b\in V;
\end{equation}
\item
(Jacobi identity)
\begin{equation}
I_{\lambda}(a,I_{\mu}(b,c))-I_{\mu}(b,I_{\lambda}(a,c))
=I_{\lambda+\mu}(I_\lambda(a,b)-I_{-\mu-\del}(a,b),c)\,, \quad \forall a,b,c\in V.
\end{equation}
\end{enumerate}
\end{defi}

\begin{oss}
    To go from Definition \ref{def va l bracket} to Definition \ref{def va integral}, define the integral $\l$-bracket $I_{\lambda}$ as 
\begin{equation}\label{def integral l bracket}
    I_\l(a,b):= \ :\!ab\!:+\int_0^\l [a_x b] dx.
\end{equation}

To go the other way, define the $\l$-bracket and the normally ordered product as
\begin{equation}
    :\! ab\!: \ =I_0(a,b), \hspace{4 mm} [a_\l b]=\frac{d}{d\l}I_\l(a,b).
\end{equation}
\end{oss}

\begin{oss}
    Intuitively, $I_\l(a,b)$ should be thought as $\int^\l_{-\infty}\brl[x]{a}{b}dx$, subject to an ``integral version'' of the vertex Lie algebra axioms. For example, for skew-symmetry,
$$I_\l(a,b)=\int^\l_{-\infty}\brl[x]{a}{b}dx=-\int^\l_{-\infty}\brl[-x-\del]{b}{a}dx=\int^{-\l-\del}_{-\infty}\brl[y]{b}{a}dy=I_{-\l-\del}(b,a), $$
using a change of variables. 
\end{oss}

The integral $\l$-bracket formalism gives a more elegant definition of a vertex algebra, which packs together properties of the normally ordered product and the $\l$-bracket in the axioms of a vertex Lie algebra written in ``integral form''. For example, the integral form of the Jacobi identity implies both the Jacobi identity of the $\l$-bracket, the non-commutative Wick formulae, and the quasi-associativity of the normally ordered product. This will be useful for computations of the chiralization of star-products, as explained in Section \ref{section 4}.

\subsection{Classical limit}\label{sec:poissonvertex}

A Poisson algebra is a commutative algebra with an additional, compatible Lie algebra structure. Similarly, a Poisson vertex algebra is a commutative vertex algebra with a compatible vertex Lie algebra structure. A commutative vertex algebra is a vertex algebra with $\l$-bracket equal to $0$. According to Definition \ref{def va l bracket}, this is just a commutative differential algebra. 

\begin{defi}\label{def:poissonvertexalgebra}
A Poisson vertex algebra is a unital, associative, commutative differential algebra $(\Vv,1,\cdot,\partial)$ with a bilinear operator, called the Poisson $\l$-bracket
$$\{\cdot_\l \cdot\}:\Vv\otimes\Vv \rightarrow \Vv\otimes\CC[\l], $$
satisfying the following properties:
\begin{enumerate}
    \item $(\Vv,\del,\{\cdot,\cdot\})$ is a vertex Lie algebra;
    \item (right Leibniz rule) $\{a_\l bc\}=\{a_\l b\}c+\{a_\l c\}b$;
    \item (left Leibniz rule) $\{ab_\l c\}=\{a_{\l+\del} c\underrightarrow{\}} b+\{b_{\l+\del}c\underrightarrow{\}} a $.
\end{enumerate}
By $\{a_{\l+\del} c\underrightarrow{\}} b$ we mean that $\del$ acts on the term indicated by the arrow, i.e. 
$$\{a_{\l+\del} c\underrightarrow{\}} b=\sum_{n\geq0}{\frac{1}{n!}(a_{(n)}c)(\l+\del)^n b}.$$
\end{defi}

\begin{es}
    Let $R$ be a (sub-linear) vertex Lie algebra. Then the symmetric algebra $\Ss(R)$ has a natural structure of Poisson vertex algebra, with Poisson $\l$-bracket
    $$\pbrl{a}{b}=\brl{a}{b},$$
    for all $a,b\in R$ and then extended by the left and right Leibniz rule. 
\end{es}

\begin{oss}
    If $\g$ is a finite-dimensional Lie algebra, then $\Ss(\g)$ is isomorphic to $\CC[\g^*]$ with the Kirillov–Kostant–Souriau Poisson bracket. Similarly, if $R=\CC[\del]\otimes\g$, then $\Ss(R)$ is isomorphic to $\CC[J_\infty \g^*]$, the algebra of functions over the arc space of $\g^*$. That is, $\Ss(R)$ is the algebra of differential polynomials on $\g$.
\end{oss}

\begin{defi}
    A {\em good filtration} on a vertex algebra $V$ is a filtration $V=\bigcup_{n\geq0} F_n V$ such that:
    \begin{enumerate}
        \item $\del F_nV\subset F_nV$;
        \item ${:}(F_nV)(F_mV){:}\subset F_{n+m}V $;
        \item $\brl{F_nV}{F_mV}\subset F_{n+m-1}V[\l] $, with $\l$ placed in degree $0$.
    \end{enumerate}
\end{defi}

If $V$ has a good filtration, then $\gr(V)$ has a natural structure of graded Poisson vertex algebra, with derivation
$\del (a+F_{n-1}V)=\del a+ F_{n-1}V$,
commutative product
$$ (a+F_{n-1} V)(b+ F_{m-1} V)={:}ab{:}+F_{n+m-1}V,   $$
and Poisson $\l$-bracket
$$\pbrl{a+F_{n-1}V}{b+F_{m-1}V}=\brl{a}{b}+F_{n+m-2}V[\l], $$
for $a\in F_nV$ and $b\in F_m V$. It is easy to see, from the identities of Definition \ref{def va l bracket} and from the fact that the $\l$-bracket is of negative degree, that all the axioms of Poisson vertex algebras are satisfied.

\begin{es}
    Let $R$ be a vertex Lie algebra and consider its universal vertex algebra $V(R)$. The PBW basis of $V(R)$ induces a good filtration by the degree of the PBW monomials. This follows easily from the relations in Definition \ref{def va l bracket}. The associated graded is isomorphic to $\Ss(R)$ as a Poisson vertex algebra. This can be extended to sub-linear vertex Lie algebras as well. To do so, we need to introduce a formal variable of degree $1$ and homogenize the $\l$-bracket, before computing the associated graded. 
\end{es}

\subsection{The Zhu algebra}\label{sec:Zhu}

Let $(V, \del, \z, Y)$ be a vertex algebra. Let $\h\in\CC^\times$ and consider the formal, invertible change of variables $x=\frac1\h\log(1+\h z)$. Define the following $\h$-deformed state-field correspondence:
\begin{equation}\label{eq:defchangeofvariable}
    Y_\h(a,z):=Y(a,x)=Y\left(a,\frac{1}{\h}\log(1+\h z)\right), \quad \forall a\in V.
\end{equation}

\begin{prop}
    For all $a\in V$, the $\h$-deformed vertex operator $Y_\h(a,z)$ is a well-defined $\End(V)$-valued quantum field.
\end{prop}
\begin{proof}
    We need to show that $Y_\h(a,z)b\in V((z))$ for all $a,b\in V$. Write
    $Y_\h(a,z)=\sum_{n\in\ZZ}a_{(n,\h)}z^{-n-1},$
    where, using the change of variables formula \eqref{eq:changeofvariabelRes}
    \begin{align*}
        a_{(n,\h)}&=\Res_z(z^n Y_\h(a,z))=\Res_z\left(\frac{1}{\h^n}(e^{\h x}-1)^n Y(a, x) \right)=\Res_x\left(\frac{1}{\h^n}(e^{\h x}-1)^ne^{\h x} Y(a, x) \right).
    \end{align*}
    We expand 
    $$\frac{1}{\h^n}(e^{\h x}-1)^ne^{\h x}=\sum_{k\geq n} c_n x^{n},$$
    so that $a_{(n,\h)}=\sum_{k\geq n}c_n a_{(n)}.$ In particular, $a_{(n,\h)}$ is a well-defined endomorphism of $V$, as $a_{(n,\h)}b\in V$ for all $n$, and $a_{(n,\h)}b=0$ for $n\gg 0$.
\end{proof}
Define the two following products, for $a,b\in V$:
\begin{align}
    &a\ast_\h b = a_{(-1,\h)}b=\Res_x\left(\frac{\h e^{\h x}}{e^{\h x}-1}Y(a,x)b \right); \\
    &a\circ_\h b = a_{(-2,\h)}b=\Res_x\left(\frac{\h^2 e^{\h x}}{(e^{\h x}-1)^2}Y(a,x)b \right).
\end{align}

\begin{theorem}
    The space $(V\circ_\h V)$ is a two-sided ideal with respect to the $\ast_\h$ product. Moreover, the $\ast_\h$ product induces the structure of a unital associative algebra on the quotient
    \begin{equation}\label{eq:defzhualgebra}
        \Zhu_\h(V)=V/(V\circ_\h V). 
    \end{equation}
\end{theorem}
\begin{proof}
    This is proved in \cite[Proposition 6.1]{huangDifferentialEquationsDuality2005} using lengthy computations with the residue. We provide a different, more natural proof, using the formalism of $\h$-vertex algebras introduced in Section \ref{section 3} (see Theorem \ref{thm:zhu}).
\end{proof}

\begin{defi}\label{def:zhualgebrahuang}
    The Zhu algebra associated to $V$ is $\Zhu(V)$ defined as in \eqref{eq:defzhualgebra}, with $\h$ set equal to $1$.
\end{defi}

\begin{oss}
    This construction of the Zhu algebra is due to Huang \cite{huangDifferentialEquationsDuality2005} and, despite being equivalent, it appears very far from the original construction by Zhu \cite{zhuModularInvarianceCharacters1996}. From our point of view, Definition \ref{def:zhualgebrahuang} is more convenient; we briefly explain the other construction and their connection below. We refer to \cite{desoleFiniteVsAffine2006} and \cite{vanekerenShortConstructionZhu2019} for a streamlined exposition of the two original constructions.
\end{oss}
Let $V$ be a VOA, with conformal grading $V=\bigoplus V[\Delta]$. Define the following $\h$-deformed state-field correspondence:
\begin{equation}\label{eq:hvertexoperatorsconformal}
    Y_\h[a,z]=(1+\h z)^{\Delta_a} Y(a,z)=\sum_{n\in\ZZ}a_{[n,\h]}z^{-n-1},
\end{equation}
for $a\in V[\Delta_a]$ and extended by linearity. Define two products on $V$:
\begin{equation}\label{def * product}
        a\bullet_\h b:=a_{[-1,\h]}b=\sum_{k\geq0}\binom{\Delta_a}{k}\h^k a_{(k-1)}b,
\end{equation}
\begin{equation}\label{def circ product}
     a\times_\h b:=a_{[-2,\h]}b=\sum_{k\geq0}\binom{\Delta_a}{k}\h^k a_{(k-2)}b,
\end{equation}
where the second equality holds only for homogeneous $a,b\in V$. Then again, one can prove that $V\times_\h V$ is a two-sided ideal for $\bullet_\h$ and that the quotient is a unital associative algebra. Then one defines the Zhu algebra as
\begin{equation}\label{eq:defzhu2}
    \Zhu'(V):=  \restr{ \left(\quotient {V}{V\times_\h V} \right)}{\h=1}. 
\end{equation}
The equivalence of the two constructions is due to the following fact.

\begin{theorem}\label{thm:isozhualgebras}
    For any $a\in V$, define 
    \begin{equation}\label{eq:changeofvariablezhu}
        Y'(a,z):=Y\left(e^{\h z L_0}a, \frac{e^{\h z}-1}{\h}\right).
    \end{equation}
    Then $Y'$ gives a VOA structure on $V$ and there exists an isomorphism of vertex algebras $T:(V, \z, Y)\rightarrow (V, \z, Y')$. In particular,
    \begin{equation}\label{eq:isochangeofvariable}
        Y'\left(a, z\right)=TY(T^{-1}a,z)T^{-1}.
    \end{equation}
\end{theorem}
\begin{proof}
    This was originally proved in \cite[Theorem 4.2.2]{zhuModularInvarianceCharacters1996}, and is a special case of Huang's change of variables formula (see \cite[Section 7.4]{huangTwodimensionalConformalGeometry1997}). 
\end{proof}
Consider now $Y'_\h(a,z)$. By \eqref{eq:defchangeofvariable} and \eqref{eq:changeofvariablezhu}, 
\begin{align*}
    Y'_\h(a,z)=&Y'\left(a, \frac{1}{\h}\log(1+\h z)\right)
    =Y\left((1+\h z)^{L_0} a, z\right)= (1+\h z)^{\Delta_a} Y(a,z)=Y_\h[a,z]. 
\end{align*}
Then formula \eqref{eq:isochangeofvariable} gives
\begin{equation*}
    Y_\h[a,z]=Y'\left(a, \frac{1}{\h}\log(1+\h z)\right)=TY_\h(T^{-1}a,z)T^{-1},
\end{equation*}
or, in other words,
\begin{equation*}
    T(a_{(n,\h)}b)=(Ta)_{[n,\h]}(Tb), \quad \forall a,b\in V, \forall n\in\ZZ.
\end{equation*}
Hence $T$ descends to an algebra isomorphism $T:\Zhu(V)\rightarrow \Zhu'(V)$.

\begin{oss}
    The isomorphism between the two constructions of the Zhu algebra is remarkable, as the second construction depends on the conformal structure, while the first one does not. In particular, this proves that the Zhu algebras associated to different conformal structures are isomorphic. The change of variables construction is more general, as it also works for vertex algebras that do not admit a conformal grading (of course, in those cases, the isomorphism $T$ cannot be defined).
\end{oss}

In \cite[Section 6]{desoleFiniteVsAffine2006}) a definition is given for the Zhu algebra of a Poisson vertex algebra, using the conformal grading definition. To the best of our knowledge, the construction via the change of variable has never been generalized to Poisson vertex algebras. We provide the construction in Section \ref{sec:zhuhvertex}.

\begin{oss}
    We have the following diagram:
    \begin{equation}\label{diagramma}
    \begin{tikzcd}
        &\text{Graded Poisson vertex algebras}\arrow[d,"\Zhu"]  &\text{Filtered vertex algebras}\arrow[l, "{\gr}"']\arrow[d,"\Zhu"']\\
        &\text{Graded Poisson algebras}   &\text{Filtered associative algebras}\arrow[l,"\gr"]
    \end{tikzcd}
\end{equation}
    In many interesting cases, it is commutative. In particular, if we start with $V=V(R)$, the enveloping algebra of a sub-linear vertex Lie algebra, then $\gr(\Zhu(V))=\Zhu(\gr(V))$, see \cite[Section 6]{desoleFiniteVsAffine2006}. It is however not true in general (for example, for simple quotients of affine vertex algebras). 
\end{oss}

\begin{defi}
    Let $A$ be an associative algebra such that $A=\Zhu(V)$, for some vertex algebra $V$. We say that $V$ is a {\em chiralization} of $A$. Similarly, if $\Aa$ is a Poisson algebra such that $\Aa=\Zhu(\Vv)$, for some Poisson vertex algebra $\Vv$, we say that $\Vv$ is a chiralization of $\Aa$.
\end{defi}
\section{\texorpdfstring{$\hbar$}{h}-Vertex algebras}\label{section 3}

We develop some results about the theory of $\h$-vertex algebras. The theory is similar to that of ordinary vertex algebras (see Section \ref{section 2} for reference). All the formulae reduce to the usual one for vertex algebras under the limit $\h\rightarrow0$. Some of the results in this section were already proved in \cite{desoleFiniteVsAffine2006} for the $\h$-deformed vertex operators $Y_\h[a,z]$, using their explicit definition in terms of the conformal grading. In those cases, we give different proofs that use only the axioms of an $\h$-vertex algebra.

\subsection{Structure theory of \texorpdfstring{$\hbar$}{h}-vertex algebras}\label{sec:structurehvertex}

Let $\h\in\CC^\times$. The definition of an $\h$-vertex algebra is the same as that of a vertex algebra (see Definition \ref{def va}), with an $\h$-deformed translation covariance axiom.

\begin{defi}\label{def h va}
    An $\h$-vertex algebra is the data of a vector space $V$, a nonzero vector $\z\in V$, and a linear map $Y_\h:V\rightarrow \End(V)[[z,z^{-1}]]$, denoted by
    $$a\mapsto Y_\h(a,z)=\sum_{n\in\ZZ}{a_{(n,\h)}z^{-n-1}}, $$ satisfying the following axioms:
    \begin{description}
        \item [(Fields)] for all $a,b\in V$, $a_{(n,\h)}b=0$ for $n\gg0$, i.e. $Y_\h(a,z)$ is a $\End(V)$-valued field for all $a\in V$.
        \item [(Vacuum)] $Y_\h(\z,z)=\id_V $, $Y_\h(a,z)\z\in V[[z]] $ and $Y_\h(a,z)\z|_{z=0}=a $, for all $a\in V$.
        \item [($\h$-translation covariance)] The endomorphism $\del\in\End(V) $ defined as $\del a:=a_{(-2,\h)}\z $, for all $a$ in $V$, is called the (infinitesimal) translation operator. Then
        $$[\del,Y_\h(a,z)]=(1+\h z)\del_z Y_\h(a,z).$$
        \item [(Locality)] For all $a,b\in V$, there exists $N>0$ (dependant on $a$ and $b$) such that
        $$ (z-w)^N[Y_\h(a,z),Y_\h(b,w)]=0. $$
    \end{description}
\end{defi}

Our first result is that $\h$-vertex algebras correspond to vertex algebras under the change of variables $x=\frac{1}{\h}\log(1+\h z)$.

\begin{prop}\label{prop:changeofvariable}
    Let $V$ be a vector space. If $(V,\z,\del,Y)$ is a vertex algebra structure on $V$, then $(V,\z,\del,Y_\h)$ is an $\h$-vertex algebra, with $\h$-deformed vertex operators $Y_\h(a,z)$ defined as 
    \begin{equation}\label{eq:hdeformedvertexoperator}
        Y_\h(a,z)=Y\left(a, \frac{1}{\h} \log(1+\h z) \right), \quad \forall a\in V.
    \end{equation}
    Conversely, if $(V,\z,\del,Y_\h)$ is an $\h$-vertex algebra, then $(V,\z,\del,Y)$ is a vertex algebra, where
    \begin{equation}
        Y(a,z):=Y_\h\left( a, \frac1\h(e^{\h z}-1) \right), \quad \text{ for all $a\in V$.}
    \end{equation}
\end{prop}
\begin{proof}
    Let $(V,\z,Y)$ be a vertex algebra and $Y_\h$ defined as in \eqref{eq:hdeformedvertexoperator}. Let $x=x(z)=\frac{1}{\h}\log(1+\h z)$ and $z=z(x)=\frac1\h(e^{\h z}-1)$. Then $Y_\h(a,z)=Y(a,x)$. Notice that, from Theorem \ref{thm:structurevertexalgebra} and \eqref{eq:changeofvariabelRes},
    \begin{align*}
        a_{(-2,\h)}\z &=\Res_x\left(\frac{\h^2 e^{\h x}}{(e^{\h x}-1)^2}Y(a,x)\z \right)= \Res_x\left(\frac{\h^2 e^{\h x}}{(e^{\h x}-1)^2}e^{x\del}a \right),
    \end{align*}
    which is equal to $\del a$, from the series expansion 
    $$\frac{\h^2 e^{\h x}}{(e^{\h x}-1)^2}=\frac1{x^2}-\frac{\h^2}{12}+O(x^2). $$
    The $\h$-translation covariance follows from:
    \begin{align*}
        [\del, Y_\h(a,z)]&=[\del, Y(a,x)]=\del_x Y(a,x)=\del_x Y_\h(a, z(x))= e^{\h x} \del_z Y_\h(a,z)=(1+\h z)\del_z Y_\h(a,z).
    \end{align*}
    Notice that, since $x(z)=z+O(z^2)$, we can write $(x(z)-x(w))^k=(z-w)^kh(z,w)$, for some $h(z,w)\in\CC[[z,w]]$ with $h(0,0)=1$. In particular, this means that $h(z,w)$ is invertible in $\CC[[z,w]]$. For all $a,b\in V$, there exists an $N\in\NN$ such that $$(x(z)-x(w))^N[Y(a,x(z)),Y(b,x(w))]=0,$$ so 
    $$(z-w)^Nh(z,w)^N[Y(a,x(z)),Y(b,x(w))]=0,  $$
    which implies locality for the $\h$-deformed vertex operators. Thus $(V,\z,\del,Y_\h)$ is an $\h$-vertex algebra. The other implication follows from an almost identical computation.
\end{proof}

We refer to the $\h$-vertex algebra constructed in Proposition \ref{prop:changeofvariable} as ``the $\h$-vertex algebra associated to a given vertex algebra''. We can look at $\h$-vertex algebras in two ways: as a deformation of vertex algebras, or as vertex algebras under a change of variables. Even if Proposition \ref{prop:changeofvariable} proves that $\h$-vertex algebras are in bijection with vertex algebras, their independent study still has useful applications, as the Zhu algebra is induced by the $\h$-deformed structure. 

\begin{oss}\label{oss:formalparameter2}
If we let $\h\rightarrow0$, then $\frac{1}{\h}\log(1+\h z)\rightarrow z$, so the change of variable is the identity. This is coherent with our interpretation of vertex algebras being $\h$-vertex algebras under the limit $\h\rightarrow0$.
\end{oss}

\begin{oss}\label{oss:hvaconformal}
    The $\h$-deformed vertex operators defined using the conformal grading $Y_\h[a,z]=(1+\h z)^{\Delta_a}Y(a,z)$ \eqref{eq:hvertexoperatorsconformal} induce another $\h$-vertex algebra structure, which is isomorphic to the one constructed in Proposition \ref{prop:changeofvariable}, as explained in Section \ref{sec:Zhu}. Notice that, in this case,
    $$a_{[-2,\h]}\z=\sum_{k\geq0}\binom{\Delta_a}{k}\h^k a_{(k-2)}\z=(\del+\h\Delta_a)a, $$
    for an homogeneous $a\in V$. Thus, the translation operator for this $\h$-vertex algebra is $\del+\h L_0$, which depends on the conformal structure. In \cite{desoleFiniteVsAffine2006}, the authors study the $\h$-vertex operators $Y_\h[a,z]$. They focus their attention on the interaction between $Y_\h[a,z]$ and $\del$. From the $\h$-vertex algebras point of view, this is a bit unnatural, as $\del$ is not the translation covariance operator of the $\h$-vertex algebra they are studying. If instead the correct operator $\del+\h L_0$ is considered, the formulae become simpler, as we show in the rest of this section.
\end{oss}

\iftrue

We recall the following existence and uniqueness lemma for a formal differential equation.

\begin{lemma}\label{lemma differential equation}
    Let $U$ be a vector space and let $R(z)\in\End(U)[[z]] $. Then the differential equation
    \begin{equation}\label{eq:differentialequation}
        \frac{d}{dz}f(z)=R(z)f(z),
    \end{equation}
    with the given initial data $f_0$ has a unique solution of the form 
    \begin{equation}
        f(z)=\sum_{n\geq0}{f_n z^n}, \ \ \ \ \ f_n\in U.
    \end{equation}
\end{lemma}

\paragraph{Notation.} Recall the expansion \begin{equation}\label{eq expansion (1+z)x}
    (1+\h z)^\alpha=\sum_{k\geq0}\frac{(\h \alpha)_{k,\h}}{k!}z^k.
\end{equation}
Let $U$ be a vector space and $R\in\End(U)$. We introduce the notation 
$$(1+\h z)^{R/\h}=\sum_{k\geq0}\frac{(R)_{k,\h}}{k!}z^k. $$

\begin{prop}\label{prop technical identities}
Let $V$ be an $\h$-vertex algebra and $a\in V$. Then 
\begin{enumerate}[(i)]
    \item $$    Y_\h(a,z)\z=(1+\h z)^{\del/\h}a; $$
    \item $$(1+\h z)^{-\del/\h}Y_\h(a,w)(1+\h z)^{\del/\h}=Y_\h\left(a,\frac{i_{w,z}(w-z)}{1+\h z}\right), $$
    where $i_{w,z}$ is the expansion in the domain $|z|<|w|$ \eqref{def espansione w>z}.
\end{enumerate}
\end{prop}
\begin{proof}
    Let $U$ be either $V$ or $\End(V)[[w,w^{-1}]] $; then $(i)$ and $(ii)$ are, respectively, identities in $U[[z]] $. Applying Lemma \ref{lemma differential equation}, we only need to show that both sides of equations $(i)$ and $(ii)$ satisfy the same formal differential equation, with the same initial condition. Due to the $\h$-translation covariance and vacuum axioms, the derivative of the left-hand side of $(i)$ is:
    $$\del_z Y_\h(a,z)\z=\frac{[\del,Y_\h(a,z)]\z}{(1+\h z)}=\frac{\del}{1+\h z}Y_\h(a,z)\z. $$
    Clearly, the right-hand side of (i) satisfies the same differential equation. Setting $z=0$ and using the vacuum axioms, we immediately get that both sides of $(i)$ satisfy the same initial condition too, hence they are equal.

    For $(ii)$, the left-hand side satisfies the formal differential equation
    $$\del_z f(z)= -\frac{\ad(\del)}{1+\h z}f(z),  $$
    with initial condition $f(0)=Y_\h(a,w)$. Using $\h$-translation covariance axiom,
    \begin{align*}
        \del_z Y_\h\left(a,g(z,w)\right)&=\frac{\del_z g(z,w)}{1+\h g(z,w)}\left[\del,Y_\h\left(a,g(z,w)\right)\right],
    \end{align*}
    for all $g(z,w)$ with $g(0,0)=0$. Thus, the right-hand side of $(ii)$ becomes
    $$ \del_z Y_\h\left(a,\frac{i_{w,z}(w-z)}{1+\h z}\right)=-\frac{\ad(\del)}{1+\h z}Y_\h\left(a,\frac{i_{w,z}(w-z)}{1+\h z}\right).$$
    Putting $z=0$ in the right-hand side, we get $Y_\h(a,w)$, so $(ii)$ is proved.
\end{proof}

\begin{prop}[Skewsymmetry]
    For every two elements $a$ and $b$ of an $\h$-vertex algebra $V$ the following relation holds:
    \begin{equation}\label{eq skewsymmetry}
        Y_\h(a,z)b=(1+\h z)^{\del/\h} Y_\h\left(b,-\frac{z}{1+\h z}\right )a.
    \end{equation}
\end{prop}
\begin{proof}
   This follows from the same proof of the skew-symmetry property of vertex algebras (\cite[Proposition 3.2.5]{frenkelVertexAlgebrasAlgebraic2004}), using $(1+\h z)^{-\del/\h}Y_\h(a,w)(1+\h z)^{\del/\h}=Y_\h\left(a,\frac{i_{w,z}(w-z)}{1+\h z}\right)$ instead of $e^{w \del}Y(a,z)e^{-w\del}$ $=Y(a,z+w)$.
\end{proof}

\begin{oss}
    Taking the coefficient of $z^{-n-1}$ ($n\in\ZZ$) of the right-hand side of equation \eqref{eq skewsymmetry}, we get the following formula for the skewsymmetry of the $(n,\h)$ products $(n\in\ZZ)$:
    \begin{equation}\label{eq skewsymmetry products}
        a_{(n,\h)}b=-\sum_{k\geq 0}(-1)^{k+n}\frac{(\del +\h(k+n+1))_{k,\h}}{k!}b_{(k+n,\h)}a.
    \end{equation}
\end{oss}
The following theorem is a generalization of Goddard's uniqueness theorem to $\h$-vertex algebras:

\begin{theorem}\label{teo Goddard's uniqueness}
    Let $V$ be an $\h$-vertex algebra and $B(z)$ an $\End(V)$-valued field, which is mutually local with all the fields $Y_\h(a,z)$, $a\in V$. Suppose that, for some $b\in V$:
    \begin{equation}\label{eq hypothesis Goddard}
        B(z)\z=(1+\h z)^{\del/\h}b.
    \end{equation}
    Then $B(z)=Y_\h(b,z)$.
\end{theorem}

\begin{proof}
    This follows from the same proof of Goddard's Theorem (\cite[Theorem 3.1.1]{frenkelVertexAlgebrasAlgebraic2004}), using $Y_\h(a,z)\z=(1+\h z)^{\del/\h}a$ instead of $Y(a,z)\z=e^{\del z}a$.
\end{proof}

\begin{prop}\label{prop n h prod identity}
    For all $a,b\in V$ and for all $n\in\ZZ$, the following $(n,\h)$-product identity holds:
    \begin{equation}\label{eq n h prod identity}
        (1+\h w)^{n+1}Y_\h(a_{(n,\h)}b,w)=Y_\h(a,w)_{(n)}Y_\h(b,w),
    \end{equation}
    where ${(n)}$ denotes the $n$-product of fields \eqref{def n products fields}. 
\end{prop}
\begin{proof}
    Both sides of equation \eqref{eq n h prod identity} are local to all $Y_\h(a,z)$, $a\in V$, due to Dong's Lemma \ref{lemma Dong}. We apply the left-hand side to $\z$, getting 
    $(1+\h w)^{\del/\h+n+1}a_{(n,\h)}b$.
    We only need to prove that the same is true for the right-hand side of \eqref{eq n h prod identity}, thanks to Goddard's uniqueness theorem.
    Note that 
    $$i_{w,z}(z-w)^nY_\h(b,w)Y_\h(a,z)\z $$
    only has positive powers of $z$, so its residue in $z$ is $0$. Thus
    \begin{align*}
        Y_\h(a,w)_{(n)}Y_\h(b,w)\z&= \Res_z(i_{z,w}(z-w)^nY_\h(a,z)(1+\h w)^{\del/\h}b).
    \end{align*}
    Applying part $(ii)$ of Proposition \ref{prop technical identities},
    \begin{align*}
        Y_\h(a,w)_{(n)}Y_\h(b,w)\z&= (1+\h w)^{\del/\h}\Res_z\left(i_{z,w}(z-w)^nY_\h\left(a, \frac{z-w}{1+\h w} \right)b\right) \\
        &=(1+\h w)^{\del/\h}\sum_{k\in\ZZ}a_{(k,\h)}b \Res_z\left((1+\h w)^{k+1}i_{z,w} (z-w)^{n-k-1} \right) \\
        &=(1+\h w)^{\del/\h+n+1}a_{(n,\h)}b,
    \end{align*}
    because $i_{z,w}(z-w)^n $ has a $z^{-1} $ term if and only if $n=-1$.
    \end{proof}

\begin{cor}\label{lemma for reconstruction theorem}
    For every collection of vectors $a^{1}, \dots, a^n\in V $ and positive integers $j_1, \dots, j_n$,
    \begin{equation}\label{eq reconstruction}
    \begin{aligned}
         \frac{1}{(j_1-1)!\cdots(j_n-1)!}{:}\del_z^{j_1-1}Y_\h(a^{1},z)\cdots \,&\del_z^{j_n-1} Y_\h(a^{n},z){:}\ = \\ &(1+\h z)^{\sum_i (-j_i+1)}Y_\h(a^1_{(-j_1,\h)}\cdots a^n_{(-j_n,\h)}\z, z),
    \end{aligned}       
    \end{equation}
    where ${:}a(z)b(z){:}\,=a(z)_{(-1)}b(z) $ is the normally ordered product of fields \eqref{def normally ordered product fields}. In particular, for any $a\in V$:
    \begin{equation}\label{eq translation covariance 2}
        Y_\h(\del a, z)=(1+\h z)\del_z Y_\h(a,z).
    \end{equation}
\end{cor}
\begin{proof}
    Let $a(z)$ and $b(z)$ be two fields. Then, by \eqref{eq del derivation n product fields}: 
    $${:}\left(\frac{\del_z^{n}}{n!}a(z)\right)b(z){:}\ =a(z)_{(-n-1)}b(z). $$
    Thus \eqref{eq reconstruction} follows from Proposition \ref{prop n h prod identity}. Equation \eqref{eq translation covariance 2} is a special case of \eqref{eq reconstruction}, when $n=1$ and $j_1=2$.
\end{proof}

\begin{cor}
    The translation operator $\del$ is a derivation of all $(n,\h)$-products.
\end{cor}
\begin{proof}
    This follows at once from $\h$-translation covariance and equation \eqref{eq translation covariance 2}.
\end{proof}
We state an $\h$-vertex algebra analogue of the reconstruction theorem for vertex algebras.

\begin{theorem}
    Let $V$ be a vector space, $\z$ a nonzero element in $V$, and $\del\in \End(V)$. Let $\{a^\alpha(z)\}_{\alpha\in A} $ ($A$ an index set) be a collection of $\End(V)$-valued fields such that:
    \begin{enumerate}
        \item $[\del, a^\alpha(z)]=(1+\h z)\del a^\alpha(z)$.
        \item $\del \z=0$, $a^\alpha(z)\z \in V[[z]] $.
        \item The linear map defined by $a^\alpha(z)\mapsto a^\alpha:=a^\alpha(z)\z|_{z=0}$ is injective.
        \item $a^\alpha(z)$ and $a^\beta(z)$ are mutually local for all $\alpha, \beta\in A$.
        \item The vectors $a^{\alpha_1}_{(j_1,\h)}\dots a^{\alpha_n}_{(j_n,\h)}\z $, with $j_s<0$ and $\alpha_s\in A$ span $V$.
    \end{enumerate}
    Then the formula 
    \begin{equation}
        \begin{aligned}
            Y_\h(a^{\alpha_1}_{(j_1,\h)}\dots &a^{\alpha_n}_{(j_n,\h)}\z, z)= \\ 
            &(1+\h z)^{-\sum_i (j_i+1)}:\!\frac{\del^{-j_1-1}}{(-j_1-1)!}Y_\h(a^{\alpha_1},z)\dots \frac{\del^{-j_n-1}}{(-j_n-1)!} Y_\h(a^{\alpha_n},z)\!:\, ,
        \end{aligned}
    \end{equation}
    defines the (unique) structure of an $\h$-vertex algebra on $V$ such that $\z$ is the vacuum vector, $\del$ the $\h$-infinitesimal translation operator, and $Y_\h(a^\alpha,z)=a^\alpha(z)$.
    
\end{theorem}
\begin{proof}
    The proof is identical to the usual one for vertex algebras (see for reference \cite[Theorem 4.4.1]{frenkelVertexAlgebrasAlgebraic2004}), using the $\h$-vertex algebra version of Goddard's theorem and Corollary \ref{lemma for reconstruction theorem}.
\end{proof}

\begin{theorem}\label{teo h borcherds}
   Let $V$ be an $\h$-vertex algebra; the following $\h$-Borcherds identity holds, for all $a,b\in V$ and all $n\in \ZZ$:
   \begin{equation}\label{eq h borcherds}
    \begin{aligned}
        i_{z,w}(z-w)^{n}Y_\h(a,z)Y_\h(b,w)-i_{w,z}(z-w)^nY_\h(b,w)Y_\h(a,z)\\
        =\sum_{j\geq0}(1+\h w)^{n+j+1}Y_\h(a_{(n+j,\h)}b,w)\del_w^{j}\delta(z-w)/j! \ . 
    \end{aligned}
\end{equation}
    In particular, for $n=0$, we obtain the commutator formula:
    \begin{equation}\label{eq h commutator formula}
        [Y_\h(a,z),Y_\h(b,w)]=\sum_{n\geq0}(1+\h w)^{n+1}Y_\h(a_{(n,\h)}b,w)\del^{n}_w\delta(z-w)/n!.
    \end{equation}
\end{theorem}
\begin{proof}
   The left-hand side of equation \eqref{eq h borcherds} is a local distribution. In fact, if we multiply it by $(z-w)^N$ such that $N+n\geq0$, it becomes equal to 
    $$ (z-w)^{N+n}[Y_\h(a,z),Y_\h(b,w)],$$
    which is local by the locality axiom. Then \eqref{eq h borcherds} follows from Decomposition Theorem \ref{theorem decomposition} and Proposition \ref{prop n h prod identity}.
\end{proof}

\begin{cor}[OPE expansion]
    The Operator Product Expansion of the operators $Y_\h(a,z)$, $Y_\h(b,w)$ has the following form:
    \begin{equation}\label{eq OPE h vertex operators}
        Y_\h(a,z)Y_\h(b,w)=\sum_{n\geq0}\frac{(1+\h w)^{n+1}Y_\h(a_{(n,\h)}b,w)}{i_{z,w}(z-w)^{n+1}}+{:}Y_\h(a,z)Y_\h(b,w){:}\, .
    \end{equation}
\end{cor}
\begin{proof}
    Due to Proposition \ref{prop n h prod identity}, we can rewrite \eqref{eq OPE h vertex operators} as
    \begin{equation*}
        Y_\h(a,z)Y_\h(b,w)=\sum_{n\geq0}\frac{Y_\h(a,w)_{(n)}Y_\h(b,w)}{i_{z,w}(z-w)^{n+1}}+{:}Y_\h(a,z)Y_\h(b,w){:}\, ,
    \end{equation*}
    which holds due to locality and Theorem \ref{teo OPE}.
\end{proof}

\begin{oss}\label{oss discrete analog}
    At the level of $(n,\h)$-products, equation \eqref{eq translation covariance 2} becomes
    \begin{equation}
        (\del a)_{(n,\h)}b=-n a_{(n-1,\h)}b-\h(n+1)a_{(n,\h)}b \quad a,b\in V, n\in\ZZ,
    \end{equation}
    that is
    \begin{equation}\label{eq h covariance DSK prods}
        ((\del+\h(n+1)) a)_{(n,\h)}b=-n a_{(n-1,\h)}b \quad a,b\in V, n\in\ZZ.
    \end{equation}
    If $n>0$, we can iterate equation \eqref{eq h covariance DSK prods} to get
    \begin{equation}\label{eq covariance product iterated}
        a_{(-n-1,\h)}b=\frac{(\del)_{n,\h}}{n!}a_{(-1,\h)}b.
    \end{equation}
    Notice that this is a discrete analogue of the formula \eqref{eq:translationcovariancenproducts} for vertex algebras. We substituted $\del^n$ with the falling factorial $(\del)_{n,\h}$, which is the discrete analogue of the monomial (see the discussion in Section \ref{sec:finitedifferences}). This analogy becomes evident with the $\h$-bracket formalism.
\end{oss}

\subsection{\texorpdfstring{$\hbar$}{h}-bracket formalism}\label{sec:hbracket}

In this section, we introduce the $\h$-analogue of the $\l$-bracket. We also define the related notions of an $\h$-vertex Lie algebra and an $\h$-Poisson vertex algebra and study their relation to $\h$-vertex algebras.

\begin{defi}\label{def:hvertexliealgebra}
    An $\h$-vertex Lie algebra is a vector space $R$, together with an endomorphism $\del$, and a bilinear operation, called the $\h$-bracket
$$[\cdot_\l\cdot]_\h:V\otimes V\rightarrow V\otimes\CC[\l], $$
that satisfies the axioms:
\begin{enumerate}
    \item (skewsymmetry) $[a_{\l} b]_\h=-[b_{-\l-\del-2\h}\,a]_\h$;
    \item (sesquilinearity) $[\del a_\l b]_\h=-(\l+\h)[a_\l b], \ \ [a_\l \del b]_\h=(\l+\del+\h)[a_\l b]_\h$;
    \item (Jacobi identity) $[[a_\l b]_{\h \ \l+\mu+\h \ }c]_\h=[a_\l[b_\mu c]_\h\,]_\h-[b_\mu [a_\l c]_\h\,]_\h$.
\end{enumerate}
\end{defi}

Let $V$ be an $\h$-vertex algebra. Define a bracket $[\cdot_\l\cdot]_\h:V\otimes V\rightarrow V\otimes\CC[\l] $ as
\begin{equation}\label{def h bracket}
    \br{a}{b}:=\Res_z((1+\h z)^{\l/\h}Y_\h(a,z)b)=\sum_{n\geq0}\frac{(\l)_{n,\h}}{n!}a_{(n,\h)}b.
\end{equation} 

\begin{oss}
    The bracket defined in \eqref{def h bracket} can be thought of as a discrete generalization of the usual $\l$-bracket. In fact, we changed $\l^n$ for the falling factorial $(\l)_{n,\h}$ (see Remark \ref{oss discrete analog}).
\end{oss}

\begin{prop}\label{prop:l+hbracket}
    Let $R$ be a vector space, with an endomorphism $\del$. Then $(R,\del, \brl{\cdot}{\cdot})$ is a vertex Lie algebra if and only if $(R,\del, \brl[\l+\h]{\cdot}{\cdot})$ is an $\h$-vertex Lie algebra.
\end{prop}
\begin{proof}
    A direct check shows that $\brl[\l+\h]{\cdot}{\cdot}$ satisfies the axioms of Definition \ref{def:hvertexliealgebra} if and only if $\brl{\cdot}{\cdot}$ is a $\l$-bracket.
\end{proof}

\begin{theorem}\label{teo h va h lca structure}
    Let $V$ be an $\h$-vertex algebra, then the bracket defined as in \eqref{def h bracket}, together with $\del$, induces the structure of an $\h$-vertex Lie algebra on $V$.
\end{theorem}
\begin{proof}
    This follows from Propositions \ref{prop:changeofvariable} and \ref{prop:l+hbracket}. If we write $Y_\h(a,z)=Y\left(a,\frac{1}{\h}\log(1+\h z) \right) $, formula \eqref{def h bracket} becomes
    \begin{equation}
        \br{a}{b}:=\Res_z(1+\h z)^{\l/\h}Y\left(a,x(z) \right)b=\Res_x e^{x (\l+\h)}Y(a,x)=\brl[\l+\h]{a}{b},
    \end{equation}
    where $\brl{a}{b}$ denotes the $\l$-bracket associated to the vertex operators $Y(a,z)$. In particular, the $\h$-bracket defined in \eqref{def h bracket} gives an $\h$-vertex Lie algebra structure. 

    We can also prove this directly from the axioms of an $\h$-vertex algebra. For completeness, we write here the proof, which is similar to the one for the $\l$-bracket of a vertex algebra. To prove the skewsymmetry relation, apply $\Res (1+\h z)^{\l/\h}$ to both sides of equation \eqref{eq skewsymmetry}:
    \begin{align*}
        \br{a}{b}&=\Res_z(1+\h z)^{(\l+\del)/\h} Y_\h\left(b,-\frac{z}{1+\h z}\right )a =\sum_{k\geq 0}(-1)^{k+1}\frac{(\l+\del +\h(k+1))_{k,\h}}{k!}b_{(k,\h)a} \\
        &=-\sum_{k\geq 0}\frac{(-\l-\del -2\h)_{k,\h}}{k!}b_{(k,\h)a} =-\br[-\l-\del-2\h]{a}{b},
    \end{align*}
    where we used formula \eqref{eq (-1)^n Pochhammer symbol}.
    
    The first part of sesquilinearity follows from applying $\Res (1+\h z)^{\l/\h}$ to both sides of \eqref{eq translation covariance 2}, and by integrating by parts when computing the residue on the right-hand side. The second part of sesquilinearity follows from the fact that $\del$ is a derivation of all $(n,\h)$-products.

    Finally, let us prove the Jacobi identity. Take equation \eqref{eq h commutator formula}, apply both sides to $c\in V$, then apply $\Res_z (1+\h z)^{\l/\h}$ to get:
    \begin{equation}\label{eq proof h va h lca structure 1}
        \begin{aligned}
            \Res_z(1+\h z)^{\l/\h}[Y_\h(a,z),Y_\h(b,w)]c=&\\
            =\Res_z(1+\h z)^{\l/\h}&\sum_{n\geq0}(1+\h w)^{n+1}Y_\h(a_{(n,\h)}b,w)\del^{n}_w\delta(z-w)c/n!.
        \end{aligned}
    \end{equation}
    To compute the right-hand side of \eqref{eq proof h va h lca structure 1}, notice that, because of equation \eqref{eq:derivativedelta},
    \begin{align*}
        &\Res_z(1+\h z)^{\l/\h}\del_w^{n}\delta(z-w)/n!=\Res_z(1+\h z)^{\l/\h}\sum_{k\in\ZZ}\binom{k}{n}w^{k-n}z^{-k-1}\\ 
        &=\sum_{k\geq 0}\frac{(\l)_{k+n,\h}}{k!n!}w^{k} =\frac{(\l)_{n,\h}}{n!}\sum_{k\geq 0}\frac{(\l-n\h)_{k,\h}}{k!}w^{k} =\frac{(\l)_{n,\h}}{n!}(1+\h w)^{\l/\h-n}.
    \end{align*}
    Using this, we rewrite \eqref{eq proof h va h lca structure 1} as
    \begin{equation}\label{eq proof Jacobi 1 step}
        \br{a}{Y_\h(b,w)c}-Y_\h(b,w)\br{a}{c}=(1+\h w)^{\l/\h+1}Y_\h(\br{a}{b},w)c.
    \end{equation}
    The Jacobi identity follows by further applying $\Res_w(1+\h w)^{\mu/\h}$ to both sides of \eqref{eq proof Jacobi 1 step}.
    
\end{proof}

\begin{oss}
    Let $(V,Y)$ be a vertex algebra and $(V_\h,Y_\h)$ its associated $\h$-vertex algebra. Notice that the expression of the $\h$-bracket of $V_\h$ in terms of the $\l$-bracket of $V$ is relatively easy, as it involves a simple translation $\l\mapsto \l+\h$. However, the indexing of the products in the $\h$-bracket is done in terms of the falling factorials, not of powers. This means that we cannot immediately relate the non-negative $(n)$-products of $V$ with the non-negative $(n,\h)$-products of $V_\h$, as we first need to rearrange all the terms in $\l$. This combinatorial problem involves the Stirling numbers, whose generating series is closely related to the change of variable $z=\frac1\h\log(1+\h x)$ . 
\end{oss}

\begin{oss}
    Let $V$ be a VOA and consider the associated $\h$-vertex algebra structure induced by the conformal grading: $Y_\h(a,z)=(1+\h z)^{\Delta_a}Y(a,z)$. In this case, formula \eqref{def h bracket} becomes (for homogeneous $a\in V$):
    \begin{equation}\label{eq h bracket deformed vertex operator case}
        \br{a}{b}:=\Res_z(1+\h z)^{\l/\h+\Delta_a}Y(a,z)b=\sum_{n\geq0}\frac{(\l+\h\Delta_a)_{n,\h}}{n!}a_{(n)}b.
    \end{equation}
\end{oss}

To ease the notation, let us denote the $(-1,\h)$-product by $\ast_\h$, that is
\begin{equation*}
    a\ast_\h b:=a_{(-1,\h)}b, \quad \forall a,b\in V.
\end{equation*}

\begin{theorem}\label{teo relazion h bracket h -1 prod}
    Let $V$ be an $\h$-vertex algebra. The $\ast_\h$ product satisfies the following quasi-commutativity and quasi-associativity formulae, expressed in terms of the $\h$-bracket for all $a,b,c\in V$:
    \begin{equation}\label{eq commutator -1 h prod}
        a\ast_\h b-b\ast_\h a=\sum\nolimits_{-\del-\h}^0\br{a}{b}\delta\l,
    \end{equation}
    \begin{equation}\label{eq associator -1 h prod}
        (a\ast_\h b)\ast_\h c-a\ast_\h(b\ast_\h c)=\left( \sum\nolimits_0^{\del} \delta \l \, b \right)\ast_\h \br{a}{c}+ \left( \sum\nolimits_0^{\del} \delta \l \, a \right)\ast_\h \br{b}{c}.    
    \end{equation}
    They also satisfy the following $\h$-analogue of the right and left non-commutative Wick formulae:
    \begin{equation}\label{eq right wick}
        \br{a}{b\ast_\h c}=b\ast_\h \br{a}{c}+\br{a}{b}\ast_\h c+\sum\nolimits_0^{\l+\h}\br[\h \ \mu]{[a_\l b]}{c}\delta\mu,
    \end{equation}
    \begin{equation}\label{eq left wick}
        \br{a\ast_\h b}{c}=a\ast_\h \br[\l+\del^{(1)}]{b}{c}+b\ast_\h \br[\l+\del^{(1)}]{a}{c}+\sum\nolimits_0^\l \br[\mu]{b}{\br[\l-\mu-\h]{a}{c}} \delta \mu. 
    \end{equation}
\end{theorem}
    
\begin{proof}
    Equation \eqref{eq skewsymmetry products} for $n=-1$ and the role of $a$ and $b$ swapped becomes:
    \begin{align*}
        b\ast_\h a&=-\sum_{k\geq0}(-1)^{k+1}\frac{(\del+\h k)_{k,\h}}{k!}a_{(k-1,\h)}b=a\ast_\h b + \sum_{k\geq 0}\frac{(-\del-\h)_{k+1,\h}}{(k+1)!}a_{(k,\h)}b,
    \end{align*}
    thus proving \eqref{eq commutator -1 h prod}. Apply the $(-1,\h)$-product identity \eqref{eq n h prod identity} to $c\in V$ and take the constant term:
    \begin{align*}
        (a\ast_\h b)\ast_\h c&=\Res_z z^{-1} Y_\h(a,z)_{(-1)}Y_\h(b,z)c. 
    \end{align*}
    Now, using formula \eqref{def normally ordered product fields} and equation \eqref{eq covariance product iterated}:
    \begin{align*}
        \Res_z z^{-1} Y_\h(a,z)_{(-1)}Y_\h(b,z)c&=\Res_z z^{-1}(Y_\h(a,z)_+Y_\h(b,z)c+Y_\h(b,z)Y_\h(a,z)_-c) \\
        &=\sum_{k\geq-1}a_{(-k-2,\h)}(b_{(k,\h)}c)+\sum_{k\geq 0}b_{(-k-2,\h)}(a_{(k,\h)}c)\\
        &=a\ast_\h(b\ast_\h c)+\sum_{k\geq 0}\left(\frac{(\del)_{k+1,\h}}{(k+1)!}a \right)\ast_\h(b_{(k,\h)c})\\
        & \quad+\sum_{k\geq 0}\left(\frac{(\del)_{k+1,\h}}{(k+1)!}b \right)\ast_\h(a_{(k,\h)c}).
    \end{align*}
    This proves \eqref{eq associator -1 h prod}. To prove \eqref{eq right wick}, equate the constant term in $w$ in both sides of \eqref{eq proof Jacobi 1 step}, getting:
    \begin{align*}
        \br {a}{b\ast_\h c}&=b\ast_\h \br {a}{c}+\sum_{k\geq -1}\frac{(\l+\h)_{k+1,\h}}{(k+1)!}(\br {a}{b})_{(k,\h)}c\\
        &=b\ast_\h \br{a}{c}+\br{a}{b}\ast_\h c+\sum\nolimits_0^{\l+\h}\br[\h \ \mu]{[a_\l b]}{c}\delta\mu.
    \end{align*}
    Finally, the left non-commutative Wick formula \eqref{eq left wick} can be derived from \eqref{eq right wick}. Using the skewsymmetry of the $\h$-bracket, write 
    $$\br {a\ast_\h b}{c}=-\br [-\l-\del-2\h \,]{c}{a\ast_\h b}. $$
    Then using \eqref{eq right wick}, this is equal to
    \begin{align*}
        -a\ast_\h \br[-\l-\del-2\h \,]{c}{b}-\br[-\l-\del-2\h \,]{c}{a}\ast_\h b-\sum\nolimits_0^{-\l-\del-\h}\br[\h \ \mu]{[c_{-\l-\del-2\h \,} a]}{b}\delta\mu,
    \end{align*}
    where each of the $\del$ is applied to the result of the $(n,\h)$-products between $a,b,c$. Since $\del$ is a derivation of the $(n,\h)$-products, we can write $\del=\del^{(1)}+\del^{(2)}$, acting respectively on the left and right factor, i.e. $\del(a\ast_\h b)=(\del^{(1)} a)\ast_\h b+a\ast_\h (\del^{(2)} b)$. Applying skewsymmetry again, we get:
    \begin{align*}
        a\ast_\h \br[\l+\del^{(1)} \,]{b}{c}+\br[\l+\del^{(2)} \,]{a}{c}\ast_\h b+\sum\nolimits_0^{-\l-\del-\h}\br[\h \ \mu]{[a_{\l+\del^{(2)} \,} c]}{b}\delta\mu.
    \end{align*}
    Now, by the commutator formula \eqref{eq commutator -1 h prod} on $\br[\l+\del^{(2)} \,]{a}{c}\ast_\h b$ and by adding and reversing the extremes of summation \eqref{eq addition extremes definite sum}, \eqref{eq inversion extremes definite sum}:
    $$ a\ast_\h \br[\l+\del^{(1)} \,]{b}{c}+b\ast_\h\br[\l+\del^{(1)} \,]{a}{c}-\sum\nolimits_{-\l-\del-\h}^{-\del-\h}\br[\h \ \mu]{[a_{\l+\del^{(2)} \,} c]}{b}\delta\mu. $$
    Finally, using skewsymmetry, Proposition \ref{prop change of variable} for the change of variables, and the first sesquilinearity relation:
    \begin{align*}
        -\sum\nolimits_{-\l-\del-\h}^{-\del-\h}\br[\h \ \mu]{[a_{\l+\del^{(2)} \,} c]}{b}\delta\mu&=\sum\nolimits_{-\l-\del-\h}^{-\del-\h}\br[-\mu-\del-2\h]{b}{[a_{\l+\del^{(1)} \,} c]_\h}\delta\mu \\
        &=\sum\nolimits_{0}^{\l}\br[\mu]{b}{[a_{\l+\del^{(1)} \,} c]_\h}\delta\mu\\
        &=\sum\nolimits_{0}^{\l}\br[\mu]{b}{[a_{\l-\mu-\h \,} c]_\h}\delta\mu.
    \end{align*}
    Putting everything together completes the proof.
\end{proof}

\begin{oss}
    The quasi-associativity formula \eqref{eq associator -1 h prod} is symmetric in $b$ and $c$. An algebra with this property is called left symmetric or pre-Lie (see \cite{burdeLeftsymmetricAlgebrasPreLie2006}). Its commutator is always a Lie bracket, and the following formula holds:
    \begin{equation}\label{eq left symmetric associator}
        a\ast_\h(b\ast_\h c)=b\ast_\h (a\ast_\h c)+(a\ast_\h b-b\ast_\h a)\ast_\h c.
    \end{equation}
    In particular, an $\h$-vertex Lie algebra $R$ has a natural Lie algebra structure, with bracket
    \begin{equation}
        [a,b]:=\sum\nolimits^0_{-\del-\h}\br{a}{b}\delta\l,\quad\forall a,b\in R.
    \end{equation}
\end{oss}

\begin{oss}
    In \cite{desoleFiniteVsAffine2006} the authors consider another bracket, which they also call $\h$-bracket, defined as
    \begin{equation}\label{def DSK h bracket}
        [a,b]_\h:=\Res_z(1+\h z)^{-1}Y_\h(a,z)b=\sum_{n\geq0}(-\h)^ja_{(j,\h)}b.
    \end{equation}
    It is clear from \eqref{def h bracket} that $[a,b]_\h=\restr{[a_{\l\,}b]_\h}{\l=-\h}$. With this in mind, the formulae in \cite{desoleFiniteVsAffine2006} involving the $\h$-bracket become a specialization of ours for $\l=-\h$. 
\end{oss}

\begin{defi}
    An $\h$-Poisson vertex algebra is a quadruple $(\Vv,1,\cdot,\del,\{\cdot_\l\cdot\}_\h)$ where
    \begin{enumerate}[(i)]
        \item $(\Vv,1,\cdot,\del)$ is a unital, commutative, differential algebra;
        \item $(\Vv,\del,\{\cdot_\l\cdot\}_\h)$ is an $\h$-vertex Lie algebra;
        \item the $\h$-bracket and the commutative product satisfy the right Leibniz rule, for all $a,b,c\in\Vv$:
        \begin{equation}\label{eq right Leibniz rule}
            \pbr{a}{bc}=b\pbr{a}{c}+c\pbr{a}{b}.
        \end{equation}
    \end{enumerate}
\end{defi}

\begin{oss}
    By skewsymmetry of the $\h$-bracket, an $\h$-Poisson vertex algebra satisfies the following left Leibniz rule:
    \begin{equation}\label{eq left Leibniz rule}
        \pbr{ab}{c}=\{b_{\l+\del} c\underrightarrow{\}_\h}a+\{a_{\l+\del} c\underrightarrow{\}_\h}b.
    \end{equation}
    Here we used the same arrow notation introduced for Poisson vertex algebras (see Definition \ref{def:poissonvertexalgebra}).
\end{oss}

\begin{prop}\label{prop:hpvachangeofvariable}
    Let $(\Vv,1,\cdot,\del)$ be a unital, commutative, differential algebra. Then $(\Vv,1,\del,\cdot, \pbrl{\cdot}{\cdot})$ is a Poisson vertex algebra if and only if $(\Vv,1,\del,\cdot,  \pbrl[\l+\h]{\cdot}{\cdot})$ is an $\h$-Poisson vertex algebra.
\end{prop}
\begin{proof}
    From Proposition \ref{prop:l+hbracket} we know that $\pbrl[\l+\h]{\cdot}{\cdot}$ defines an $\h$-vertex Lie algebra structure. Clearly, it still satisfies the right Leibniz rule. The other implication follows in the same way.
\end{proof}

\begin{oss}\label{oss:hpvaconformal}
    If a Poisson vertex algebra $\Vv$ admits a conformal grading, then equation \eqref{eq h bracket deformed vertex operator case} provides another way to construct an $\h$-Poisson vertex algebra, isomorphic to the one defined in Proposition \ref{prop:hpvachangeofvariable}.
\end{oss}

As with the case of ordinary vertex algebras, the semi-classical limit of an $\h$-vertex algebra is an $\h$-Poisson vertex algebra. 

\begin{defi}
    A {\em good filtration} on an $\h$-vertex algebra $V$ is a filtration $V=\bigcup_{n\geq0} F_n V$ such that:
    \begin{enumerate}
        \item $\del F_nV\subset F_nV$;
        \item $(F_nV)\ast_\h(F_mV)\subset F_{n+m}V $;
        \item $\br{F_nV}{F_mV}\subset F_{n+m-1}V[\l] $, where $\l$ is placed in degree $0$.
    \end{enumerate}
\end{defi}

\begin{prop}
    Let $V$ be an $\h$-vertex algebra equipped with a good filtration. Then $\gr(V)$ has a natural structure of $\h$-Poisson vertex algebra.
\end{prop}

\begin{oss}
    A good filtration for a vertex algebra $V$ is also a good filtration for the corresponding $\h$-vertex algebra given by the change of variables. This is because we can always write
    $$a_{(k,\h)}b=\sum_{n\geq k}c_n \h^n a_{(n)}b, \quad \forall a,b\in V, $$
    for some coefficients $c_n\in\CC$. In particular, this implies that the $(-1)$ and $(-1,\h)$ products coincide on the associated graded $\gr(V)$.
\end{oss}

\subsubsection{The sum \texorpdfstring{$\hbar$}{h}-bracket}\label{sec:sumhbraket}

We introduce an $\h$-version of the integral $\l$-bracket discussed in Section \ref{sec:integrallbracket}. Following the analogy between infinitesimal and discrete calculus, the integral is substituted by a definite sum and derivatives by finite difference operators. We refer to Appendix \ref{sec:finitedifferences} for the relevant definitions and key properties. The sum $\h$-bracket formalism allows us to pack together different axioms for the $\h$-bracket and the $\ast_\h$ product. Let $V$ be an $\h$-vertex algebra, and define a bilinear map $I_{\lambda,\h}:V\otimes V\to V[\lambda]$ as 
\begin{equation}\label{def sum h bracket}
    I_{\l,\h}(a,b):=a\ast_\h b+\sum\nolimits_0^\l [a_x b]_\h \delta x.
\end{equation}
We call it the sum $\h$-bracket.

\begin{prop}\label{prop sum h bracket}
    The sum $\h$-bracket satisfies the following axioms:
    \begin{enumerate}[(a)]
\item
Unity:
\begin{equation}\label{eq vacuum sum h}
I_{\lambda,\h}(|0\rangle,a)=I_{\lambda,\h}(a,|0\rangle)=a\,, \quad \forall a\in V.
\end{equation}

\item Sesquilinearity:
\begin{equation}\label{eq sesquilinearity 1 sum h}
\Delta_\h I_{\lambda,\h}(\del a,b)=-(\lambda+\h) \Delta_\h I_{\lambda,\h}(a,b)\,, \quad \del\left(I_{\lambda,\h}(a,b)\right)=I_{\lambda,\h}(\del a,b)+I_{\lambda,\h}(a,\del b)\,, \quad \forall a,b\in V,
\end{equation}
where $\Delta_\h$ is the finite difference operator (see \ref{def:finitedifference}).
\item
Skewsymmetry:
\begin{equation}\label{eq skewsymmetry sum h}
I_{\lambda,\h}(a,b)=I_{-\lambda-\del-\h,\h}(b,a)\,,\quad \forall a,b\in V.
\end{equation}
\item
Jacobi identity:
\begin{equation}\label{eq Jacobi sum h}
I_{\lambda,\h}(a,I_{\mu,\h}(b,c))-I_{\mu,\h}(b,I_{\lambda,\h}(a,c))
=I_{\lambda+\mu,\h}(I_{\lambda,\h}(a,b)-I_{-\mu-\del-\h,\h}(a,b),c)\,,
\end{equation}
for every $a,b,c\in V$.
\end{enumerate}
\end{prop}
\begin{proof}
    The unity axiom follows from definition. Sesquilinearity formulae follow from definition and from the fact that $\del$ is a derivation of all the $(n,\h)$-products. Skewsymmetry follows from a direct computation:
    \begin{align*}
        I_{-\l-\del-\h,\h}(b,a)&=
        b\ast_\h a+\sum\nolimits_0^{-\l-\del-\h}\br[x]{b}{a}\delta x =a\ast_\h b+\sum\nolimits_{-\del-\h}^{-\l-\del-\h}\br[x]{b}{a}\delta x \\
        &=a\ast_\h b-\sum\nolimits_{-\del-\h}^{-\l-\del-\h}\br[-x-\del-2\h]{a}{b}\delta x =a\ast_\h b +\sum\nolimits_{0}^{\l}\br[x]{a}{b}\delta x,
    \end{align*}
    where we used the commutator formula for $\ast_\h$ \eqref{eq commutator -1 h prod}, skewsymmetry of the $\h$-bracket and Proposition \ref{prop change of variable} for the change of variables.

    Recall that if two functions have the same finite difference, they coincide modulo an $\h$-periodic function. In the case of the Jacobi identity, both sides are polynomial functions, and the only periodic polynomial functions are the constants. Notice also that the Jacobi identity remains unchanged if we swap $a,b$ and $\l,\mu$. Thus, we only need to check that identity \eqref{eq Jacobi sum h} holds when we put $\l=\mu=0$, when we put $\mu=0$ and compute the finite difference with respect to $\l$, and finally when we compute the finite difference with respect to both $\l$ and $\mu$. If $\l=\mu=0$ we get 
    $$a\ast_\h(b\ast_\h c)-b\ast_\h(a\ast_\h c)=-\left(\sum\nolimits_0^{-\del-\h}\br[x]{a}{b}\delta x \right) \ast_\h c, $$
    which is true due to equation \eqref{eq left symmetric associator}.
    Put now $\mu=0$ in \eqref{eq Jacobi sum h} and compute the finite difference with respect to $\l$:
    $$\br{a}{b\ast_\h c}=b\ast_\h \br{a}{c}+\br{a}{b}\ast_\h c+\Delta_\h\left [\sum\nolimits_0^\l\delta y\sum\nolimits_y^\l \delta x \br[\h \, y]{[a_x b]}{c}\right ]. $$
    To compute the finite difference in the right-hand side above, we change the order of summation, using Proposition \ref{prop change order summation}, so
    $$\Delta_\h\left [\sum\nolimits_0^\l\delta y\sum\nolimits_y^\l \delta x \br[\h \, y]{[a_x b]}{c}\right ]=\Delta_\h\left[\sum\nolimits_0^\l\delta x\sum\nolimits_0^{x+\h}\delta y\br[\h \, y]{[a_x b]}{c}\right]=\sum\nolimits_0^{\l+\h}\br[\h \, \mu]{[a_\l b]}{c}\delta\mu, $$
    recovering the right Wick formula \eqref{eq right wick}.

    Now compute the finite difference of \eqref{eq Jacobi sum h} with respect to both $\l$ and $\mu$. The left-hand side becomes $\br{a}{\br[\mu]{b}{c}}-\br[\mu]{b}{\br{a}{c}} $. The right-hand side becomes
    $$\Delta_\h^{(\mu)}\Delta_\h^{(\l)}\sum\nolimits_0^{\l+\mu}\delta y\sum\nolimits_{-\mu+y}^\l\delta x \br[\h \, y]{[a_x b]}{c}. $$
    Changing the order of summation gives
    $$\Delta_\h^{(\mu)}\Delta_\h^{(\l)}\sum\nolimits_{-\mu}^\l\delta x\sum\nolimits_0^{x+\mu+\h}\delta y \br[\h \, y]{[a_x b]}{c}=\br[\h \, \l+\mu+\h]{[a_\l b]}{c},$$
    recovering the Jacobi identity of the $\h$-bracket.
\end{proof}
The following result will be useful for computations in Section \ref{section 4}. 

\begin{lemma}\label{lemma recursion integral h bracket}
    Let $V$ be an $\h$-vertex algebra. Then the sum $\h$-bracket $I_{\l,\h} $ has the following recursion properties:
\begin{align}
    & I_{\l,\h}(\z,A)=A, \label{eq recursion base case left}\\
    & I_{\l,\h}(a\ast_\h B,A)=a\ast_\h I_{\l+\del^{(1)}}(B,A)+I_{\l,\h}\left(B,\sum\nolimits_0^{\l+\del^{(1)}}{[a_x A]_\h\delta x}\right), \label{eq recursion left}
\end{align}
for all $a,A,B\in V$, and
\begin{align}
    &I_{\l,\h}(A,\z)=A, \label{eq recursion base case right}\\
    &I_{\l,\h}(A,a\ast_\h B)=a\ast_\h I_{\l,\h}(A,B)+I_{\l,\h}\left(\sum\nolimits_{-\del-\h}^\l{[A_x a]_\h\delta x},B\right), \label{eq recursion right}
    \end{align}
for all $a,A,B\in V$.
    
\end{lemma}
\begin{proof}
    The two base cases \eqref{eq recursion base case left} and \eqref{eq recursion base case right} follow from definition.

    Notice that $a\ast_\h B=I_{0,\h}(a,B)$. Thus we have, using the skewsymmetry and Jacobi identity:
    \begin{align*}
        I_{\l,\h}(a\ast_\h B,A)&=I_{\l,\h}(I_{0,\h}(a,B),A) =I_{-\l-\del-\h}(A,I_{0,\h}(a,B)) \\
        &=I_{0,\h}(a,I_{-\l-\del^{(1)}-\del^{(2)}-\h}(A,B)) +I_{-\l-\del-\h,\h}(I_{-\l-\del^{(1)}-\del^{(2)}-\h,\h}(A,a)-I_{-\del^{(1)}-\h,\h}(A,a),B).
    \end{align*}
    Using skewsymmetry again we get
    \begin{equation*}
        I_{0,\h}(a,I_{-\l-\del^{(1)}-\del^{(2)}-\h}(A,B))=a\ast_\h I_{\l+\del^{(1)}}(B,A),
    \end{equation*}
    and
    \begin{align*}
        &I_{-\l-\del-\h,\h}(I_{-\l-\del^{(1)}-\del^{(2)}-\h,\h}(A,a)-I_{-\del^{(1)}-\h,\h}(A,a),B) \\
        &=I_{\l,\h}(B,I_{\l+\del^{(1)},\h}(a,A)-I_{0,\h}(a,A)) \\
        &=I_{\l,\h}\left(B,\sum\nolimits_0^{\l+\del^{(1)}}{[a_x A]_\h\delta x}\right).
    \end{align*}
    Equation \eqref{eq recursion right} follows from one application of the Jacobi identity:
    \begin{align*}
        I_{\l,\h}(A,I_{0,\h}(a,B))&=I_{0,\h}(a,I_{\l,\h}(A,B))+I_{\l,\h}(I_{\l,\h}(A,a)-I_{-\del-\h,\h}(A,a),B) \\
        &=a\ast_\h I_{\l,\h}(A,B)+I_{\l,\h}\left(\sum\nolimits_{-\del-\h}^\l{[A_x a]_\h\delta x},B\right).
    \end{align*}
\end{proof}

\subsection{The Zhu algebra}\label{sec:zhuhvertex}

    In this section, we construct an associative algebra associated to an $\h$-vertex algebra, a Poisson algebra associated to an $\h$-Poisson vertex algebra, and a Lie algebra associated to an $\h$-vertex Lie algebra. These algebras are exactly the Zhu algebras of the corresponding vertex algebra, Poisson vertex algebra, and vertex Lie algebra. The $\h$-vertex algebra formalism we developed in the previous sections makes the construction and proofs more natural than the usual ones, which require lengthy computations. 
    
    Let $V$ be an $\h$-vertex algebra and define $J_\h$ as the vector space $V_{(-2,\h)}V$. Equivalently, $J_\h$ can be defined as $(\del V)\ast_\h V$, by formula \eqref{eq h covariance DSK prods}. For the rest of the section, we use the notation $\br[-\h]{\cdot}{\cdot}:=\restr{\br{\cdot}{\cdot}}{\l=-\h}$.

    \begin{lemma}\label{lemma J_h ideal}
        The space $J_\h$ is a two-sided ideal with respect to both the $\ast_\h$ product and the bracket $\br[-\h]{\cdot}{\cdot}$. Moreover, it is preserved by the action of $\del$.
    \end{lemma}
    \begin{proof}
        From the quasi-associativity formula \eqref{eq associator -1 h prod}, it follows that $(\del V)\ast_\h V$ is a right ideal. Let now $a,b,c\in V$. Using equation \eqref{eq left symmetric associator}:
        $$a\ast_\h ((\del b)\ast_\h c)=(\del b)\ast_\h (a\ast_\h c)+(a\ast_\h (\del b)-(\del b)\ast_\h a)\ast_\h c. $$
        Notice that $a\ast_\h (\del b)=\del(a\ast_\h b)-(\del a)\ast_\h b $, so 
        $a\ast_\h ((\del b)\ast_\h c)\subset J_\h\ast_\h V\subset J_\h$. Thus $J_\h$ is a left ideal too. By the right Wick formula \eqref{eq right wick}, 
        \begin{equation*}
            \br[-\h]{a}{(\del b)\ast_\h c}=(\del b)\ast_\h\br[-\h]{a}{c}+\br[-\h]{a}{(\del b)}\ast_\h c,
        \end{equation*}
        and, by sesquilinearity, $\br[-\h]{a}{(\del b)}=\del\br[-\h]{a}{b}$. By skewsymmetry, $\br[-\h]{a}{b}=-\br[-\h]{b}{a} \mod J_\h $, so $J_\h$ is automatically a two-sided $\br[-\h]{\cdot}{\cdot}$-ideal. Since $\del$ is a derivation of the $(-2,\h)$-product, it clearly preserves the ideal $J_\h$.
    \end{proof}

\begin{defi}
    Let $V$ be an $\h$-vertex algebra. Define the $\Zhu_\h$ algebra associated to $V$ as the vector space quotient $\Zhu_\h(V):=V/J_\h$.
\end{defi}

\begin{theorem}\label{teo Zhu h va}
        The vector space $\Zhu_\h(V)$ is a unital associative algebra, with the product induced by the $\ast_\h$ product, and the quotient class of $\z$ as identity. The commutator of $\Zhu_\h(V)$ is induced by $\h\br[-\h]{\cdot}{\cdot} $.     
\end{theorem}
\begin{proof}
    Since $J_\h$ is a two-sided ideal of $\ast_\h$, $\Zhu_\h(V)$ is a well-defined algebra. By the formula for quasi-associativity \eqref{eq associator -1 h prod}, it is clear that $\ast_\h$ is associative in the quotient. Moreover, from the equation for the commutator \eqref{eq commutator -1 h prod}, 
    $$a\ast_\h b-b\ast_\h a\equiv \sum\nolimits_{-\h}^0\br{a}{b}\delta\l \ \ \ mod \ J_\h.$$
    By Theorem \ref{teo definite sum equal sum}, the definite sum is equal to $\h\br[-\h]{a}{b} $.      
\end{proof} 

Let now $V$ be an ordinary vertex algebra. Recall Huang's construction of the Zhu algebra $\Zhu(V)$ (see Definition \ref{def:zhualgebrahuang}). We can now give a new, easy proof of its associativity.

\begin{theorem}\label{thm:zhu}
    The algebra $\Zhu(V)$ is associative. The commutator on $\Zhu(V)$ is induced by $\restr{\brl{\cdot}{\cdot}}{\l=0}$.
\end{theorem}
\begin{proof}
    The proof is now trivial. In fact, the algebra $\Zhu(V)$ is by definition the $\Zhu_\h$ algebra of the $\h$-vertex algebra associated to $V$ as in Proposition \ref{prop:changeofvariable}. This is associative by Theorem \ref{teo Zhu h va}. Moreover, by Proposition \ref{prop:l+hbracket}, $\br{\cdot}{\cdot}=\brl[\l+\h]{\cdot}{\cdot} $, so the commutator on the Zhu algebra is induced by $$\restr{\h\br[-\h]{a}{b}}{\h=1}=\restr{\h\brl[-\h+\h]{a}{b}}{\h=1}=\restr{\brl{a}{b}}{\l=0}.$$
\end{proof}

\begin{oss}
    We could consider instead a VOA $V$ and the Zhu algebra $\Zhu'(V)$ constructed using the conformal structure (see \eqref{eq:defzhu2}). Then $\Zhu'(V)$ is the $\Zhu_\h$ algebra of the $\h$-vertex algebra associated to $V$ using the conformal structure (see Remark \ref{oss:hvaconformal}), after specializing $\h$ to $1$.
\end{oss}

\begin{oss}
    Using the formalism of $\h$-vertex algebras, the proof of Theorem \ref{thm:zhu} becomes more natural (see \cite[Proposition 6.1]{huangDifferentialEquationsDuality2005} for comparison). This is because the Zhu algebra is not directly associated to the vertex algebra, but rather to the $\h$-vertex algebra associated to the vertex algebra. The usual construction has a ``hidden step'':
    \begin{equation}
\begin{tikzcd}
{(V,Y)} \arrow[r, dotted] \arrow[r, dotted] \arrow[rr, "\Zhu", bend left] & {(V,Y_\h)} \arrow[r] & {\Zhu_{\h=1}(V,Y_\h)}
\end{tikzcd}
    \end{equation}
    Now, a vertex algebra can also be seen as an $\h$-vertex algebra, after the limit $\h\rightarrow0$. It is then natural to ask, what is the $\Zhu_{\h\rightarrow0}$ algebra of the $\h\rightarrow0$-vertex algebra $V$? This is simply $$\Zhu_{\h\rightarrow0}(V)=R_V:=V/V_{(-2)}V.$$ 
    The algebra $R_V$ is known as the $C_2$-algebra associated to $V$. It was introduced by Zhu in \cite{zhuModularInvarianceCharacters1996}, where it is used to study the modular invariance of characters of a vertex algebra. Zhu proved that the dimension of $R_V$ (namely, whether it is finite) plays a crucial role. It is shown in \cite{zhuModularInvarianceCharacters1996} that $R_V$ has the structure of a Poisson algebra. This is a special case of Theorem \ref{teo Zhu h va}. In fact, $\Zhu_{\h\rightarrow0}(V)$ is clearly commutative. Moreover, by the relations in Theorem \ref{def va l bracket},
     $\restr{\brl{a}{b}}{\l=0}$
     induces a Poisson bracket on $\Zhu_{\h\rightarrow0}(V)$. 
\end{oss}

\begin{defi}
Let $\Vv$ be an $\h$-Poisson vertex algebra. Define $\Jj_\h$ as the ideal $(\del \Vv)\Vv$. The $\Zhu_\h$ algebra associated to $\Vv$ is the quotient $\Zhu_\h(\Vv):=\Vv/\Jj_\h$.
\end{defi}

\begin{theorem}\label{thm:Poissonzhualgebrahpva}
    The ideal $\Jj_\h$ is a two-sided ideal for $\pbr[-\h]{\cdot}{\cdot} $. The algebra $\Zhu_\h(\Vv)$ has the structure of a Poisson algebra, with Poisson bracket induced by $\pbr[-\h]{\cdot}{\cdot} $.
\end{theorem}

\begin{proof}
    By right Leibniz rule
    \begin{equation*}
        \pbr[-\h]{a}{(\del b)c}=(\del b)\pbr[-\h]{a}{c}+\pbr[-\h]{a}{(\del b)}c, 
    \end{equation*}
    and by sesquilinearity $\pbr[-\h]{a}{(\del b)}=\del(\pbr[-\h]{a}{b})$. Since by skewsymmetry $\pbr[-\h]{a}{b}=-\pbr[-\h]{b}{a} \mod \Jj_\h $, we have that $\Jj_\h$ is a two-sided $\pbr[-\h]{\cdot}{\cdot}$-ideal.

    Specializing $\l=-\h$, the axioms of an $\h$-vertex Lie algebra recover those of a Lie algebra. So $\pbr[-\h]{\cdot}{\cdot} $ is a Lie bracket on $\Zhu_\h(\Vv)$. Since it satisfies the Leibniz rule, it is a Poisson bracket.
\end{proof}

\begin{oss}
    Let $\Vv$ be a Poisson vertex algebra. If $\Vv$ admits a conformal grading, the Poisson Zhu algebra of $\Vv$ was constructed by De Sole and Kac in \cite[Section 6]{desoleFiniteVsAffine2006}. This corresponds to the $\Zhu_\h$ algebra of the $\h$-Poisson vertex algebra associated to $\Vv$ via the conformal structure (see Remark \ref{oss:hpvaconformal}). 

    Theorem \ref{thm:Poissonzhualgebrahpva} provides a construction of the Poisson Zhu algebra that does not require the conformal structure. It corresponds to the construction introduced by Huang for the associative Zhu algebra.    
\end{oss}

\begin{defi}
    Let $\Vv$ be a Poisson vertex algebra. The Zhu algebra of $\Vv$ is
    \begin{equation}\label{eq:defpoissonzhu}
        \Zhu(\Vv):=\Vv/(\del\Vv)\Vv.
    \end{equation}
\end{defi}

\begin{oss}
    The algebra defined in \eqref{eq:defpoissonzhu} is the $\Zhu_{\h}$ algebra of the $\h$-Poisson vertex algebra associated to $\Vv$ as in Proposition \ref{prop:hpvachangeofvariable}, at $\h=1$. Since by Proposition \ref{prop:l+hbracket} we have $\pbr{\cdot}{\cdot}=\pbrl[\l+\h]{\cdot}{\cdot} $, the Poisson bracket on the Zhu algebra is induced by $\restr{\pbrl{\cdot}{\cdot}}{\l=0}$. 
\end{oss}
Finally, let us define the $\Zhu_\h$ algebra of an $\h$-vertex Lie algebra.

\begin{defi}
    Let $R$ be an $\h$-vertex Lie algebra. Define the $\Zhu_\h$ algebra associated to $R$ as the vector space quotient $\Zhu_\h(R):=R/(\del R)$.
\end{defi}

\begin{theorem}
    The vector space $\del R$ is a two-sided ideal for $\br[-\h]{\cdot}{\cdot}$. The algebra $\Zhu_\h(R)$ has the structure of a Lie algebra, with Lie bracket induced by $\br[-\h]{\cdot}{\cdot}$. 
\end{theorem}
\begin{proof}
    This follows from the same computations from the proof of Theorem \ref{teo Zhu h va}. 
\end{proof}

\begin{oss}
    In \cite[Corollary 3.22]{desoleFiniteVsAffine2006}, the authors also introduce the Zhu algebra of a vertex Lie algebra $R$, again assuming the existence of a conformal grading. In their definition, 
    \begin{equation}\label{eq:liezhualgebraDSK}
        \Zhu(R):=R/(\del+L_0)R, \quad [a,b]:=\sum_{n\geq0}\frac{(\Delta_a-1)_{n}}{n!}a_{(n)}b.
    \end{equation}
    This coincides with the $\Zhu_\h$ algebra of the $\h$-vertex Lie algebra associated to $R$ as in Remark \ref{oss:hvaconformal}, with $\h$ specialized to $1$. Again, we can give another equivalent definition that does not rely on the conformal structure. 
\end{oss}

\begin{defi}
    Let $R$ be a vertex Lie algebra. The Zhu algebra of $R$ is 
    \begin{equation}\label{eq:zhuvertexliehuang}
        \Zhu(R):=R/\del R.
    \end{equation}
\end{defi}

\begin{oss}
    The Lie algebra defined in \eqref{eq:zhuvertexliehuang} is the $\Zhu_\h$ algebra of the $\h$-vertex Lie algebra associated to $R$ as in Proposition \ref{prop:l+hbracket}. Again, the Lie bracket is simply induced by $\restr{\brl{\cdot}{\cdot}}{\l=0}$. The Lie algebra \eqref{eq:zhuvertexliehuang} was already well-known (it appears already in the original paper of Borcherds \cite{borcherdsVertexAlgebrasKacMoody1986}). What is new is its interpretation as a Zhu algebra. 
\end{oss}

\subsection{Connection to vertex \texorpdfstring{$F$}{F}-algebras}\label{sec:F}

The translation covariance axiom in the definition of a vertex algebra \ref{def va} is equivalent to the following weak associativity axiom: for all $a,b,c\in V$, there exists $N\in\NN$ such that

\begin{equation}\label{eq:weak}
(z+w)^NY(a,i_{z,w}(z+w))Y(b,w)c=(z+w)^NY(Y(a,z)b,w)c.
\end{equation}
In \cite{liVertexFalgebrasTheir2011}, Li introduced the notion of a vertex $F$-algebra. Here, $F$ is a formal group law, i.e. a power series $F\in\CC[[x,y]]$ such that
\begin{equation}
    F(x,0)=x, \quad F(0,y)=y, \quad F(F(x,y),z)=F(x,F(y,z)).
\end{equation}
Examples of formal group laws are the additive group law $F_a(x,y)=x+y$ and the multiplicative group law $F_m(x,y)=x+y+xy$. A vertex $F$-algebra is a tuple $(V,\z,Y_F)$, where $V$ is a vector space, $\z\in V$ is the (non-zero) vacuum vector, and $Y_F(a,z):=\sum_{n\in\ZZ}a_{(n,F)}z^{-n-1}$ is a state-field correspondence. It satisfies the same axioms of a vertex algebra, with weak associativity \eqref{eq:weak} substituted by an $F$-weak associativity:
\begin{equation}\label{eq:Fweak}                F(z,w)^NY_F(a,i_{z,w}F(z,w))Y_F(b,w)c=F(z,w)^NY_F(Y_F(a,z)b,w)c.
\end{equation}
By this definition, a vertex algebra is an $F_a$-vertex algebra. In \cite[Proposition 3.12]{liVertexFalgebrasTheir2011}, $F$-weak associativity is shown equivalent to $F$-translation covariance 

\begin{equation}
    [\del,Y_F(a,z)]=\frac{1}{f'(z)}\del_zY_F(a,z),
\end{equation}
where $\del a:=a_{(-2,F)}\z$ and $f(z)$ is the \emph{logarithm} of the formal group law $F$. For $F_a$, $f(z)=z$. For $F_m$, $f(z)=\log(1+z)$. It is then clear that a $\h$-vertex algebra is the vertex $F$-algebra associated to the formal group law $F_\h(x,y)=x+y+\h xy$, which has logarithm $f(z)=\frac{1}{\h}\log(1+\h z)$ (see Definition \ref{def h va}). 

It would be interesting to see how many of our results can be generalized to general vertex $F$-algebras. Recently, the interest in $F$-vertex algebras has been renewed by the discovery of connections with enumerative geometry \cite{grossVertexFalgebraStructures2022,upmeierVertexFalgebrasTheir2025}. The further study of vertex $F$-algebras may also be of interest for the theory of vertex algebras. In \cite[Theorem 3.7]{liVertexFalgebrasTheir2011}, Li proved the following (compare with \ref{prop:changeofvariable})

\begin{theorem}\label{thm:teoLi}
    For every formal group law $F$ with $f=\log(F)$, there is an isomorphism between the category of vertex algebras and vertex $F$-algebras, given by
    \begin{equation}
        Y_F(a,z):=Y(a,f(z)).
    \end{equation}
\end{theorem}

From the discussion in Section \ref{sec:zhuhvertex}, additive (i.e. ordinary) and multiplicative (i.e. $\h$)-vertex algebras have a canonically associated algebra, defined over a quotient of $V$. For $F_a$-vertex algebras it is Zhu's $C_2$-algebra. For $F_m$-vertex algebras it is Zhu's associative algebra. But we know that Zhu's associative algebra plays an important role in the representation theory of the corresponding (via the change of variable) vertex algebra. It would then be very interesting to generalize the constructions in Section \ref{sec:zhuhvertex} to general $F$-vertex algebras, as this would generate new interesting algebras associated to a vertex algebra $V$, via the isomorphism of Theorem \ref{thm:teoLi}.
\section{Chiralization of classical star products}\label{section 4}

\subsection{Star-products for Poisson algebras}\label{sec:starproducts}

We begin the section by recalling some basic facts and definitions about star-products for Poisson algebras. We follow the Einstein summation convention of summing over repeated indices in the appropriate range. Let $(\Aa,1,\{\cdot,\cdot\})$ be a unitary Poisson algebra. 

\begin{defi}
    A star-product on $\Aa$ is a bilinear operation
    \begin{equation*}
        \star:\Aa[[\e]]\otimes \Aa[[\e]]\rightarrow \Aa[[\e]]
    \end{equation*}
    defined for $a,b\in \Aa$, and then extended to $\Aa[[\e]]$ by linearity, as 
    \begin{equation}\label{eq def star-product}
        a\star b=\sum_{n\geq0}\e^n B_n(a,b),
    \end{equation}
     where $B_n:\Aa\otimes\Aa\rightarrow\Aa$ are bilinear operators that satisfy the following axioms:
    \begin{enumerate}[(i)]
        \item The product $\star$ is associative and unitary, with unit $1$.
        \item $B_0(a,b)=ab$, i.e. $\star$ is a deformation of the commutative product on $\Aa$.
        \item $B_1(a,b)-B_1(b,a)=\{a,b\}$, i.e. the star commutator $a\star b-b\star a$ is a deformation of the Poisson bracket.
    \end{enumerate}
    The star-product is called {\em strict} if the sum \eqref{eq def star-product} is convergent for some non-zero values of $\e$.
\end{defi}

Assume now that $\Aa$ is a positively graded Poisson algebra
\begin{equation*}
    \Aa=\bigoplus_{n\geq 0}\Aa_n,
\end{equation*}
with $\Aa_0=\CC$ and Poisson bracket of degree $-i$ for some $i\geq 1$. That is, $\{\Aa_n,\Aa_m\}\subset\Aa_{n+m-i}$ for all $n,m\in\NN$. It is natural to require the star-product to be a graded deformation of the product of $\Aa$. Namely, this means requiring that
\begin{equation}\label{eq:starproductgraded}
    B_k:\Aa_n\otimes\Aa_m\longrightarrow \Aa_{n+m-ki} \hspace{10mm} \forall n,m,k\in\NN.
\end{equation}
Notice that, for graded star-products, the sum in \eqref{eq def star-product} is always finite for degree reasons. So the star-product is strict and we can specialize $\e$ to $1$.

\begin{oss}
    The convention of specializing $\e$ to $1$ instead of considering it a formal parameter, as done in this paper, is not the most common. However, once we assume the grading conditions \eqref{eq:starproductgraded}, this becomes a natural choice. In fact, there is no issue with the convergence of the series, and the various quantization terms can be identified by the degree instead of by the power of $\e$. To go back, one can use a procedure similar to the Rees algebra:
    $$a\star b:=\sum_{n\geq0}\e^{n}\pi_{\deg a+\deg b-n}(a\star b), $$
    where $\pi_i$ is the projection to the $i$-th degree component.
    For more references on the use of this convention, see the papers  \cite{etingofShortStarproductsFiltered2020,etingofTwistedTracesPositive2021}.
\end{oss}

Let $A=\bigcup_{n\geq0}F_nA$ be a filtered associative algebra with $F_0A=\CC$. We assume that the commutator is of negative degree, that is $[F_nA,F_mA]\subset F_{n+m-i}A$. Then, the associated graded
$$\gr(A)=\bigoplus_{n\geq0}F_nA/F_{n-1}A,$$
has a natural structure of a graded Poisson algebra.

\begin{defi}
    Let $A$ be a filtered, associative algebra. We say that $A$ is a quantization of $\Aa$ if $ \gr A\cong \Aa $ as graded Poisson algebras.
\end{defi}

\begin{defi}
    Let $U=\bigcup_{n\geq0}F_nU$ be a filtered vector space, and let $\Uu:=\gr(U)$. A quantization map is a vector space isomorphism $\phi:\Uu\rightarrow U$, such that, for all $n\in\NN$, $\phi(\Uu_n)\subset F_nU$ and $\pi_n\circ\phi=\id$, where $\pi_n$ is the quotient projection $  \pi_n:F_nU\rightarrow F_nU/F_{n-1}U\cong \Uu_n$.   
\end{defi}

\begin{oss}\label{oss:gradedstarproducts}
    Let $A=\bigcup_{n\geq0} F_nA$ be a filtered associative algebra with commutator of degree $-1$ and take $\Aa:=\gr(A)$. 
    If we have a quantization map $\phi:\Aa\rightarrow A$, then we can define a graded star-product as
    \begin{equation}\label{eq star-product from quant map}
        a\star_\phi b:=\phi^{-1}(\phi(a)\phi(b)).
    \end{equation}
    In fact, $\star_\phi$ is associative by definition and it is unital because $\phi(1)=1$. Let $a\in A_n$ and $b\in A_m$. Decomposing the star-product into its graded components gives
    $$a\star_\phi b=C_0(a,b)+C_1(a,b)+\dots, $$
    with $C_i(a,b)=p_{n+m-i}(a\star_\phi b)$, where $p_n$ is the projection to the $n$-th graded component.\\    
    Since $\phi$ is a quantization map, $a=\phi(a)+F_{n-1}A$ and $b=\phi(b)+F_{m-1}A$. Hence,
    $$C_0(a,b)=p_{n+m}(a\star_\phi b)=\pi_{n+m}(\phi(a)\phi(b))=\phi(a)\phi(b)+F_{n+m-1}A=ab. $$
    Similarly,
    $$C_1(a,b)-C_1(b,a)=p_{n+m-1}(a\star_\phi b-b\star_\phi a)=[\phi(a),\phi(b)]+F_{n+m-2}=\{a,b\}.  $$
    Indeed, all graded star-products can be realized this way. If $\star$ is a graded star-product, take $A:=(\Aa,\star)$. The identity function is a quantization map $\Aa\rightarrow A$ and $\star_{\id}=\star$.
\end{oss}

\begin{oss}
    The construction in Remark \ref{oss:gradedstarproducts} can be generalized to filtered associative algebras of arbitrary negative degrees. In those cases, we need to require some additional conditions, related to the grading, on the algebra $A$ and the quantization map $\phi$ (see \cite{etingofShortStarproductsFiltered2020}). We restrict to the degree $-1$ case for simplicity, since all our examples can be reduced to this case.
\end{oss}

\begin{es}[Moyal-Weyl star-product]\label{es Moyal Weyl}
    Let $U$ be a symplectic vector space, with symplectic form $\omega$. Let $\Ss(U)$ be the symmetric algebra, which is canonically a Poisson algebra with Poisson structure induced by $\omega$. The Weyl algebra 
    $$W(U)=\ \quotient{\Tt(U)}{\langle u\otimes v-v\otimes u-\o(u,v) \, | \, \forall u,v\in U\rangle}$$
    is well-known to be a quantization of $\Ss(U)$. Let $v_1,\dots,v_n$ be a basis of $U$. The map $\phi:\Ss(U)\rightarrow W(U)$, given by symmetrization,
     \begin{equation}\label{eq symmetrization map}
         v_{i_1}\cdots v_{i_k}\mapsto \frac{1}{k!}\sum_{\sigma\in S_k}v_{i_{\s(1)}}\cdots v_{i_{\s(k)}},
     \end{equation}
     is a quantization map. Let $\pi\in\Lambda^2U$ be the Poisson bivector $\pi=\pi^{ij}\del_i\wedge \del_j$ induced by $\o$. The star-product $\star_\phi$ has the following explicit formula: 
    \begin{equation}\label{eq moyal weyl star-product}
        a\star_\phi b=m(\exp\left(\pi/2)(a\otimes b)\right),
    \end{equation}
     where $m$ is the multiplication map $a\otimes b\mapsto ab$. For example, if $U=\CC x\oplus \CC y$, $\pi=\del_x\wedge\del_y$
     $$a\star_\phi b=m\circ\exp\left(\frac{\del_x\otimes\del_y-\del_y\otimes\del_x}{2} \right)(a\otimes b). $$
\end{es}

\begin{oss}
    The commutator on the Weyl algebra is of degree $-2$. One can check explicitly that the quantization map $\phi$ defined in \eqref{eq symmetrization map} induces a star-product (it satisfies the $\ZZ/2\ZZ$-equivariant condition described in \cite{etingofShortStarproductsFiltered2020}). Alternatively, we can add a formal variable $t$, central and of degree $1$, and homogenize the defining relations to 
    $$u\otimes v-v\otimes u-t\o(u,v), \quad u,v\in U. $$
    We extend $\phi:\Ss(U)[t]\rightarrow W(U)[t]$ by sending $t\mapsto t$ and obtain a star-product $\star_\phi$ on $\Ss(U)[t]$ as in Remark \ref{oss:gradedstarproducts}. Setting $t=1$, this becomes the Moyal-Weyl star-product. We prefer the latter point of view, as it makes the connection between the Moyal-Weyl and Gutt star-product more explicit (see Remark \ref{oss:moyalweylspecialcaseGutt}).
\end{oss}

\begin{es}[Gutt star-product]\label{es gutt star-product}
    Let $\g$ be a Lie algebra. Then the symmetric algebra $\Ss(\g)$ is a graded Poisson algebra with Poisson bracket induced by the Lie bracket on $\g$. From the Poincaré–Birkhoff–Witt theorem, the universal enveloping algebra $\UUU(\g)$ is a quantization of $\Ss(\g)$. Let $\{v^i\}_{i\in\Ii}$ be a basis of $\g$. The map $\phi:\Ss(\g)\rightarrow \UUU(\g)$, given by symmetrization as in \eqref{eq symmetrization map}, is a quantization map. For all $x,y\in\g$, let $\mathrm{BCH}(x,y)$ be the formal series given by the Baker-Campbell-Hausdorff formula:
     \begin{equation}\label{eq BCH}
         \begin{aligned}
             &\exp(x)\exp(y)=\exp(\mr{BCH}(x,y))=\exp(x+y+\frac{1}{2}[x,y]+\frac{1}{12}([x,[x,y]]+[[x,y],y])+\dots),
         \end{aligned}
     \end{equation}
     where the identity above holds in the ring of formal power series in the non-commutative variables $x,y$. Define now the following symbol in the commuting variables $\underline{s}=(s_i)_{i\in\Ii}$ and $\underline{t}=(s_i)_{i\in\Ii}$:
     \begin{equation}\label{eq def diff operator BCH}
         D(\underline{s},\underline{t}):=\frac{1}{2}s_it_j[v^i,v^j]+\frac{1}{12}(s_is_jt_k[v^i[v^j,v^k]]+s_it_jt_k[[v^i,v^j],v^k]])+\dots
     \end{equation}
     so that, if $\underline{s},\underline{t}\in\CC^n$, $x=s_iv^i$ and $y=t_iv^i$,
     \begin{equation*}
         D(\underline{s},\underline{t})=\mr{BCH}(x,y)-x-y.
     \end{equation*}
     The star-product $\star_\phi$ has the following explicit formulation. For $a,b\in\Ss(\g)$,
     \begin{equation}\label{eq BCH Gutt star-product}
         a\star_\phi b=m\circ\exp(D(\overleftarrow{\underline{\del}},\overrightarrow{\underline{\del}}))(a\otimes b),
     \end{equation}
     where $\overleftarrow{\underline{\del}}=(\del_{v_i}\x1)_{i\in\Ii}$ and $\overrightarrow{\underline{\del}}=(1\x \del_{v_i})_{i\in\Ii}$.
\end{es}
\begin{oss}
    The Gutt star-product is usually defined over finite-dimensional Lie algebras, but the construction and formula \eqref{eq BCH Gutt star-product} work for infinite-dimensional Lie algebras too (see, for example, \cite[Section 2.2]{espositoConvergenceGuttStar2017}). Notice that all the sums appearing in the definition of $D(\underline{s},\underline{t})$ are now infinite, but only a finite number of terms of $D(\overleftarrow{\underline{\del}},\overrightarrow{\underline{\del}})$ applied to elements in $\Ss(\g)$ are non zero.
\end{oss}

\begin{oss}\label{oss:moyalweylspecialcaseGutt}
    The Moyal-Weyl star-product can be seen as a special case of the Gutt star-product. In fact, consider the $2n+1$ dimensional Heisenberg Lie algebra $\hh_n$, with basis $x_1,\dots,x_n$, $y_1,\dots,y_n,Z$, and bracket
    \begin{equation*}
        [x_i,y_j]=\delta_{i,j}Z, \hspace{5mm} [x_i,x_j]=[y_i,y_j]=[Z,x_i]=[Z,y_i]=0,
    \end{equation*}
    for all $i,j$. The quotient $\UUU(\hh_n)/(Z-1)$ is isomorphic to the Weyl algebra $W(U)$, where $U$ is a $2n$ dimensional vector space. In the Baker-Campbell-Hausdorff formula for $\hh_n$ only the first two terms may appear, so that $$D(\overleftarrow{\underline{\del}},\overrightarrow{\underline{\del}})=Z\sum_{i=1}^n\frac{\del_{x_i}\otimes\del_{y_i}-\del_{y_i}\otimes\del_{x_i}}{2},$$ 
    and, in the quotient, \eqref{eq BCH Gutt star-product} becomes \eqref{eq moyal weyl star-product}. Another way of seeing this fact is to consider a symplectic vector space $U$ as a non-linear Lie algebra, with a Lie bracket with coefficients in $\CC$. 
\end{oss}

\begin{oss}
    Note how both formulae \eqref{eq moyal weyl star-product} and \eqref{eq BCH Gutt star-product} are of the form
    $$\exp (\text{bidifferential operator}),$$
    where the bidifferential operator only depends on the structure constants of the Poisson bracket.
\end{oss}

\subsection{Star-deformations for Poisson vertex algebras}\label{sec:stardeformations}

The definition of star-deformations for Poisson vertex algebras was given by Li in \cite[Definition 5.20]{liVertexAlgebrasVertex2004}, motivated by the analogy with Poisson algebras. In this definition, a star-deformation is a formal state-field correspondence $Y_\e(-,z)$ on a Poisson vertex algebra $\Vv$, such that the non-negative $(n)$-products give a deformation of the Poisson $\l$-bracket. We give here a different, equivalent definition, based on the integral $\l$-bracket formalism (see Definition \ref{def va integral}).

\begin{defi}\label{def:stardeformation}
    Let $(\Vv,1,\del,\{\cdot_\l\cdot\})$ be a Poisson vertex algebra. Then a star-deformation of $\Vv$ is a bilinear operation
    \begin{equation*}
        I_{\l,\star}:\Vv[[\e]]\otimes \Vv[[\e]]\rightarrow \Vv[[\e,\l]]
    \end{equation*}
    defined for $a,b\in \Vv$, and then extended to $\Vv[[\e]]$ by linearity, as 
    \begin{equation}\label{eq def star integral l bracket}
        I_{\l,\star}(a,b)=\sum_{n\geq0}\e^n I_{\l,n}(a,b),
    \end{equation}
     where $I_{\l,n}:\Vv\otimes \Vv\rightarrow \Vv[\l]$ are bilinear operators that satisfy the following axioms:
    \begin{enumerate}[(i)]
        \item $(1,\del,I_{\l,\star})$ induce a vertex algebra structure on $\Vv/\e^n \Vv$, for all $n\geq1$.
        \item $I_{0,0}(a,b)=ab$ and $\frac{d}{d \l}I_{\l,0}(a,b)=0$ for all $a,b\in\Vv$.
        \item $\frac{d}{d\l}I_{\l,1}(a,b)=\{a_\l b\}$ for all $a,b\in\Vv$.
    \end{enumerate}
    The star-deformation is called {\em strict} if the sum \eqref{eq def star integral l bracket} is convergent for some non-zero values of $\e$.
\end{defi}

\begin{oss}
     Due to convergence issues, a star-deformation $I_{\l,\star}$ will not define a vertex algebra structure on $\Vv[[\e]]$. In fact, even if $I_{\l,\star}(a,b)\in\Vv[[\e]][\l]$ for $a,b\in\Vv$, 
     \begin{equation}
         I_{\l,\star}\left(\,\sum_{n\geq0}a_n\e^n, b\,\right)=\sum_{n\geq0}I_{\l,\star}(a_n,b)\e^n\in \Vv[[\e,\l]].
     \end{equation}
     This explains the technicality in condition (i). In the terminology of \cite{liVertexAlgebrasVertex2004}, $I_{\l,\star}$ induces a $\e$-adic vertex algebra structure on $\Vv[[\e]]$. 
\end{oss}
Assume now that $\Vv$ is positively graded with Poisson $\l$-bracket of degree $-i$, for some $i\geq 1$, that is
\begin{equation*}
    \Vv=\bigoplus_{n\geq 0}\Vv_n, \quad \Vv_n\cdot\Vv_m\subset\Vv_{n+m}, \quad \{\Vv_n\,_\l\Vv_m\}\subset\Vv_{n+m-i}\otimes\CC[\l], \quad \forall m,n\in\NN.
\end{equation*}
We again require that the star-deformation is a graded deformation of the Poisson vertex algebra structure on $\Vv$. Namely, this means imposing that
\begin{equation*}
    I_{\l,k}:\Vv_n\otimes\Vv_m\longrightarrow \Vv_{n+m-ki}\otimes\CC[\l] \hspace{10mm} \text{for all $n,m$ and $k$}.
\end{equation*}
Since the sums are finite, the star-deformation is strict and we can specialize $\e$ to $1$ and work on $\Vv$.

\begin{defi}
    Let $V$ be a vertex algebra with a good filtration $\{F_nV\}_{n\geq0}$. Then $\gr(V)$ is a Poisson vertex algebra (see Section \ref{sec:poissonvertex}). We say that $V$ is a quantization of $\Vv$ if $ \gr V\cong \Vv $ as graded Poisson vertex algebras.
\end{defi}

\begin{oss}\label{oss:gradedstardeformations}
    Let $V$ be a filtered algebra with a good filtration, with $\l$-bracket of degree $-1$ and set $\Vv=\gr(V)$. From a quantization map $\phi:\Vv\rightarrow V$, such that $\phi(\del a)=\del\phi(a)$, for all $a\in\Vv$, we can construct a star-deformation in the following way. Extend $\phi:\Vv[\l]\rightarrow V[\l]$ by linearity, then define
    \begin{equation}\label{eq star integral l bracket from quant map}
        I_{\l,\star}(a,b):=\phi^{-1}I_\l(\phi(a)\phi(b)).
    \end{equation}
    On the other hand, if $I_{\l,\star}$ is a graded star-deformation on $\Vv$, then $(\Vv,1,\del,I_{\l,\star})$ with filtration induced by the grading is a quantization of $\Vv$, with the identity as quantization map. The proof is basically the same as the one for star-products (see Remark \ref{oss:gradedstarproducts}). 
\end{oss}

\begin{problem}
    Compute explicit, closed formulae for star-deformations of Poisson vertex algebras, using quantization maps.
\end{problem}

As we explain in the next section, a star-deformation is not the ``correct'' chiral analogue of a star-product. Nonetheless, explicit formulae for star-deformations have important applications. Vertex algebras have very complicated relations and computations, even with the $\l$-bracket, become complicated very quickly. A formula for a star-deformation allows one to perform vertex algebra computations in the setting of Poisson vertex algebras, which are much more tractable. From the physics point of view, it corresponds to constructing a (very special) quantum field theory inside the formalism of classical field theory. We defer the computation of these formulae to Section \ref{sec:explicitstar}

\subsection{Chiral star-products}\label{sec:chiralstarproducts}

\iftrue

Our aim is to compute explicit formulae for some chiral analogues of the Moyal-Weyl and Gutt star-products. First, we need to find a suitable definition of what a chiral star-product should be. Informally, chiralization should be thought of as the inverse of the Zhu functor. Let $\Aa$ be a Poisson algebra and $\Vv$ a Poisson vertex algebra, such that $\Aa=\Zhu(\Vv)$. Let $\star$ be a star-product on $\Aa$, so $\star$ is a deformation of the Poisson structure on $\Aa$ such that $(\Aa,\star)$ is an associative algebra. Ideally, a chiral star-product would satisfy the following two properties:
\begin{enumerate}
    \item It is a suitable deformation of the Poisson vertex algebra structure on $\Vv$.
    \item Let $\hat\star$ denote the deformation of the commutative product on $\Vv$ given by the chiral star-product and let $p_Z:\Vv\rightarrow \Vv/\Jj_\h=\Aa$ be the Zhu quotient map. Then, 
\begin{equation}\label{eq:conditionchiralstarproduct}
    {p_Z(a\hat{\star}b)}=p_Z(a)\star p_Z(b), \quad \forall a,b\in \Vv.
\end{equation}
\end{enumerate}
A strict star-deformation satisfies the first condition. Unfortunately, a star-deformation in general does not satisfy the second condition. In fact, the star-deformation is in general not even compatible with the Zhu quotient. This is because the space we quotient by is defined in terms of the $\h$-Poisson vertex algebra structure, and may not be directly compatible with the vertex algebra structure induced by the star-deformation (see Section \ref{sec:zhuhvertex}). This suggests the following definition.

\begin{defi}\label{def:chiralstarproduct}
    Let $(\Vv,1,\del,\pbrl{\cdot}{\cdot})$ be a Poisson vertex algebra and $\Aa=\Zhu(\Vv)$. A chiral star-product is a bilinear operation
    $$I_{\l,\h,\star}:\Vv\otimes\Vv\rightarrow \Vv[\l], $$
    satisfying the following:
    \begin{enumerate}[(i)]
    \item $(\Vv,1,\del,I_{\l,\h,\star})$ defines an $\h$-vertex algebra on $\Vv$ (with the $\h$-sum bracket formalism, see Section \ref{sec:sumhbraket}), such that $I_{\l,\h=0,\star}$ is a star-deformation on $\Vv$.  
    \item The operation $I_{0,\h,\star}$ induces a star-product $\star$ on the Poisson Zhu algebra $\Aa$ by
        $$p_Z(a)\star p_Z(b):=p_ZI_{0,\h,\star}(a, b),$$
        such that $(\Aa,\star)$ is isomorphic to the associative Zhu algebra of the $\h$-vertex algebra $(\Vv, I_{\l,\h,\star})$. 
    \end{enumerate}
\end{defi}

Let $V$ be a vertex algebra with a good filtration, $\Vv$ a Poisson vertex algebra, $A$ an associative algebra, and $\Aa$ a Poisson algebra that fit in the following diagram:
\begin{equation}\label{diagrammasimboli}
    \begin{tikzcd}
\Vv \arrow[d, "p_Z"'] & V \arrow[l, "\gr"'] \arrow[d, "q_Z"] \\
\Aa                    & A \arrow[l, "\gr"']                   
\end{tikzcd}
\end{equation}
where $p_Z,q_Z$ denote the Zhu quotient maps. Let $\Jj_\h$ and $J_\h$ be the respective kernels and let $\tau_n:F_n V\rightarrow F_nV/F_{n-1}V$ and $\pi_n:F_n A\rightarrow F_nA/F_{n-1}A$ denote the projection maps to the associated graded, for all $n\in\NN$. Then
\begin{equation}\label{eq zhu gr commute}
    \pi_n\circ q_Z=p_Z\circ\t_n, \quad \forall n\in\NN.
\end{equation}

\begin{prop}\label{prop chiralization star-product}
    Assume that the $\l$-bracket on $V$ is of degree $-1$. Let $\phi$ be a quantization map $\Vv\rightarrow V$, such that $\phi\circ\del=\del\circ\phi$ and $\phi(\Jj_\h)=J_\h$. Then $\phi$ induces a quantization map $\hat{\phi}:\Aa\rightarrow A$, defined by
\begin{equation*}
    \hat{\phi}(a+\Jj_\h)=\phi(a)+J_\h.
\end{equation*}
    Moreover, the operation $I_{\l,\h,\star}$ defined by
\begin{equation*}
   I_{\l,\h,\star}(a,b):=\phi^{-1}\left(I_{\l,\h}(\phi(a),\phi(b)) \right), \quad \forall a,b\in\Vv,
\end{equation*}
    is a chiral star-product, where $I_{\l,\h}$ is the sum $\h$-bracket of the $\h$-vertex algebra associated to $V$. It is the chiralization of the star-product $\star_{\hat\phi}$.
\end{prop}
\begin{proof}
    Since $\Zhu(\Vv)=\Vv/\Jj_{\h} $ and $\Zhu(V)=V/J_{\h}$, the map $\hat{\phi}$ is a well defined vector space isomorphism. We are left to prove that $\pi_n\circ\hat{\phi}=\id$ for all $n\in\NN$. Let $a$ be a homogeneous element of degree $n$ in $\Aa$. Then $a=p_Z(x)$, for some $x\in\Vv$. By \eqref{eq zhu gr commute}, $\hat{\phi}$ inherits the quantization map property from $\phi$:
    \begin{equation*}
        \p_n\circ\hat{\phi}(a)=\p_n\circ q_z(\phi(x))=p_Z(\t_n\circ \phi(x))=p_Z(x)=a.
    \end{equation*}
    Since $\phi$ is an isomorphism of vector spaces commuting with $\del$, $(\Vv,1,\del,I_{\l,\h,\star})$ defines an $\h$-vertex algebra structure on $\Vv$. Moreover, $I_{\l,0,\star}$ is a star-deformation on $\Vv$ by Remark \ref{oss:gradedstardeformations}. Now, let $\star_{\hat{\phi}}$ be the star-product on $\Aa$ defined by $\hat{\phi}$. From a direct computation:
    \begin{align*}
        p_Z(I_{0,\h,\star}(a,b))&=p_Z[\phi^{-1}(\phi(a)\ast_\h\phi(b))]=\hat{\phi}^{-1}\circ q_Z(\phi(a)\ast_\h\phi(b))\\
        &=\hat{\phi}^{-1}[q_Z(\phi(a))\cdot q_Z(\phi(b))]=\hat{\phi}^{-1}[\hat{\phi}(p_Z(a))\cdot\hat{\phi}(p_Z(b))]\\
        &=p_Z(a)\star_{\hat{\phi}}p_Z(b).
    \end{align*}
    Thus $I_{\l,\h,\star}$ is a chiral star-product. 
\end{proof}

We now have a suitable definition of chiral star-products and a way to construct them. We would like to construct chiral star-products that are the chiralization of known star-products, in particular the Moyal-Weyl and Gutt star-products. We put ourselves in the following setting. Let $R$ be a sub-linear vertex Lie algebra. Assume that $R=\CC[\del]\otimes \g$ is a free $\CC[\del]$-module, generated by a finite dimensional vector space $\g$. Consider the universal enveloping vertex algebra $V(R)$ (see Theorem \ref{thm:universalvertexalgebra}). Recall the definition of the Lie Zhu algebra of a vertex Lie algebra $R$ as in \eqref{eq:zhuvertexliehuang}: 
$$\Zhu(R)=R/\del R.$$
This is naturally a Lie algebra, with bracket induced by $\restr{\brl{\cdot}{\cdot}}{\l=0}$.

\begin{theorem}\label{teo zhu algebra of universal enveloping va}
    Let $R=\CC[\del]\otimes \g$ be as above. Then 
\begin{equation*}
    \Zhu(V(R))\cong \UUU(\g).
\end{equation*}
\end{theorem}
\begin{proof}
    This is essentially {\cite[Corollary 3.26]{desoleFiniteVsAffine2006}}, rewritten by considering the change of variables Zhu algebra construction. In particular, we consider the Lie Zhu algebra $\Zhu(R)$ defined as in \eqref{eq:zhuvertexliehuang}, while De Sole and Kac consider the one defined in \eqref{eq:liezhualgebraDSK}.  
\end{proof}

Fix $\overline{\Bb}=\{u_1,\dots,u_n\}$ a basis of $\g$ and extend it to a basis of $R$ by setting $\Bb=\{u_\a\}_{\a\in I}$, where $I=\{1,\dots,n\}\times\NN$ and $u_{(i,k)}:=\del^ku_i$. Order $\Bb$ by lexicographic order. Since $V(R)$ is PBW generated by $R$, ordered monomials in elements of $\Bb$ with respect to the normally ordered product form a basis of $V(R)$. Consider the good filtration $\{F_nV(R)\}_{n\geq0}$ induced by the grade of the PBW monomials.

\begin{oss}\label{oss:PBWfiltrationhproduct}
    As proved in \cite[Lemma 3.18]{desoleFiniteVsAffine2006}, monomials with respect to the $\ast_\h$ product give an alternative PBW basis of $V(R)$. Since $a\ast_\h b$ and ${:}ab{:}$ differ only by non-negative $(n)$-products, which are in lower terms of the filtration, the filtration induced by this second PBW basis coincides with $\{F_nV(R)\}_{n\geq0}$.
\end{oss}

Consider $\Ss(R)$, the symmetric algebra on $R$, with its natural structure of Poisson vertex algebra, induced by $R$. The Poisson vertex algebra $\Ss(R)$ can also be thought of as the algebra of differential polynomials on $\g$:
\begin{equation*}
    \Ss(R)\cong\CC[\del^{k}u_i]^{k\geq0}_{i=1,\dots,n}=\CC[u_\alpha]_{\alpha\in I},
\end{equation*} 
with Poisson $\l$-bracket 
$$ \pbrl{\del^k u_i}{\del^m u_j}=(-\l)^k(\l+\del)^m\brl{u_i}{u_j}. $$
The Zhu algebra associated to $\Ss(R)$ is
\begin{equation*}
    \Zhu(\Ss(R))=\Ss(R)/(\Ss(R)\del\Ss(R))\cong \Ss(\Zhu (R))=\Ss(\g).
\end{equation*}
Thus, diagram \eqref{diagrammasimboli} in this new setting becomes:
\begin{equation}\label{diagramma zhu gr commute}
    \begin{tikzcd}
        \Ss(R)  \arrow[d, "p_Z"']            & V(R) \arrow[l, "\gr"'] \arrow[d, "q_Z"]\\
        \Ss(\g)             & \UUU(\g) \arrow[l, "\gr"]
    \end{tikzcd}
\end{equation}

\begin{theorem}\label{thm:liftingquantizationmap}
    Any quantization map $\hat{\phi}:\Ss(\g)\rightarrow\UUU(\g)$ can be lifted to a quantization map $\phi:\Ss(R)\rightarrow V(R)$ such that $\hat{\phi}(a+\Jj_\h)=\phi(a)+J_\h$. The chiral star-product defined by $\phi$ as in Proposition \ref{prop chiralization star-product} is the chiralization of the star-product $\star_{\hat{\phi}}$, defined as in Remark \ref{oss:gradedstarproducts}.
\end{theorem}

To prove the theorem, we need a technical result.

\begin{lemma}\label{lemma basis J_h}
\leavevmode
\begin{enumerate}[(a)]
    \item $B_1=\{u_{i_1}\dots(\del u_{i_k})\dots u_{i_n} \, | \, u_{i_j}\in\Bb, n\geq 1, i_1<i_2<\dots<i_n \}$ is a basis of $\Jj_\h$.
    \item $B_2=\{u_{i_1}\ast_\h\dots\ast_\h(\del u_{i_k})\ast_\h\dots\ast_\h u_{i_n} \, | \, u_{i_j}\in\Bb, n\geq 1, i_1<i_2<\dots<i_n \}$ is a basis of $J_\h$.
\end{enumerate}
In particular, as vector spaces, $J_\h\cong\Jj_\h$.
\end{lemma}
\begin{proof}
    A generic element in $\Jj_\h$ is of the form $(\del a)b$, with $a,b\in\Ss(R)$. Writing $a,b$ as polynomials in elements of $\Bb$ and since $\del$ is a derivation, $(\del a)b$ is spanned by elements in $B_1$. Since it is also a linearly independent set, $B_1$ is a basis of $\Jj_\h$.

    For part (b), we first show that $J_\h\subset \Span_\CC (B_2) $. Let $x\in J_\h$. We proceed by induction on $n$, the lowest natural number such that $x\in F_n V(R)$. We can assume that $x=(\del a)\ast_\h b$ for some $a,b\in V(R)$. Write $a,b$ in the PBW basis with respect to the product $\ast_\h$. Since $\del$ is a derivation of $\ast_\h$, $x$ is a linear combination of elements 
    \begin{equation*}
        x_\a=(u_{i_1}\ast_\h\dots\ast_\h(\del u_{i_k})\ast_\h\dots\ast_\h u_{i_m})\ast_\h(u_{j_1}\ast_\h\dots\ast_\h u_{j_t}),
    \end{equation*}
    with $u_{i_s}, u_{j_q}\in\Bb$, $i_1<\dots<i_m$, $j_1<\dots<j_t$, and $m+t\leq n$.
    By the quasi-associativity axiom of $\ast_\h$ \eqref{eq associator -1 h prod}, we can rearrange the parentheses so that $x_\a\in B_2$, modulo elements in $F_{n-1}V(R)$. By Theorem \ref{teo Zhu h va}, all the associators are elements of $J_\h$. So $x_\a\in B_2$ modulo some terms in $F_{n-1}V(R)\cap J_\h$. By induction hypothesis, we can write $x$ as a linear combination of elements in $B_2$. 
    
    Since it is a subset of the PBW basis, $B_2$ is a linearly independent set. We are left to prove that $B_2\subset J_\h$. Let 
    $B_3:=\{a_{1}\ast_\h\dots\ast_\h(\del a_{k})\ast_\h\dots\ast_\h a_{n} | a_{i}\in R, n\geq1 \}.$  
    Clearly, $B_2\subset B_3$. We show that $B_3\subset J_\h$. Take $y=a_{1}\ast_\h\dots\ast_\h(\del a_{k})\ast_\h\dots\ast_\h a_{n}\in B_3$. We proceed by induction on $n$. Using equation \eqref{eq left symmetric associator} multiple times, we can move $(\del a_{k})$ to the first position on the left, modulo terms of the form 
    \begin{equation*}
        (\dagger)\ a_{1}\ast_\h\dots\ast_\h {a_{{l-1}}}\ast_\h (a_{l}\ast_\h (\del a_{k})- (\del a_{k})\ast_\h a_{l})\ast_\h a_{{l+1}}\ast_\h \dots\ast_\h a_{{k-1}}\ast_\h a_{{k+1}}\ast_\h\dots\ast_\h  a_{n}.
    \end{equation*}
    By the proof of Lemma \ref{lemma J_h ideal}, and since the commutator goes down in the filtration, 
    \begin{equation*}
        a_{l}\ast_\h (\del a_{k})- (\del a_{k})\ast_\h a_{l}\in J_\h\cap F_1V(R)=\del R.
    \end{equation*}
    Hence $(\dagger)$ is an element in $B_3\cap F_{n-1}V(R)$. By induction, it follows that $y\in J_\h$.
\end{proof}

\begin{proof}[Proof (of Theorem \ref{thm:liftingquantizationmap})]
    Recall that $R=\CC[\del]\otimes \g$, so, as vector spaces,
    \begin{equation}\label{eq:decompositionS(R)}
        V(R)\cong\Ss(R)=\Ss(\g)\oplus \del\Ss(R),
    \end{equation}
    where $\del\Ss(R)$ is the span of the ordered monomials in $\Bb$, where at least one factor is of the form $\del u_{i}$, for $u_i\in \Bb$. By Lemma \ref{lemma basis J_h}, $\del\Ss(R)=\Jj_\h\cong J_\h$, as vector spaces. We can extend $\hat{\phi}$ to an isomorphism of vector spaces $\phi:\Ss(R)\rightarrow V(R)$ using decomposition \eqref{eq:decompositionS(R)} and   
    identifying the bases of $J_\h$ and $\Jj_h$. By Remark \ref{oss:PBWfiltrationhproduct}, the PBW basis of $V(R)$ with respect to the $\ast_\h$-product induces the good filtration. Thus $\phi$ is a quantization map. 

    We now want to apply Proposition \ref{prop chiralization star-product}. First, we need to make sure that the $\l$-bracket is of degree $-1$ (recall that $R$ is a sub-linear vertex Lie algebra, so this is not guaranteed). Consider a central extension of $\g$ by a new variable $t$. Let $R'=\CC[\del]\otimes (\g\oplus\CC t)$. It is a vertex Lie algebra, with $\l$-bracket obtained by homogenizing the $\l$-bracket of $R$. The vertex algebra $V(R')$ has a good filtration and $\l$-bracket of degree $-1$. Extend $\phi:\Ss(R')\rightarrow V(R')$ by sending $t\mapsto t$. We are now in the hypothesis of Proposition \ref{prop chiralization star-product}, so $\phi$ induces a chiral star-product on $\Ss(R')$. To get a chiral star-product on $\Ss(R)$, it is sufficient to put $t=1$.
\end{proof}

\begin{es}
    Let $\hat{\phi}:\Ss(\g)\rightarrow\UUU(\g)$ be the symmetrization map
    \begin{equation}
         v_{i_1}\cdots v_{i_k}\mapsto \frac{1}{k!}\sum_{\sigma\in S_k}v_{i_{\s(1)}}\cdots v_{i_{\s(k)}}, \quad v_{i_j}\in \bar{\Bb}.
     \end{equation}
     A chiralization of $\hat\phi$ is $\phi:\Ss(R)\rightarrow V(R)$ given by
    \begin{equation}\label{eq symmetrization map h prod}
    u_{i_1}\dots u_{i_n}\mapsto \frac{1}{n!}\sum_{\s\in S_n}u_{i_{\s(1)}}\ast_\h\dots\ast_\h u_{\s(i_n)}, \quad u_{i_j}\in\Bb.
\end{equation}
\end{es}
\begin{es}
    Both the Moyal-Weyl and Gutt star-products are constructed using the symmetrization map $\hat{\phi}$, where $\g$ is either a Lie algebra (Gutt) or a symplectic vector space (Moyal-Weyl), that can be thought of as a sub-linear Lie algebra with coefficients in $\CC$. The quantization map $\phi$ \eqref{eq symmetrization map h prod} induces chiralization of these two star-products.
\end{es}

\subsection{Explicit formulae for chiralized star-products}\label{sec:explicitcomputations}

We keep the same notation and conventions of Section \ref{sec:chiralstarproducts}. In this section we compute an explicit formula for the chiral star-product induced by the quantization map $\phi$ \eqref{eq symmetrization map h prod}. 

\begin{oss}
    The map $\phi:\Ss(R)\rightarrow V(R)$ can be written as the following composition:
    \begin{equation}\label{diagram factor quantization map}
        \begin{tikzcd}
            \Ss(R) \arrow[r, "\ga"]\arrow[rr, "\phi"', bend right] & \UUU(R_L) \arrow[r, "\psi"] &V(R),
        \end{tikzcd}
    \end{equation}
    where 
    \begin{itemize}
        \item $R_L$ is the Lie algebra associated to $R$, with Lie bracket given by 
        \begin{equation*}
            [a,b]:=\sum\nolimits^0_{-\del-\h}\br[x]{a}{b} \d x,
        \end{equation*}
        and $\UUU(R_L)$ is its universal enveloping algebra;
        \item $\gamma$ is the symmetrization map $\Ss(R)\rightarrow \UUU(R_L)$ given by \eqref{eq symmetrization map};
        \item $\psi$ is the map defined inductively by $\psi(a):=a$, for all $a\in R$ and $\psi(aB):=a\ast_\h B$, for all $a\in R$ and  for all $B\in \UUU(R_L)$.
    \end{itemize}
    Both $\ga$ and $\psi$ are vector spaces isomorphisms; we are thus performing the quantization $\Ss(R)\rightarrow V(R)$ in two steps: the first is a non-commutative deformation, the second is a non-associative deformation.
\end{oss}
Let us first consider the non-associative deformation induced by $\psi$. Define $$\hat{I}_{\l,\h}:\UUU(R_L)\otimes \UUU(R_L)\rightarrow \UUU(R_L)[\l]$$ as, for $a,b\in \UUU(R_L)$,
\begin{equation*}
    \hat{I}_{\l,\h}(a,b):=\psi^{-1}(I_{\l,\h}(\psi(a),\psi(b))).
\end{equation*}

Define the following operators:
\begin{enumerate}[(a)]
    \item $L_a^\l:\UUU(R_L)\rightarrow \UUU(R_L)[\l]$, for every $a\in R_L$, is
    \begin{equation*}
        L_a^\l(x):=\psi^{-1}\left(\sum\nolimits^\l_{-\del-\h}\br[z]{\psi(x)}{a}\d z\right);
    \end{equation*}
    \item  let $x\in \UUU(R_L)$ and $a=a_1\otimes\dots\otimes a_n$, with $a_i\in R$ for all $i$; then $L_\l:\UUU(R_L)\otimes \Tt(R)\rightarrow \UUU(R_L)[\l]$ is defined as
\begin{equation}\label{eq:operatorLnewdef}
    L_\l(x,a):=\sum_{k=0}^n\, \sum_{1\leq i_1<\dots<i_k\leq n}(a_1\cdots a_n)_{i_1,\dots,i_k} L^\l_{a_{i_k}}\circ\cdots\circ L^\l_{a_{i_1}}(x),
\end{equation}
where the subscript $i_1,\dots,i_k$ means that the terms $a_{i_1},\dots,a_{i_k}$ are removed from the product.
\end{enumerate}

\begin{prop}\label{prop S well defined}
    The operator $L_\l$ induces a well-defined operator (that we still denote by $L_\l$)
    \begin{equation*}
        L_\l:\UUU(R_L)\otimes \UUU(R_L)\rightarrow \UUU(R_L)[\l].
    \end{equation*}
    Moreover, for all $a,b\in \UUU(R_L)$,
    \begin{equation*}
        \hat{I}_{\l,\h}(a,b)=L_\l(a,b).
    \end{equation*}
\end{prop}
\begin{proof}
    We want to prove that $L_\l$ satisfies the following recursive relations, for all $x\in \UUU(R_L)$, $a_0\in R$ and $B\in\Tt(R)$:
   \begin{enumerate}[(i)]
        \item $L_\l(x,1)=x$;
       \item $L_\l(x, (a_0\otimes B))=a_0L_\l(x, B)+L_\l(L^\l_{a_0}(x),B)$.
   \end{enumerate}
   The base case is true by definition. Let now $x\in \UUU(R_L)$, $a_0\in R$ and $B\in\Tt(R)$, with $B=a_1\otimes\dots\otimes a_n$. Then 
   \begin{equation*}
       L_\l(x,a_0\otimes B)=\sum_{k=0}^{n+1}\, \sum_{0\leq i_1<\dots<i_k\leq n} (a_0a_1\cdots a_n)_{i_1,\dots,i_k}  L^\l_{a_{i_k}}\dots L^\l_{a_{i_1}}(x).
   \end{equation*}
   We can split the inner sum into the two cases $i_1\neq 0$ and $i_1=0$. In the first case, we can take $a_0$ out of the sum, getting $a_0L_\l(x,B)$. In the case $i_1=0$ we have instead
   \begin{equation*}
       \sum_{k=1}^{n+1}\, \sum_{1\leq i_2<\dots<i_{k}\leq n}(a_1\cdots a_n)_{i_2,\dots,i_k} L^\l_{a_{i_k}}\dots L^\l_{a_{i_2}}(L^\l_{a_0}(x))=L_\l(L^\l_{a_0}(x),B).
   \end{equation*}
   On the other hand, from Lemma \ref{lemma recursion integral h bracket}, for all $x,B\in \UUU(R_L)$ and $a\in R$,
    \begin{enumerate}[(i)]
       \item $\hat{I}_{\l,\h}(x,1)=x$,
       \item $\hat{I}_{\l,\h}(x, aB)=a\hat{I}_{\l,\h}(x, B)+\hat{I}_{\l,\h}(L^\l_a(x),B)$.
   \end{enumerate}
   It follows that $L_\l(x,a_1\otimes\cdots\otimes a_n)=\hat{I}_{\l,\h}(x,a_1\cdots a_n)$, for all $x\in \UUU(R_L)$, $a_i\in R$. Since $L_\l$ depends only on the projection of the second factor to $\UUU(R_L)$, it induces a well-defined operator on $\UUU(R_L)\otimes\UUU(R_L)$.
\end{proof}

\begin{cor}\label{cor commutation relation S and D}
     For all $a,b\in R$, 
    \begin{equation}\label{eq Lie action S}
        L^\l_bL_a^\l-L_a^\l L_b^\l=L^\l_{[a,b]}.
    \end{equation}
    In particular, $L^\l_{(-)}$ defines an action of the opposite Lie algebra $R_L^{op}$ on $\UUU(R_L)$.
\end{cor}
\begin{proof}
    Let $a,b\in R$ and $x\in \UUU(R_L)$. By direct computation:
    \begin{equation*}
        \begin{aligned}
            0&=L_\l(x,ab-ba-[a,b])\\
            &=abx+aL^\l_b(x)+bL^\l_a(x)+L^\l_b(x)L^\l_a(x)-bax-bL^\l_a(x)-aL^\l_b(x)-L^\l_a(x)L^\l_b(x)-[a,b]x-L^\l_{[a,b]} \\
            &=L^\l_bL_a^\l-L_a^\l L_b^\l-L^\l_{[a,b]}.
        \end{aligned}
    \end{equation*}
\end{proof}

For what follows, it is convenient to rewrite formula \eqref{eq:operatorLnewdef} in a different way. We embed $\Tt(R)$ in $\Tt(R\oplus\CC)$, and define operators $\frac{\delta}{\d i}:\Tt(R\oplus\CC)\rightarrow\Tt(R\oplus\CC)$, as
\begin{equation}
    \frac{\d}{\d i}(a_1\otimes \cdots \otimes a_n)=\begin{cases}
        0 & \text{if } i>n \\
        a_1\otimes\cdots\x a_{i-1}\otimes 1 \otimes a_{i+1}\otimes\cdots\x a_n & \text{if }a_i\in R \\
        0 & \text{if } a_i\in\CC
    \end{cases}
\end{equation}
It is easy to check that, for all $a,x\in\UUU(R_L)$,
\begin{equation}\label{eq S operator}
    \begin{aligned}
        \hat{I}_{\l,\h}(x,a)&=\sum_{k=0}^n\, \sum_{1\leq i_1<\dots<i_k\leq n}\frac{\delta^k a}{\delta{i_1}\dots \delta{i_k}}  L^\l_{a_{i_k}}\circ\dots\circ L^\l_{a_{i_1}}(x).
    \end{aligned}
\end{equation}
\begin{oss}\label{oss:abuseofnotation}
    There is an abuse of notation in formula \eqref{eq S operator}, because the operators $\frac{\d}{\d i}$ are defined on the tensor algebra. What we mean is ``take a lift of $a$ to the tensor algebra, apply $\frac{\d}{\d i}$ and project back to $\UUU(R_L)$''. By Proposition \ref{prop S well defined}, we know that $\hat{I}_{\l,\h}$ does not depend on the choice of the lift.
\end{oss}
Notice that $(\frac{\d}{\d i})^n=0$ for all $i$ and $n\geq2$. In particular,
\begin{equation}
    \exp(\sum_{i=1}^n\frac{\d}{\d i})=\sum_{k=0}^\infty \frac{1}{k!}\left(\frac{\d}{\d 1}+\cdots+\frac{\d}{\d n}\right)^k=\sum_{k=0}^\infty \frac{1}{k!}\sum_{1\leq i_1<\dots<i_k\leq n}k!\frac{\d^k}{\d i_1\cdots\d i_k}.
\end{equation}
If the operators $L^\l_{a_i}$ were commutative, \eqref{eq S operator} could be rewritten in exponential form. 

\paragraph{Notation} Consider the free algebra generated by an ordered set of elements $\{a_i\}_{i\in I}$. The ``normal order'' of a monomial $a=a_{i_1}\cdots a_{i_k}$ is 
\begin{equation}
    {:}a{:}={:}a_{i_1}\cdots a_{i_k}{:}=a_{i_\s(1)}\cdots a_{i_\s(k)},
\end{equation}
where $\s$ is the permutation such that $i_{\s(1)}\geq i_{\s(2)}\geq\dots\geq i_{\s(k)}$. On linear combinations of monomials, the normal order is extended by linearity. We also use the notation ${:}\exp\!:$, which means that the powers in the expansion of the exponential are normally ordered. That is, for any $P$ in the free algebra,
\begin{equation}
    {:}\exp\!:(P)=1+{:}P{:}+\frac12{:}PP{:}+\frac16{:}P P P{:}+\cdots
\end{equation}

If we consider the ordered set $\{\frac{\d}{\d i}\x L_{a_i}\}_{i\geq 1}$, then \eqref{eq S operator} becomes
\begin{equation}
    \hat{I}_{\l,\h}(x,a)=m\circ{:}\exp:\left(\sum_{i\geq 1}\frac{\d}{\d i}\x L^\l_{a_i} \right)(a\x x),
\end{equation}
where $m$ is the associative multiplication.
\begin{prop}
    Let $a\in\UUU(R_L)$ and suppose that there are distinct elements $a_1,\dots,a_s\in R$ such that $a=\pi(a_1^{\otimes n_1}\otimes\cdots\otimes a_s^{\otimes n_s})$, for some $n_1,\dots,n_s\geq1$.
    Then, for all $x\in\UUU(R_L)$, 
    \begin{equation}\label{eq:lastminutederivative}
        \begin{aligned}
            \hat{I}_{\l,\h}(x,a)&=\sum_{k=0}^{n}\frac{1}{k!} \sum_{\substack{        l_1+\cdots+l_s=k}}\binom{k}{l_1\cdots l_s}\frac{\del^k a}{\del^{l_1} a_{1}\cdots \del^{l_s} a_{s}} (L^\l_{a_{s}})^{l_s}\cdots (L^\l_{a_{1}})^{l_1}(x)\\
            &= m\circ {:}\exp\!:\left(\sum_{i=1}^s\frac{\del}{\del a_i}\x L^\l_{a_i} \right)(a\x x).
        \end{aligned}
    \end{equation}
    where $n=n_1+\cdots+n_s$ and the normal order is given by $1<2<\dots<s$. 
\end{prop}
Again, there is a slight abuse of notation in \eqref{eq:lastminutederivative}, as the derivatives are computed on the lift $a_1^{\otimes n_1}\otimes\cdots\otimes a_s^{\otimes n_s}$ of $a$ to the tensor algebra.
\begin{proof}
    Notice that, by Leibniz rule, 
    $$\pdv{}{a_1}()(a_1^{\otimes n_1}\otimes\cdots\otimes a_s^{\otimes n_s})=\left(\frac{\d}{\d1}+\frac{\d}{\d2}+\cdots+\frac{\d}{\d n_1}\right)(a_1^{\otimes n_1}\otimes\cdots\otimes a_s^{\otimes n_s}),$$
    and similarly for all the other partial derivatives. Hence,
    $$\sum_{i\geq 1}\frac{\d}{\d i}\x L^\l_{a_i}=\sum_{j=1}^s\pdv{}{a_i}\x L_{a_i}^\l. $$
    This implies that
    \begin{align*}
    \hat{I}_{\l,\h}(x,a)&=m\circ{:}\exp\!:\left(\sum_{i\geq 1}\frac{\d}{\d i}\x L^\l_{a_i} \right)(a\x x)=m\circ {:}\exp\!:\left(\sum_{i=1}^s\frac{\del}{\del a_i}\x L^\l_{a_i} \right)(a\x x).
    \end{align*}
    Notice that this is only true because we considered a lift of the form $a_1^{\otimes n_1}\otimes\cdots\otimes a_s^{\otimes n_s}$, with $a_i$ distinct. This ensures that the two normal orders are compatible with each other.
\end{proof}

\begin{oss}
    If we take the convention of writing $b$ in the PBW basis, formula \eqref{eq:lastminutederivative} becomes, for $a,b\in \UUU(R_L)$,
    \begin{equation}\label{eq S operator PBW basis}
        \begin{aligned}
            \hat{I}_{\l,\h}(a,b)&=m\circ{:}\exp\!:\left(\sum_{i\in\Ii}\del_{u_i}\otimes L_i^\l\right)(b\otimes a) \\
            &=\sum_{k=0}^\infty\, \frac{1}{k!}\sum_{\substack{i_1<\dots<i_s\in\Ii \\
        l_{i_1}+\cdots+l_{i_s}=k}}\binom{k}{l_{i_1}\cdots l_{i_s}}\frac{\del^k b}{\del^{l_{i_1}} u_{i_1}\dots \del^{l_{i_s}} u_{i_s}}  (L^\l_{i_k})^{l_{i_k}}\dots (L^\l_{{i_1}})^{l_{i_1}}(a), 
        \end{aligned}
    \end{equation}
    where $L^\l_i:=L^\l_{u_i}$. When applied to $b\otimes a$, only finite terms of the infinite sum inside the exponential are nonzero.
\end{oss}

If we compute $\hat{I}_{\l,\h}(a,b)$ for $a,b\in \UUU(R_L)$ using \eqref{eq S operator PBW basis}, we obtain an expression that depends on $L^\l_i(a)$, with $i\in \Ii$. But notice that, by definition of the $\h$-sum bracket and skew-symmetry,  
    \begin{equation}\label{eq operatori Sl}
        \begin{aligned}
            L^\l_i(x)&=\hat{I}_{\l,\h}(x,u_i)-\hat{I}_{-\del-\h,\h}(x,u_i)\\
            &=\hat{I}_{-\l-\del-\h,\h}(u_i,x)-\hat{I}_0(u_i,x)=\hat{I}_{-\l-\del-\h,\h}(u_i,x)-u_i x.
        \end{aligned}
    \end{equation}
    Computing $\hat{I}_{-\l-\del-\h,\h}(u_i,x)$ explicitly, we get    \begin{equation}
        \hat{I}_{-\l-\del-\h,\h}(u_i,x)=m\circ{:}\exp\!:\left(\sum_{j\in\Ii}\del_{u_j}\x L^{-\l-\del-\h}_{j}\right)(x\x u_i),
    \end{equation} 
    which only depends on the $\h$-brackets of elements of $\Bb$:
    \begin{equation*}
        L^{-\l-\del-\h}_j(u_i)=\sum\nolimits^{-\l-\del-\h}_{-\del-\h}\br[x]{u_j \,}{u_i}\d x.
    \end{equation*}
    So each $L^\l_{{i_1}}$ is a (rather complicated) differential operator, defined by \eqref{eq operatori Sl}. Hence, by \eqref{eq S operator PBW basis}, $\hat{I}_{\l,\h}(a,b)$ is a ``normally ordered exponential'' of a bidifferential operator, with coefficients that depend only on the $\h$-bracket within the basis elements of $R$. 

\begin{defi}
    A vertex algebra $V$ is a ``free-field'' vertex algebra if $V=V(R)$, and the $\l$-bracket on $R$ takes values in $\CC[\l]$. Many algebras fall into this class, for example, the $\beta\gamma$-systems and free boson vertex algebras (see Example \ref{es:sublinearvertexliealgebras}).
\end{defi}

\begin{prop}\label{prop cap bracket free field case}
    If $V(R)$ is a free-field vertex algebra, formula \eqref{eq S operator PBW basis} reduces to
    \begin{equation}\label{eq cap bracket free field case}
        \begin{aligned}
            \hat{I}_{\l,\h}(a,b)&=m\circ\s\circ\exp(\sum_{i\in\Ii}L_i^\l\otimes \del_{u_i})(a\otimes b),
        \end{aligned}
    \end{equation}
    where $\sigma$ is the operator $x\otimes y\mapsto y\otimes x$, and
    \begin{equation}
        L_i^\l(a)=\sum_{j\in\Ii}\sum\nolimits^{\l}_{-\del-\h}\pdv{a}{u_j}\br[x+\del^{(1)}]{u_j \,}{u_i}\, \d x .
    \end{equation}
\end{prop}
\begin{proof}
  Note that $[a,b]=\sum\nolimits^0_{-\del-\h}\br{a}{b} \d \l=\sum\nolimits^0_{-\h}\br{a}{b} \d \l\in\CC[\l] $, so $L^\l_{[a,b]}=0$. From Corollary \ref{cor commutation relation S and D}, the operators $L^\l_a$, for $a\in R$, commute with each other. We can thus rewrite \eqref{eq S operator PBW basis} as
    \begin{equation}
        \begin{aligned}
            \hat{I}_{\l,\h}(a,b)&=m\circ\s\circ\exp(\sum_{i\in\Ii}L_i^\l\otimes \del_{u_i})(a\otimes b).
        \end{aligned}
    \end{equation}
    By the left Wick formula \eqref{eq left wick}, for all $x,y\in R$ and $Y\in V(R)$,
    \begin{equation*}
        \br{x\ast_\h Y}{\, y}=x\br[\l+\del^{(1)}]{Y}{y}+Y\br[\l+\del^{(1)}]{x}{y}+\sum\nolimits_0^\l \br[\mu]{Y}{\br[\l-\mu-\h]{x}{y}} \delta \mu.
    \end{equation*}
    The definite sum is $0$ because $\br{x}{y}\in \CC[\l]$. Thus, the $\h$-bracket $\br{x\ast_\h Y}{\, y}$ can be expanded by the left Leibniz rule. This implies that
    \begin{equation*}
        L^\l_i(a)=\psi^{-1}\sum\nolimits^\l_{-\del-\h}\br[x]{\psi(a)}{u_i}\d x=\sum_{j\in\Ii}\sum\nolimits^{\l}_{-\del-\h}\pdv{a}{u_j}\br[x+\del^{(1)}]{u_j \,}{u_i}\, \d x.
    \end{equation*}
\end{proof}
We consider now the whole quantization. Take the chiral star-product $I_{\l,\h,\star}$ induced by $\phi$, that is, for $a,b\in\Vv$,
\begin{equation}\label{def sum h star bracket}
    I_{\l,\h,\star}(a,b):=\phi^{-1}(I_{\l,\h}(\phi(a),\phi(b))).
\end{equation}

\begin{oss}\label{oss symmetric part of the quantization}
    As defined in \eqref{diagram factor quantization map}, the quantization map $\phi$ is equal to $\ga\circ\psi$. Thus, for all $a,b\in \Ss(R)$:
    \begin{align*}
        I_{\l,\h,\star}(a,b)&=\ga^{-1}(\hat{I}_{\l,\h}(\ga(a),\ga(b)))=\ga^{-1}\circ m\circ {:}\exp\!:\left(\sum_{i\in\Ii}\del_{u_i}\otimes L_i^\l\right)(\ga(b)\otimes \ga(a)).
    \end{align*}
    So, if we can find some $x,y\in\Ss(R)$ such that
    \begin{equation*}
    {:}\exp\!:\left(\sum_{i\in\Ii}\del_{u_i}\otimes L_i^\l\right)(\ga(b)\otimes \ga(a))=\ga(x)\otimes\ga(y),
    \end{equation*}
    then $I_{\l,\h,\star}(a,b)=\ga^{-1}(\ga(x)\ga(y))=x\star_\ga y$, where $\star_\ga$ is the Gutt star-product defined in Example \ref{es gutt star-product}.
\end{oss}

\begin{theorem}\label{teo sum h star bracket general case}
    For all $a,b\in\Ss(R)$, the chiral star-product $I_{\l,\h,\star}$ has the following expression:
    \begin{equation}\label{eq sum h star bracket general case}
       \begin{aligned}
            I_{\l,\h,\star}(a,b)&=m_{\star_\ga}\circ\exp(\sum_{i\in\Ii}\del_{u_i}\x (D_i^{-\l-\del-\h}-\Ll^\star_{i}))(b\x a)
            \\
            =&\sum_{k=0}^\infty\frac{1}{k!} \sum_{\substack{i_1\dots i_k\in\Ii}}\frac{\del^k b}{\del u_{i_1}\dots \del u_{i_k}}\star_\ga  (D^{-\l-\del-\h}_{i_k}-\Ll^\star_{i_k})\cdots (D^{-\l-\del-\h}_{{i_1}}-\Ll_{i_1}^\star)(a),
       \end{aligned}
    \end{equation}
    where $\star_\gamma$ denotes the Gutt star-product, $\Ll_i^\star$ the left multiplication operator $u_i\star_\ga-$, and $D^\l_j$ for $j\in\Ii$ is the operator
    \begin{equation}        
    \begin{aligned}
        D^\l_j(a):=\sum_{k=0}^\infty\, \frac{1}{k!}\sum_{\substack{i_1\dots i_k\in\Ii}}&\left(\frac{\del^k a}{\del u_{i_1}\dots \del u_{i_k}}\right)\star_\ga\left( \sum\nolimits^\l_{-\del-\h} \d \l_1\sum\nolimits^{\l_1+\h}_{-\del-\h} \d \l_2\dots \right. \\
        &\left.\dots\sum\nolimits^{\l_{k-1}+\h}_{-\del-\h} \d \l_k \{\cdots\{\{u_{j \ \l_1} u_{i_1}\}_{\h \,\l_2} u_{i_2}\}_\h\dots _{\l_k}u_{i_k}\}_\h \right).
    \end{aligned}
    \end{equation}
\end{theorem}
\begin{proof}
    It is sufficient to prove \eqref{eq sum h star bracket general case} when $a$ and $b$ are PBW monomials, say of degree $m$ and $n$ respectively. Denote $b=b_1\cdots b_n$, with $b_i\in \Bb$, and
    \begin{equation*}
         \ga(b)=\frac{1}{n!}\sum_{\s\in S_n}b_\s,
    \end{equation*}
    with $b_\s:=b_{\s(1)}\cdots b_{\s(n)}\in \UUU(R_L)$. Then
    \begin{equation*}
        I_{\l,\h,\star}(a,b)=\ga^{-1}\left(\frac{1}{n!}\sum_{\s\in S_n}\hat{I}_{\l,\h}(\ga(a),b_\s) \right).
    \end{equation*}
    We compute $\hat{I}_{\l,\h}(\ga(a),b_\s)$ using \eqref{eq S operator}:
    \begin{equation*}
        \hat{I}_{\l,\h}(\ga(a),b_\s)=\sum_{k=0}^\infty\, \sum_{\substack{1\leq i_1<\dots<i_k\leq n}}\frac{\delta^k b_\s}{\delta {i_1}\dots \delta {i_k}}  L^\l_{b_{\s(i_k)}}\circ\dots\circ L^\l_{b_{\s(i_1)}}(\ga(a)).
    \end{equation*}
    Thus
    \begin{align*}
        I_{\l,\h,\star}(a,b)&=\ga^{-1}\sum_{k=0}^\infty\!\!\sum_{\substack{i_1\neq\cdots\neq i_k \\
        i_1,\dots,i_k\in\{1,\dots,n\}}}\left(\frac{1}{n!}\!\!\!\!\!\sum_{\substack{\s\in S_n \\ \s(i_1)<\dots<\s(i_k)} }\!\!\!\frac{\delta^k b_{\s^{-1}}}{\delta \s(i_1)\dots \delta {\s(i_k)}}\right) L^\l_{b_{i_k}}\circ\dots\circ L^\l_{b_{i_1}}(\ga(a)).
    \end{align*}
    Notice that 
    \begin{equation*}
        \frac{1}{n!}\!\!\!\!\!\sum_{\substack{\s\in S_n \\ \s(i_1)<\dots<\s(i_k)} }\!\!\!\frac{\delta^k b_{\s^{-1}}}{\delta \s(i_1)\dots \delta {\s(i_k)}}=\frac{1}{k!}\ga\left(\frac{\delta^k b}{\delta {i_1}\dots \delta {i_k}}\right),
    \end{equation*}
    so we get 
    \begin{equation}\label{eq proof sum h star bracket 0}
    \begin{aligned}
        I_{\l,\h,\star}(a,b)&= \ga^{-1}\sum_{k=0}^\infty\, \frac{1}{k!}\!\!\!\sum_{\substack{i_1\neq\cdots\neq i_k \\
        i_1,\dots,i_k\in\{1,\dots,n\}}}\!\!\!\ga\left(\frac{\d^k b}{\d {i_1}\dots \d {i_k}}\right) L^\l_{b_{i_k}}\circ\dots\circ L^\l_{b_{i_1}}(\ga(a))
        \\
        &=\gamma^{-1}\circ(\ga\x1)\circ\exp(\sum_{i=1}^n \frac{\d}{\d i}\x L_{b_i}^\l)(b\x\ga(a))
        \\        &=\gamma^{-1}\circ(\ga\x1)\circ\exp(\sum_{i\in\Ii} \frac{\del}{\del u_i}\x L_{i}^\l)(b\x\ga(a))
        \\
        &=\ga^{-1}\sum_{k=0}^\infty\, \frac{1}{k!}\sum_{\substack{i_1\dots i_k\in\Ii}}\ga\left(\frac{\del^k b}{\del u_{i_1}\dots \del u_{i_k}}\right) L^\l_{{i_k}}\circ\dots\circ L^\l_{{i_1}}(\ga(a)).
    \end{aligned}
    \end{equation}
    The exponentials here are the usual exponentials of non-commutative variables.\\ By \eqref{eq operatori Sl}, for every $x\in\Ss(R)$ and $j\in\Ii$,
    \begin{equation*}
        \ga^{-1}L^\l_{j}(\ga(x))=I_{-\l-\del-\h,\h,\star}(u_j,x)-u_j\star_\ga x.
    \end{equation*}
    As explained in Remark \ref{oss symmetric part of the quantization}, the only thing left to prove is that $I_{\l,\h,\star}(u_j,a)=D^\l_j(a)$. By \eqref{eq proof sum h star bracket 0} it follows that
    \begin{equation*}
        \hat{I}_{\l,\h}(u_j,\ga(a))=\sum_{k=0}^\infty\,\frac{1}{k!} \sum_{\substack{i_1\dots i_k\in\Ii}}\ga\left(\frac{\del^k a}{\del u_{i_1}\dots \del u_{i_k}}\right)  L^\l_{i_k}\dots L^\l_{i_1}(u_j),
    \end{equation*}
    where $L^\l_{i_k}\dots L^\l_{i_1}(u_j)\in R[\l]$. So, by Remark \ref{oss symmetric part of the quantization},
    \begin{equation}\label{eq proof sum h star bracket 1}
        I_{\l,\h,\star}(u_j,a)=\sum_{k=0}^\infty\,\frac{1}{k!} \sum_{\substack{i_1\dots i_k\in\Ii}}\frac{\del^k a}{\del u_{i_1}\dots \del u_{i_k}} \star_\ga L^\l_{i_k}\dots L^\l_{i_1}(u_j).
    \end{equation}
    The $\h$-bracket of elements of $R$ is the same in $\Ss(R)$ and $V(R)$, so, for all $i\in\Ii$
    \begin{equation*}
        L^\l_{i}(u_j)=\psi^{-1}\sum\nolimits^\l_{-\del-\h}\br[x]{u_j \,}{u_{i}}\d x=\sum\nolimits^\l_{-\del-\h}\pbr[x]{u_j \,}{u_{i}}\d x .
    \end{equation*}
    We can expand the composition of two operators $L^\l_{i_2}$ and $L^\l_{i_1}$, using the sesquilinearity property of the sum $\h$-bracket and the change of order of summation (Proposition \ref{prop change order summation}):
    \begin{equation}\label{eq proof sum h star bracket 2}
        \begin{aligned}
            L^\l_{i_2}L^\l_{i_1}(u_j)&=\sum\nolimits^\l_{-\del-\h}\d\l_1\pbr[\l_1]{\left(\sum\nolimits^\l_{-\del-\h}\d\l_2\pbr[\l_2]{u_j \,}{u_{i_1}}\right)  \,}{u_{i_2}} \\
            &=\sum\nolimits^\l_{-\del-\h}\d\l_1\sum\nolimits^\l_{\l_1}\d\l_2\pbr[\l_1]{\pbr[\l_2]{u_j \,}{u_{i_1}}  \,}{u_{i_2}} \\
            &=\sum\nolimits^\l_{-\del-\h}\d\l_2\sum\nolimits^{\l_1+\h}_{-\del-\h}\d\l_1\pbr[\l_1]{\pbr[\l_2]{u_j \,}{u_{i_1}}  \,}{u_{i_2}}.
        \end{aligned}
    \end{equation}
    Putting together \eqref{eq proof sum h star bracket 1} and \eqref{eq proof sum h star bracket 2} we get $I_{\l,\h,\star}(u_j,a)=D^\l_j(a)$, completing the proof.
\end{proof}

\begin{cor}\label{cor star sum h bracket caso free field}
    If $V(R)$ is a free-field vertex algebra, then 
    \begin{equation}\label{eq star sum h bracket caso free field}
        I_{\l,\h,\star}(a,b)=m_{\star_\ga}\circ\sigma\circ\exp(\sum_{i\in\Ii}L_i^\l\otimes \del_{u_i})(a\otimes b),
    \end{equation}
    where $\sigma$ is the operator $x\otimes y\mapsto y\otimes x$, and
    \begin{equation}\label{eq Sl operator caso free field}
        L_i^\l(a)=\sum_{j\in\Ii}\sum\nolimits^{\l}_{-\del-\h}\{u_j\,_{x+\del} \, u_i\underrightarrow{\}_\h} \pdv{a}{u_j}\, \d x.
    \end{equation}
   More explicitly:
    \begin{equation}\label{eq star sum h bracket caso free field esplicito}
        \begin{aligned}
&I_{\lambda,\h,\star}(a,b)=\sum_{k=0}^\infty \ \
\sum_{\substack{i_1,j_1,\dots,i_k,j_k\in \Ii}}
\frac{\partial^k b}{\partial {u_{j_1}}\dots\partial {u_{j_k}}}\star_\ga
\\
&\Big(\sum\nolimits_{\l_{k+1}-\h}^{\lambda}\d\lambda_1\sum\nolimits_{\l_{k+1}-\h}^{\lambda_1}\d\lambda_2\dots\sum\nolimits_{\l_{k+1}-\h}^{\lambda_{k-1}}\d{\lambda_k}\prod_{s=1}^k\pbr[\l_s-\l_{s+1}]{u_{i_s}}{u_{j_{s}}}\Big)\Big{(}\restr{}{\l_{k+1}=-\del}
\frac{\partial^k a}{\partial {u_{i_1}}\dots\partial {u_{i_k}}}\Big)
\,.
\end{aligned}
    \end{equation}
\end{cor}
\begin{proof}
    Equations \eqref{eq star sum h bracket caso free field} and \eqref{eq Sl operator caso free field} follow from putting together Theorem \ref{teo sum h star bracket general case} and Proposition \ref{prop cap bracket free field case}. Formula \eqref{eq star sum h bracket caso free field esplicito} follows from iterating \eqref{eq Sl operator caso free field} and by the fact that $\pbr{u_i \,}{u_j}\in\CC[\l] $ for all $i,j\in\Ii$.
\end{proof}

\begin{oss}
    Formulae \eqref{eq sum h star bracket general case} and \eqref{eq star sum h bracket caso free field} for the free-field case consists of an exponential of a bidifferential operator, which is similar to the Moyal-Weyl and Gutt star-products (see Examples \ref{es Moyal Weyl} and \ref{es gutt star-product}).
\end{oss}

\begin{es}
    In the case of $\beta\gamma$-systems, we can see explicitly how the chiral star-product $I_{\l,\h,\star}$ reduces to the Moyal-Weyl star-product in the Zhu algebra. First of all, notice that the Gutt star-product $\star_\gamma$ in the Zhu algebra reduces to the Moyal-Weyl star-product. Denote by $\pi$ the Poisson bivector of the Zhu algebra and let $a\star_\h b:=I_{0,\h,\star}(a,b)$; from equation \eqref{eq star sum h bracket caso free field} and by expanding the formula for the Moyal-Weyl star-product  \eqref{eq moyal weyl star-product}, it follows that
    \begin{equation*}
        a\star_\h b\equiv m\circ\exp(\pi/2)\circ\sigma\circ\exp(\sum_{i\in\Ii}L_i^0\otimes \del_{u_i})(a\otimes b) \mod \Jj_\h.
    \end{equation*}
    Moreover, by \eqref{eq Sl operator caso free field} and Theorem \ref{teo definite sum equal sum},
    \begin{equation*}
        \begin{aligned}
            L^0_i(a)&\equiv\sum_{j\in\Ii}\pdv{a}{u_j}\sum\nolimits^{0}_{-\h}\pbr[x]{u_j \,}{u_i}\, \d x \mod \Jj_\h \\
            &=\sum_{j\in\Ii}\pdv{a}{u_j}\pbr[-\h]{u_j \,}{u_i} \mod \Jj_\h.
        \end{aligned}
    \end{equation*}
    Since the Poisson bracket in the Zhu algebra is induced by $\pbr[-\h]{\cdot}{\cdot}$,
    \begin{align*}
        a\star_\h b&\equiv m\circ\exp(\pi/2)\circ\sigma\circ\exp(\pi)(a\otimes b) \mod \Jj_\h \\
        &=m\circ\sigma\circ\exp(\pi/2)(a\otimes b) \mod \Jj_\h.
    \end{align*}
    Since $m$ is commutative, we can get rid of $\s$, getting back the formula for the Moyal-Weyl star-product.
\end{es}

\subsection{Star-deformations for Poisson vertex algebras: explicit formulae}\label{sec:explicitstar}

A (Poisson) vertex algebra is an $\h$-(Poisson) vertex algebra under the limit $\h\rightarrow0$. So, our formulae for $I_{\l,\h,\star}$ can be specialized to obtain formulae for the star-deformation of the corresponding Poisson vertex algebras. 

\begin{theorem}\label{teo star integral l bracket}
    The following formulae describe the star-deformation of $\Ss(R)$ induced by the symmetrization map:
    \begin{equation}\label{eq:stardeformationgeneral}
        \begin{aligned}
            I_{\l,\star}(a,b)&=m_{\star_\ga}\circ\exp(\sum_{i\in\Ii}\del_{u_i}\x (D_i^{-\l-\del}-\Ll^\star_{i}))(b\x a)
            \\
            &=\sum_{k=0}^\infty\frac{1}{k!} \sum_{\substack{i_1\dots i_k\in\Ii}}\frac{\del^k b}{\del u_{i_1}\dots \del u_{i_k}}\star_\ga  (D^{\l-\del}_{i_k}-\Ll^\star_{i_k})\dots (D^{\l-\del}_{{i_1}}-\Ll_{i_1}^\star)(a),
        \end{aligned}
    \end{equation}
    for all $a,b\in\Ss(R)$, where $\star_\gamma$ denotes the Gutt star-product \eqref{eq BCH Gutt star-product}, $\Ll_i^\star$ the left multiplication operator by $u_i$ with respect to $\star_\ga$, and $D^\l_j$ for $j\in\Ii$ is the operator:
    \begin{equation}        
    \begin{aligned}
        D^\l_j(a):=\sum_{k=0}^\infty\, \frac{1}{k!}\sum_{\substack{i_1\dots i_k\in\Ii}}&\left(\frac{\del^k a}{\del u_{i_1}\dots \del u_{i_k}}\right)\star_\ga\left( \int^\l_{-\del} d \l_1\int^{\l_1}_{-\del} d \l_2\dots \right. \\
        &\left.\dots\int^{\l_{k-1}}_{-\del} d \l_k \{\cdots\{\{u_{j \ \l_1} u_{i_1}\}_{\l_2} u_{i_2}\}\dots _{\l_k}u_{i_k}\} \right).
    \end{aligned}
    \end{equation}
\end{theorem}
\begin{proof}
    It follows directly from Theorem \ref{teo sum h star bracket general case}, by sending $\h\rightarrow0$.
\end{proof}

\begin{cor}\label{cor:stardeformationfreefield}
    If $R_L$ is an abelian Lie algebra, then
    \begin{equation}\label{eq star integral l bracket free field case}
        I_{\l,\star}(a,b)=m\circ\exp(\sum_{i\in\Ii}L_i^\l\otimes \del_{u_i})(a\otimes b),
    \end{equation}
    where $m$ is the multiplication map and
    \begin{equation}
        L_i^\l(a)=\sum_{j\in\Ii}\int^{\l}_{-\del}\pdv{a}{u_j}\pbr[x+\del^{(1)}][]{u_j \,}{u_i}\, d x.
    \end{equation}
   More explicitly:
    \begin{equation}
        \begin{aligned}
&I_{\lambda,\star}(a,b)=\sum_{k=0}^\infty \ \
\sum_{\substack{i_1,j_1,\dots,i_k,j_k\in \Ii}}
\frac{\partial^k b}{\partial {u_{j_1}}\dots\partial {u_{j_k}}}\cdot
\\
&\left(\int_{\l_{k+1}}^{\lambda}d\lambda_1\int_{\l_{k+1}}^{\lambda_1}d\lambda_2\dots\int_{\l_{k+1}}^{\lambda_{k-1}}d{\lambda_k}\prod_{s=1}^k\{{u_{i_s}}\,_{\lambda_s-\lambda_{s+1}}u_{j_s}\}\right)\left(\restr{}{\l_{k+1}=-\del}
\frac{\partial^k a}{\partial {u_{i_1}}\dots\partial {u_{i_k}}}\right)
\,.
\end{aligned}
    \end{equation}
\end{cor}
\begin{proof}
     If $R_L$ is abelian, then the map $\ga$ in \eqref{diagram factor quantization map} is the identity. This implies that $I_{\l,\star}=\hat{I}_{\l,\h |_{\h\rightarrow0}}$. All the properties used in the proof of Proposition \ref{prop cap bracket free field case} are true in the case where $R_L$ is abelian. So  \eqref{eq star integral l bracket free field case} follows from \eqref{eq cap bracket free field case}.
\end{proof}

\begin{oss}
    If $V(R)$ is a free field vertex algebra, then $R_L$ is an abelian Lie algebra. This is not true if $\h\neq 0$, because
    \begin{equation*}
        [a,b]=\sum\nolimits^0_{-\del-\h}\br[x]{a}{b}\d x=\sum\nolimits^0_{-\h}\br[x]{a}{b}\d x\neq 0.
    \end{equation*}
\end{oss}

\begin{oss}
    In \cite{barakatPoissonVertexAlgebras2009}, the authors derive an explicit, closed formula for a Poisson vertex algebra structure on an algebra of differential polynomials, which they call the ``Master Formula''. The formulae in Theorem \ref{teo star integral l bracket} are a quantization of the Master Formula. 
\end{oss}

\begin{appendices}

\section{Calculus of formal distributions}\label{sec:formaldistributions}

\begin{defi}

    Let $U$ be a vector space. A $U$-valued formal distribution in the variable $z_1,\dotso,z_k$ is a bilateral power series $a(z_1,\dotso,z_k)\in U[[z_1^{\pm1},\dotso,z_k^{\pm1}]] $
    \begin{equation}        a(z_1,\dotso,z_k)=\sum_{i_1\dots i_k\in\ZZ}a_{i_1,\dotso,i_k}z_1^{i_1}\dotso z_k^{i_k}.
    \end{equation} 
    For a one-variable formal distribution $a(z)\in U[[z,z^{-1}]]$, we use the notation
\begin{equation}\label{eq:formaldistributiononevariable}
    a(z)=\sum_{n\in\ZZ}a_{(n)}z^{-n-1}, \hspace{10 mm} a(z)_+=\sum_{n<0}a_{(n)}z^{-n-1}, \hspace{10 mm} a(z)_-=\sum_{n\geq0}a_{(n)}z^{-n-1}.
\end{equation}

\end{defi}

\begin{oss}
   Assume that $U,V,W$ are vector spaces with a product $U\otimes V\rightarrow W$. It is always possible to extend it formally to a product on the spaces of distributions 
   $$U[[z,z^{-1}]]\otimes V[[w,w^{-1}]]\rightarrow W[[z^{\pm1},w^{\pm1}]]. $$ If $z=w$ though, the product is in general ill-defined, as the coefficient of one of the terms $z^{i}$ may be an infinite series.  There are some cases in which the products are always defined, for example, the product of a formal distribution and a Laurent polynomial or the product of two formal Laurent series.  
\end{oss}

The formal derivative on $U[[z,z^{-1}]]$ is defined as 
\begin{equation}\label{def formal derivative}
    \del_z a(z)=-\sum_{n\in\ZZ}na_{(n-1)}z^{-n-1}.
\end{equation}
The formal residue is, by definition, the coefficient of the $z^{-1}$ term:
\begin{equation}
    \Res_z a(z)=a_{(0)}.
\end{equation}
In particular, this implies:
\begin{equation}
    a_{(n)}=\Res_zz^na(z), \quad \forall n\in\ZZ.
\end{equation}
This motivates the seemingly unnatural indexing in \eqref{eq:formaldistributiononevariable}. From \eqref{def formal derivative} we get a short exact sequence:

\begin{equation}
   \begin{tikzcd}
0 \arrow[r] & U \arrow[r] & {U[[z,z^{-1}]]} \arrow[r, "\del_z"] & {U[[z,z^{-1}]]} \arrow[r, "\Res_z"] & U \arrow[r] & 0.
    \end{tikzcd}
\end{equation}

\paragraph{Integration by parts} Let $a(z),b(z)\in A((z))$ be formal Laurent series with coefficients in an algebra $A$. The product $a(z)\cdot b(z)$ is well-defined and the formal derivative $\del_z$ acts as a derivation. Since $\Res_z(\del_z (a(z)\cdot b(z)))=0$, we must have
\begin{equation}
    \Res_z(\del_z a(z)\cdot b(z))=-\Res_z (a(z)\cdot\del_z b(z)),  
\end{equation}
which corresponds to the integration by parts formula of the analytical residue.

\paragraph{Change of variables} Let $g(w)\in U((w))$ and $f(z)\in \CC[[z]]$ such that $g(w)=\sum_{n\geq N}a_n w^n$ and $f(z)=\sum_{n\geq 1}f_n z^n$,  with $f_1\neq0$. Then the composition $g(f(z))\in U((z))$ is well-defined as $g(f(z))=\sum_{n\geq N}a_n(f(z))^n$. In fact, $f(z)^n\in z^n\CC[[z]]$ for all $n\geq0$, so $\sum_{n\geq0}a_n(f(z))^n$ converges in the topology of $U[[z]]$. The usual change of variables formula applies:
\begin{equation}\label{eq:changeofvariabelRes}
    \Res_w g(w)=\Res_z (g(f(z))\del_z f(z)).
\end{equation}
In fact, we can write $g(w)=a_{-1}w^{-1}+\del_z A(w)$ for some $A(w)\in U((w))$, so
\begin{align*}
    \Res_z (g(f(z))\del_z f(z))=a_{-1}\Res_z (f(z)^{-1}\del_z f(z))+\Res_z \del_z (A(f(z))=a_{-1},
\end{align*}
because $\Res_z(f(z)^{-1}\del_z f(z))=1$ by our assumptions on $f(z)$. 

There are natural embeddings of the spaces of Laurent series $\CC((z))((w))$ and $\CC((w))((z))$ into $\CC[[z^{\pm1},w^{\pm1}]]$. Consider now the element $(z-w)$: it is invertible in both $\CC((z))((w))$ and $\CC((w))((z))$, but the two inverses are different in $\CC[[z^{\pm1},w^{\pm1}]]$. This defines two different embeddings:
\begin{equation}\label{def espansione z>w}
        i_{z,w}:\CC((z-w))\hookrightarrow\CC((z))((w))\subset\CC[[z^{\pm1},w^{\pm1}]]  \ \ \ \frac{1}{(z-w)}\mapsto \sum_{n\geq0}w^nz^{-n-1}, 
\end{equation}
\begin{equation}\label{def espansione w>z}
    i_{w,z}:\CC((z-w))\hookrightarrow\CC((w))((z))\subset\CC[[z^{\pm1},w^{\pm1}]] \ \ \ \frac{1}{(z-w)}\mapsto -\sum_{n<0}w^nz^{-n-1}.
\end{equation}
From an analytical point of view, they correspond to series expansions in the domains $|z|>|w|$ and $|w|>|z|$, respectively.

\begin{defi}\label{def formal delta}
    The formal delta function is the two variables formal distribution $\delta(z-w)\in\CC[[z^{\pm1},w^{\pm1}]]$ defined as:
    \begin{equation}\label{eq:defformaldelta}
        \delta(z-w)=(i_{z,w}-i_{w,z})\left(
        \frac{1}{z-w}\right)=\sum_{n\in\ZZ}w^nz^{-n-1}.
    \end{equation}
\end{defi}

For all $j\in\NN$, the $j$-th derivative of the formal delta function is given by:
\begin{equation}\label{eq:derivativedelta}
    \frac{1}{j!}\del_w^{j}\delta(z-w)=(i_{z,w}-i_{w,z})\left(\frac{1}{(z-w)^{j+1}}\right)=\sum_{n\in\ZZ}\binom{n}{j}z^{-n-1}w^{n-j}.
\end{equation}

The following is a well-known property of the delta function.

\begin{prop}
    For all formal distributions $f(z)\in U[[z,z^{-1}]]$ the product \\ $f(z)\delta(z-w)$ converges and 
    \begin{equation}
        \Res_zf(z)\delta(z-w)=f(w).
    \end{equation}
\end{prop}
A fundamental property in the theory of vertex algebras is locality.

\begin{defi}
    A formal distribution in two variables $a(z,w)\in U[[z^{\pm1},w^{\pm1}]]$ is called local if there is a number $N\in\NN$ such that $(z-w)^N a(z,w)=0$.
\end{defi}
It is clear by equation \eqref{eq:derivativedelta} that $\del^n_w\delta(z-w)$ is local, with $N=n-1$. The following theorem, due to Kac, provides a complete characterization of local formal distributions.

\begin{theorem}[Decomposition theorem, {\cite[Theorem 1.2]{kacIntroductionVertexAlgebras2017}}]\label{theorem decomposition}
    For all $j\geq0$, the derivatives of the delta function $\del_w^{j}\delta(z-w)$ are local.
    Any local formal distribution $a(z,w)\in U[[z^{\pm1},w^{\pm1}]] $ can be uniquely decomposed as a finite sum of derivatives of the formal delta function, with formal distributions in $w$ as coefficients:
\begin{equation}
    a(z,w)=\sum_{j=0}^{N-1}\frac{1}{j!}c^j(w)\del_w^{(j)}\delta(z-w),
\end{equation}
where 
\begin{equation}
    c^j(w)=\Res_za(z,w)(z-w)^j\in U[[w,w^{-1}]].
\end{equation}
and $N$ is such that $(z-w)^N a(z-w)=0$.
\end{theorem}
We consider now the case where $U$ has some additional algebraic structure, namely it is a Lie algebra or an associative algebra. 

\begin{defi}
    Let $\g$ be a Lie algebra, and $a(z),b(z)$ two $\g$-valued formal distributions. Then $a(z),b(z)$ are called mutually local (or simply local) if the formal distribution $[a(z),b(w)]\in\g[[z^{\pm1},w^{\pm1}]] $ is local.
\end{defi}

\begin{defi}
    Let $A$ be an associative algebra, and $a(z),b(z)$ two $A$-valued formal distributions. The $(n)$-th product between $a(z),b(z)$ is defined, for all $n\in\ZZ$, as
    \begin{equation}\label{def n products fields}
        a(w)_{(n)}b(w)= \Res_z(i_{z,w}(z-w)^na(z)b(w)-i_{w,z}(z-w)^nb(w)a(z)).
    \end{equation}
    If $n\in\NN$, formula \eqref{def n products fields} becomes
    \begin{equation}
        a(w)_{(n)}b(w):= \Res_z(z-w)^n[a(z),b(w)].
    \end{equation}
\end{defi}

\begin{oss}
    The special case of the $(-1)$-product is denoted by
    \begin{equation}
        a(z)_{(-1)}b(z)=\, {:}a(z)b(z){:}\, ,
    \end{equation}
    and it is known as the normally ordered product. From \eqref{def espansione z>w} and
    \eqref{def espansione w>z} it follows that the normally ordered product can also be written as
    \begin{equation}\label{def normally ordered product fields}
         {:}a(z)b(z){:}\,=a(z)_+b(z)+b(z)a(z)_-.
    \end{equation}
\end{oss}

\begin{oss}\label{oss del derivation n product fields}
    We can compute $(\del_wa(w))_{(n)}b(w)$ by integrating by part the residue in \eqref{def n products fields}:
    \begin{equation}\label{eq del derivation n product fields}
        (\del_wa(w))_{(n)}b(w)=-na(w)_{(n-1)}b(w).
    \end{equation}
    By the Leibniz rule and \eqref{def n products fields} \eqref{eq del derivation n product fields},
    \begin{equation}
        \begin{aligned}
            \del_w(a(w)_{(n)}b(w))&=a(w)_{(n)}\del_wb(w)-na(w)_{(n-1)}b(w)\\ &=a(w)_{(n)}\del_wb(w)+(\del_wa(w))_{(n)}b(w),
        \end{aligned}
    \end{equation}
    hence $\del_w $ is a derivation of all $(n)$-products, for $n\in\ZZ$.
\end{oss}

\begin{theorem}[{\cite[Theorem 2.3]{kacVertexAlgebrasBeginners1998}}]\label{teo OPE}
    Let $A$ be an associative algebra, and $a(z),b(z)$ two $A$-valued formal distributions, mutually local. Then
    \begin{equation}\label{eq OPE fields}
            a(z)b(w)=\sum_{j=0}^{N-1}\frac{a(w)_{(j)}b(w)}{i_{z,w}(z-w)^{j+1}}+{:} a(z)b(z) {:} \,.
    \end{equation}
\end{theorem}

\begin{defi}
    A (quantum) field on a vector space $U$ is an $\End(U)$-valued formal distribution $a(z)$ such that, for all $b\in U$, there exists an $N\in\NN$ such that $a_{(n)}b=0$, for all $n\geq N$, i.e. $a(z)b\in U((z))$ for all $a,b\in U$.
\end{defi}

\begin{oss}
\begin{enumerate}[(i)]
    \item If $a(z)$ is a field, then by \eqref{def formal derivative} $\del_za(z)$ is still a field. 
    \item  If $a(z),b(z)$ are fields, then $a(z)b(w)v$ and $b(w)a(z)v$ are elements of the space $V((z))((w))$ and $V((w))((z))$ respectively, for all $v\in V$. By taking the residue in $z$, we get an element in $V((w))$. This implies that the $(n)$-product of fields is again a field for all $n\in\ZZ$.
\end{enumerate}  
\end{oss}

\begin{lemma}[Dong's Lemma]\label{lemma Dong}
    Let $a(z),b(z),c(z)$ be pairwise mutually local fields. Then $a(z)_{(n)}b(z)$ and $c(z)$ are mutually local fields for all $n\in\ZZ$.
\end{lemma}
\begin{proof}
    See \cite[Lemma 3.2]{kacVertexAlgebrasBeginners1998}.
\end{proof}

\section{Calculus of finite differences}\label{sec:finitedifferences}

Let $A$ be a commutative ring containing $\ZZ$, and $\h\in A$ a nonzero element (e.g. $A=\CC$ and $\h=1$, or $A=\CC[\h]$). All the functions in this section are functions $A\rightarrow A$, and the operators are defined on this function space.

\begin{defi}\label{def:finitedifference}
    Let $S_\h$ be the shift operator, defined by
    $$S_\h[f](x):=f(x+\h). $$
    Then the finite difference operator $\Delta_\h$ is defined as $(S_\h-I)/\h$, where $I$ is the identity operator, i.e.
    $$\Delta_\h[f](x):=\frac{f(x+\h)-f(x)}{\h}. $$
\end{defi}
It is clear by definition that $\Delta_\h$ is a linear operator, and that $\Delta_\h\rightarrow\del_x$ as $\h\rightarrow0$. Denoting by $x$ the operator multiplication by $x$, the following commutation relation holds:
\begin{equation}\label{eq commutator difference}
    [\Delta_\h, x\, S_\h^{-1}]=I.
\end{equation}
Under the limit $\h\rightarrow0$, this becomes the usual commutation relation $[\del_x, x]=I$. Because of relation \eqref{eq commutator difference}, a large number of formal relations of standard differential calculus involving functions $f(x)$ map systematically to discrete analogs involving $f(xS^{-1}_\h)$.
\begin{defi}\label{def pochhammer symbols}
    The falling factorial (also known in the literature as Pochhammer symbol) is defined, for $n\in\NN$, as
    \begin{equation}\label{def falling factorial}
        (x)_n:=x(x-1)\dots (x-(n-1))=\prod_{i=0}^{n-1}{(x-i)}. 
    \end{equation}
    If $x\in\NN$, then the falling factorial is related to the binomial coefficient by
    \begin{equation}\label{def binomial coefficients}
      \frac{(x)_n}{n!} =\binom{x}{n}.
    \end{equation}
    We can use \eqref{def binomial coefficients} as a definition to extend the binomial coefficient to all $x\in\CC$, $n\in\NN$. \\
    The $k$-Pochhammer symbol or $k$-falling factorial is a generalization of \eqref{def falling factorial}, depending on a formal parameter $k$. To be consistent with our notation, we will use $\h$ in place of $k$ and define:
    \begin{equation}\label{def k pochhammer symbol}
        (x)_{n,\h}:=x(x-\h)\dots (x-(n-1)\h)=\prod_{i=0}^{n-1}{(x-i\h)}=\h^n(x/\h)_n. 
    \end{equation}
    By a slight abuse of notation, we will use the terminology falling factorial to refer to both \eqref{def falling factorial} and \eqref{def k pochhammer symbol}.
\end{defi}
Since $(x)_{n,\h}=(xS_\h^{-1})^n(1)$, we can consider the falling factorial a discrete analog of the monomial $x^n$. In particular, from \eqref{eq commutator difference} we get 
\begin{equation}\label{eq difference pochhammer symbol}
    \Delta_\h [(x)_{n,\h}]=\Delta_\h[(xS_\h^{-1})^n(1)]=n(xS_\h^{-1})^{n-1}(1)=n(x)_{n-1,\h}.
\end{equation}
Notice also that
\begin{equation}
    \frac{(x)_{n,\h}}{n!}=\h^n\frac{(x/\h)_n}{n!},
\end{equation}
The following are some useful formulae when dealing with falling factorials:
\begin{equation}
    (x)_{n,\h}=x(x-\h)_{n-1,\h}=(x-\h(n-1))(x)_{n-1,\h},
\end{equation}
\begin{equation}\label{eq falling factorial sum exponents}
    (x)_{n+m,\h}=(x)_{n,\h}(x-n\h)_{m,\h},
\end{equation}
\begin{equation}\label{eq (-1)^n Pochhammer symbol}
    (-1)^n(x)_{n,\h}=(-x+(n-1)\h)_{n,\h}.
\end{equation}
\\
We consider next the sum operator, which is a discrete analogue of the integral.
\begin{defi}
    The indefinite sum operator of a function $g$ is defined as 
    $$\sum g(x)\delta x := \{f \, | \, \Delta_\h[f](x)=g(x)\}. $$
    Take $f\in\sum g(x)\delta x$ and $a,b\in A$ such that $a-b\in\h\ZZ$. The definite sum of $g$ from $a$ to $b$ is defined as
    \begin{equation}\label{def definite sum}
        \sum\nolimits_{a}^{b} {g(x) \delta x}:= f(b)-f(a). 
    \end{equation}
\end{defi}
The definition of the definite sum is well posed due to the following theorem
\begin{theorem}[{\cite[(2.46)]{grahamConcreteMathematicsFoundation1994}}]\label{thm:fundamentalfinitedifference}
If $f$ and $g$ are two functions such that $\Delta_\h[f]=\Delta_\h[g]$, then $f=g+C$, where $C$ is a $\h$-periodic function. 
\end{theorem}

\begin{oss}
    If $g$ is a polynomial function, then the definite sum \eqref{def definite sum} is well defined for arbitrary $a,b\in A$, because the only periodic polynomial functions are the constant functions.
\end{oss}

Note that, by definition, definite and indefinite sums are linear operators, and the definite sum satisfies the following relations:
\begin{equation}\label{eq addition extremes definite sum}
    \sum\nolimits_a^c=\sum\nolimits_a^b+\sum\nolimits_b^c,
\end{equation}
\begin{equation}\label{eq inversion extremes definite sum}
    \sum\nolimits_a^b=-\sum\nolimits_b^a.
\end{equation}

\begin{theorem}[{\cite[(2.48)]{grahamConcreteMathematicsFoundation1994}}]\label{teo definite sum equal sum}
    If the extremes $a,b$ of a definite sum are such that $a-b\in\h\ZZ$, then the definite sum is equal to
    $$\sum\nolimits_{a}^{b}{g(x)\delta x}=\h\sum_{k=a}^{b-\h}g(k), $$
    where by the sum on the right we mean $g(a)+g(a+\h)+\dots+g(b-2\h)+g(b-\h)$.
\end{theorem}

There are many more similarities between discrete and infinitesimal calculus, for example there is an analog of Leibniz rule and of integration by parts (see \cite{grahamConcreteMathematicsFoundation1994}). However, we have no general formula for the discrete change of variables, because there is no easy formula for the discrete chain rule. In our case, we will only need a basic change of variable, of the form $x\mapsto c\pm x$.

\begin{prop}\label{prop change of variable}
    The following formulae for the change of variables hold, for all $\alpha\in A$
    \begin{equation}\label{eq change of variable 1}
        \sum\nolimits_a^b(x+\alpha)_{n,\h}\,\delta x=\sum\nolimits^{b+\alpha}_{a+\alpha}(x)_{n,\h}\,\delta x; 
    \end{equation}
    \begin{equation}\label{eq change of variable 2}
        \sum\nolimits_a^b(-x+\alpha)_{n,\h}\,\delta x=-\sum\nolimits^{-b+\alpha+\h}_{-a+\alpha+\h}(x)_{n,\h}\,\delta x. 
    \end{equation}
\end{prop}
\begin{proof}
    Since $\Delta_\h$ commutes with multiplication by constants, we can use the same proof of \eqref{eq difference pochhammer symbol} to get 
    \begin{equation}
        \Delta_\h[(x+\alpha)_{n,\h}]=n(x+\alpha)_{n-1,\h},
    \end{equation}
    from which \eqref{eq change of variable 1} follows. For \eqref{eq change of variable 2}, notice that
    \begin{align*}
        \Delta_\h[(-x+\alpha)_{n,\h}]&=\frac{(-x+\alpha-\h)_{n,\h}-(-x+\alpha)_{n,\h}}{\h} \\
        &=\frac{(-x+\alpha-n\h)(-x+\alpha-\h)_{n-1,\h}-(-x+\alpha)(-x+\alpha-\h)_{n-1,\h}}{\h} \\
        &=-n(-x+\alpha-\h)_{n-1,\h}.
    \end{align*}
    The left-hand side of \eqref{eq change of variable 2} then becomes 
    $$-\frac{(-b+\alpha+\h)_{n+1,\h}-(-a+\alpha+\h)_{n+1,\h}}{n+1}, $$
    which is equal to the right-hand side of \eqref{eq change of variable 2} by \eqref{eq difference pochhammer symbol}.    
\end{proof}

It makes sense to consider multiple definite sums. As happens for multiple integrals, in some cases it is convenient to exchange the order of summation. We do not have a general rule for that, but we can do it in some special cases.

\begin{prop}\label{prop change order summation}
    Suppose $f(x,y)\in\CC[\h,x,y]$. Take a finite number of formal parameters $t,t_1,\dotso,t_n$. Take $a,a',b,b'$ in the $\ZZ$-span of $t_1,\dotso,t_n$, and $c,c',d,d'$ in the $\ZZ$-span of $t,t_1,\dotso,t_n$, such that, if we specialize $(t_1,\dotso,t_n)$ to any element in $(\h\ZZ)^n$, we have
    $$(a\leq x\leq b-\h)\wedge(c(x)\leq y\leq d(x)-\h)\iff (a'\leq y\leq b'-\h)\wedge(c'(y)\leq x\leq d'(y)-\h). $$
    Then
    \begin{equation}\label{eq change order summation}
        \sum\nolimits^{b}_{a}\delta x\sum\nolimits_{c(x)}^{d(x)}\delta y f(x,y)=\sum\nolimits^{b'}_{a'}\delta y\sum\nolimits_{c'(y)}^{d'(y)}\delta x f(x,y),
    \end{equation}
    where $c(x):=c(x,t_1,\dots,t_n)$ and so on.
\end{prop}

\begin{proof}
    Both sides of the equation \eqref{eq change order summation} are polynomials in $\CC[\h,t_1,\dotso,t_n]$. Thus, if they coincide for every value of $(t_1,\dotso,t_n)\in(\h\ZZ)^n$, they are equal. 
    
    Let $(t_1,\dotso,t_n)\in(\h\ZZ)^n$, then we can use Theorem \ref{teo definite sum equal sum} to rewrite \eqref{eq change order summation} as 
    $$\sum^{b-\h}_{k=a}\ \sum_{l=c(k)}^{d(k)-\h} f(k,l)=\sum^{b'-\h}_{l=a'}\ \sum_{k=c'(l)}^{d'(l)-\h}f(k,l), $$
    which is just the reordering of ordinary sums.
\end{proof}

\end{appendices}

\printbibliography
\Addresses
\end{document}